\documentclass{article}
\usepackage{amsfonts}

\usepackage[utf8]{inputenc}
\usepackage{color}

\usepackage{mathrsfs}
\usepackage[OT2, T1]{fontenc}
\usepackage{url}
\usepackage{amsmath}
\usepackage{mathabx}
\usepackage{array}
\usepackage{graphicx}
\usepackage{amssymb}
\usepackage{amstext}
\usepackage{amsthm}
\usepackage{hyperref}
\usepackage{colonequals}
\usepackage{enumitem}
\usepackage[alphabetic,lite]{amsrefs}
\usepackage{cleveref}
\usepackage[all,cmtip]{xy}
\usepackage{fullpage}
\numberwithin{equation}{subsection}

\newtheorem{theorem}{Theorem}[section]
\newtheorem{lemma}[theorem]{Lemma}
\newtheorem{proposition}[theorem]{Proposition}
\newtheorem{corollary}[theorem]{Corollary}

\theoremstyle{definition}
\newtheorem{definition}[theorem]{Definition}

\theoremstyle{remark}

\newtheorem*{claim}{Claim}
\newtheorem{rmk}[theorem]{Remark}

\DeclareMathOperator{\image}{\mathrm{im}}
\DeclareMathOperator{\res}{\mathrm{res}}
\DeclareMathOperator{\G}{\mathbb{G}}
\DeclareMathOperator{\Cb}{\mathbb{C}}
\DeclareMathOperator{\Fb}{\mathbb{F}}
\DeclareMathOperator{\Zb}{\mathbb{Z}}
\DeclareMathOperator{\bF}{\mathbb{F}}

\DeclareMathOperator{\GL}{GL}

\DeclareMathOperator{\ur}{ur}

\DeclareMathOperator{\sM}{\mathcal{M}}

\DeclareMathOperator{\Spec}{Spec}
\DeclareMathOperator{\End}{\mathsf{End}}

\DeclareMathOperator{\Oo}{\mathcal{O}}

\DeclareMathOperator{\gen}{gen}

\newcommand{\Dht}{D_{\textrm{HT}}}

\newcommand{\HT}{\textrm{HT}}
\newcommand{\ra}{\rightarrow}

\newcommand{\A}{\mathbb{A}}

\let\C\relax
\newcommand{\C}{\mathbb{C}}

\newcommand{\F}{\mathbb{F}}

\newcommand{\Q}{\mathbb{Q}}
\newcommand{\ol}{\overline}
\newcommand{\Qbar}{\overline{\Q}}
\newcommand{\R}{\mathbb{R}}
\newcommand{\Ss}{\mathbb{S}}

\newcommand{\V}{\mathbb{V}}
\newcommand{\Z}{\mathbb{Z}}
\newcommand{\bL}{\mathbb{L}}
\newcommand{\W}{\mathbb{W}}

\newcommand{\bbA}{\mathbf{A}}

\newcommand{\cA}{\mathcal{A}}
\newcommand{\cB}{\mathcal{B}}

\newcommand{\cD}{\mathcal{D}}

\newcommand{\cL}{\mathcal{L}}

\newcommand{\cO}{\mathcal{O}}
\newcommand{\cOB}{\mathcal{O}B}

\newcommand{\cR}{\mathcal{R}}
\newcommand{\cS}{\mathcal{S}}

\newcommand{\cV}{\mathcal{V}}

\newcommand{\cX}{\mathcal{X}}
\newcommand{\cY}{\mathcal{Y}}

\newcommand{\fp}{\mathfrak{p}}

\newcommand{\ho}{\widehat{\cO}}

\newcommand{\ShimK}{S_K(G,X)}

\newcommand{\integralShimK}{\mathscr{S}_K(G,X)}

\DeclareMathOperator{\MF}{MF}
\newcommand{\FL}{\MF^{\nabla}_{[0,p-2],\textrm{free}}}

\newcommand{\rng}{{\mathrm{rng}}}
\newcommand{\dR}{_{\mathrm{dR}}}
\newcommand{\et}{{\mathrm{et}}}
\newcommand{\cris}{_{\mathrm{cris}}}

\newcommand{\Fil}{{\mathrm{Fil}}}

\newcommand{\an}{{\mathrm{an}}}
\newcommand{\ds}{{\displaystyle}}

\DeclareMathOperator{\dr}{dR}

\newcommand{\OBdr}{\mathcal{O}B_{\dr}}
\newcommand{\OBdrlog}{\mathcal{O}B_{\dr,\log}}
\newcommand{\Bdr}{B_{\dr}}

\DeclareMathOperator{\GSp}{GSp}

\DeclareMathOperator{\gr}{gr}

\DeclareMathOperator{\tr}{tr}
\DeclareMathOperator{\Gal}{Gal}
\DeclareMathOperator{\Hom}{Hom}

\DeclareMathOperator{\Lie}{Lie}
\DeclareMathOperator{\Res}{Res}
\DeclareMathOperator{\Spf}{Spf}
\DeclareMathOperator{\ad}{ad}

\DeclareMathOperator{\der}{der}
\DeclareMathOperator{\disc}{disc}

\DeclareMathOperator{\Nm}{Nm}

\DeclareMathOperator{\Id}{Id}

\DeclareMathOperator{\Disc}{Disc}

\DeclareMathOperator{\Gr}{Gr}
\DeclareMathOperator{\rig}{rig}
\DeclareMathOperator{\std}{Std}

\definecolor{ao(english)}{rgb}{0.0, 0.5, 0.0}

\title{Canonical Heights on Shimura Varieties and the Andr\'e-Oort Conjecture}
\author{Jonathan Pila, Ananth N. Shankar, Jacob Tsimerman\smallskip\\\emph{\MakeLowercase{with an appendix by} H\'el\`ene Esnault and Michael Groechenig} }

\begin{document}

\maketitle

\tableofcontents\newpage


\section{Introduction}

\subsection{History and main results}

The main purpose of this work is to prove the Andr\'e-Oort conjecture in full generality. Recall the statement of the conjecture:

\begin{theorem}\label{mainao}
Let $S$ be a Shimura variety. Let $V\subset S$ be a subvariety. Then there are only finitely many maximal special subvarieties contained in $V$. 
\end{theorem}

The first unconditional result was obtained by Andr\'e \cite{Andre} for a product of two modular curves. In the past two decades there has been much work on the conjecture. First, spurred by an idea of Edixhoven \cite{Edixhoven}, the conjecture was proven conditionally on GRH in a series of works by Klingler-Ullmo-Yafaev \cite{KY,UYA}. 

Further unconditional results were obtained, starting with \cite{Pila}, using a different strategy. This strategy was originally proposed by Zannier
and had been implemented to reprove the Manin-Mumford conjecture in \cite{PZ}. The approach has three main ingredients:

\begin{itemize}
    \item Estimates \cite{PW} for counting rational points on a transcendental set
    \item Functional Transcendence theorems, specifically for the uniformization map of Shimura varieties
    \item Lower bounds for Galois orbits of special points.
\end{itemize}

The required functional transcendence results were generalized to arbitrary Shimura varieties in a series of works including \cite{PT-axlag, UY}, and inspired stronger transcendence results in these and other contexts, including recently a proof of the so called Ax-Schanuel theorem in the context of arbitrary variations of mixed Hodge structures \cite{MPT, BT-hodgeaxs, chiu, GK}. Such generalizations are important for studying deeper unlikely intersection questions such as the Zilber-Pink conjecture.

The remaining missing ingredient was the lower bound for Galois orbits of special points. Such bounds were obtained for $\cA_g$ in \cite{T} by relying on the average Colmez conjecture\cite{col} (proved independently by Andreatta-Goren-Howard-Madapusi--Pera\cite{AGHM} and Yuan-Zhang \cite{YZ}) to obtain bounds for the \emph{height} of Galois orbits and the Masser-W\"ustholz isogeny estimates. The isogeny estimates are not available in arbitrary Shimura varieties, but this ingredient was recently removed in a paper of Binyamini-Schmidt-Yafaev \cite{BSY} based on a recent breakthrough of Binyamini \cite{Binyamini} obtaining strong point-counting results in terms of both the height and degree of the points being counted. The idea for this strategy to obtain Galois orbit bounds first appeared in work of Schmidt
\cite{Schmidt} in the context of tori and elliptic curves.

In essence, this means that the conjecture is reduced to finding suitable upper bounds for heights of special points.
Our main contribution is to establish this result:

\begin{theorem}[Theorem \ref{thm:mainhtbound} + \S11]
Fix a Shimura variety $S_K(G,X)$ with $G$ of adjoint type. Let $(T,r)\subset (G,X)$ be a (varying) 0-dimensional Shimura datum such that $K\cap T(\bbA_f)$ is of index $M$ in the maximal compact $K_T$, and let $E_T$ be the splitting field of $T$. Then the height of $(T,r)$ with respect to an ample line bundle is $(M\disc E_T)^{o(1)}$. 
\end{theorem}

Gao \cite{Gao-reduction} has shown how to deduce the Andr\'e-Oort conjecture for a mixed Shimura variety from Galois bounds for the associated pure Shimura variety, and so the conjecture (which is as stated in Theorem \ref{mainao} for $S$ a mixed
Shimura variety) follows for all mixed Shimura varieties also.

\subsection{Method of proof}

The idea is to reduce the general case to the height bounds for CM points in $\cA_g$. We thus need a way to compare heights for CM points which embed in different Shimura varieties. To facilitate this comparison, we require a `canonical' height on arbitrary Shimura varieties, similar to the Faltings height on $\cA_g$. At first glance this might feel minor, because all Weil heights on a line bundle are the same up to bounded functions, so why is it so important to get a canonical height function? The reason is that we have a different comparison for each CM point, so knowing the result up to a bounded function for each CM point separately tells us nothing!  As such, we need to have better pointwise control. We explain how to do this in the next subsection. Once this is done, then for every $0$-dimensional Shimura variety and automorphic line bundle - i.e. a character of the split torus -  we obtain a canonical  height \footnote{In fact the details are a bit more complicated and require us to make some arbitrary and unaesthetic choices between what we dub the \emph{intrinsic} and \emph{crystalline} norms, but morally this is what is happening.}.

We first explain in \ref{S-SV-partialCMtypes} how to associate a 0-dimensional Shimura datum $(E^\times/F^\times,r_\Phi)$ to a partial CM type $\Phi$  associated to a CM field $E/F$, and a canonical character $\chi_\Phi$ on the associated torus. We are thus reduced to bounding the corresponding intrinsic heights of $\chi_\Phi$ on $(T_\Phi,r_\Phi)$, and the case where $\Phi$ is a full CM-type is covered by the case of $\cA_g$.

The idea is then to use Deligne's construction for augmenting partial CM types into full CM types. Namely, suppose that $E_1,E_2$ are CM fields over the same real totally real field $F$. Say $\Phi_1,\Phi_2$ are partial CM types for $E_1,E_2$ respectively, such that their restrictions to $F$ are complementary. Then one may form a complete CM type $\Phi$ on $E_{tot}:=E_1E_2$ by taking the union of the pullbacks of the $\Phi_i$. Moreover, 
$(E_{tot}^\times/F_{tot}^\times,r_\Phi)$ admits maps (up to isogeny) to $(E_i^\times/F^\times,r_{\Phi_i})$ via the norm map on tori, and $\chi_{\Phi}$ is the product of the pullbacks of the $\chi_{\Phi_i}$. In this way, we are able to deduce from the case of full CM type, bounds for the sum of the intrinsic heights of $\Phi_1$ and $\Phi_2$.

Armed with this technique, the key observation is that we may combine these partial CM types in many different ways,  allowing us to extract individual bounds by taking linear combinations. This is a simple combinatorial argument, which we carry out in Theorem \ref{delignesidea}.  

\subsection{Constructing canonical heights on Shimura varieties}

Given a $\Z$-local system $\V$ arising from an algebraic representation $V$ of $G_\Q$, we may form an associated flat vector bundle $_{\dR}V$ via the Riemann-Hilbert correspondence, which in this setting is equipped with a filtration. Our line bundles will arise as determinants of sub-bundles of $_{\dR}V$ , and so we seek to give $\Gr^\bullet_{\dR}V$ the structure of a normed vector bundle. At the Archimedean places, we simply use the Hodge norm\footnote{This depends on the choice of polarization}.
Our construction of canonical heights is similar in spirit to previous constructions of heights for motives, in particular, works of Kato \cite{Kato} and Koshikawa \cite{Koshikawa}. 

At the finite places, work of Scholze\cite{S} and Liu-Zhu\cite{LZ} has shown how to interpret $_{\dR}V$ in terms of a $p$-adic Riemann-Hilbert correspondence, which specializes to the functor $D_{\dR}$ from $p$-adic Hodge theory pointwise. This allows us to equip its grading with a canonical norm, and it is not hard to show that this norm behaves well at each finite place. 

A difficulty arises in showing that the above local norms piece together to give a well-behaved global norm. We suspect that this is the case for all Shimura varieties\footnote{This would follow from the (conjectured) existence of suitable motives over exceptional Shimura varieties}  (and perhaps much more generally), but have been unable to establish this.

Here we use in an essential way the work of Esnault-Groechenig, which they generalize in the appendix to accommodate non-proper Shimura varieties. Specifically, they establish that if the local system is rigid, then for all but finitely many places it is in fact crystalline in the sense of Faltings. It is this extra data that allows us to prove that the heights piece together well.


\subsection{Solid height functions}

Most fundamentally for us is a height on 0-dimensional Shimura varieties which behaves well with respect to embeddings. For a $0$-dimensional Shimura variety $(T,r)$ we can assign automorphic line bundles corresponding to \emph{solid} characters $\chi$ of $T$ (See definition \ref{def:solidcharacter}). To such a triple $(T,r,\chi)$ we associate a real-valued height function $h$ (note that this is level invariant), satisfying the following four properties:

    \begin{enumerate}
        \item If $\chi_1,\chi_2$ are both solid characters then \footnote{We denote by $E_T$ the splitting field of $T$}
        $$h(T,r,\chi_1\chi_2)=h(T,r,\chi_1)+h(T,r,\chi_2)+O_{\rng_{\chi_1}+\rng_{\chi_2}+\dim T}(\log\Disc E_T)$$

        \item Given a solid triple $(T,r,\chi)$ and a positive integer $m$ we have 
        $$h(T,r^m,\chi)=mh(T,r,\chi) + O_{m+\rng_\chi+\dim T}(\log\Disc E_T)$$

        \item If $(T_2,r_2,\chi)$ is solid, and $f:(T_1,r_1)\ra (T_2,r_2)$ is a morphisms of Shimura data with $T_1,T_2$ tori, then 
        $$h(T_1,r_1,\chi\circ f)=h(T_2,r_2,\chi) + O_{\rng_\chi+\dim T_1+\dim T_2}(\log \Disc E_{T_1}+\log \Disc E_{T_2}).$$

        \item Let $S=S_K(G,X)$ be a Shimura variety, and $V$ an irreducible representation of $G$ with a sublattice $\V$ fixed by $K$. Assume that the highest weight piece $L:=\Fil^a _{\dR}V_\C$ is 1-dimensional, and let $h_L$ be a Weil height on a Toroidal compactification $\bar S$ of $S$ corresponding to the Deligne extension of $L$. Let $h_A$ be a Weil height corresponding to any ample bundle on a Toroidal compactification of $S$. Finally, let $(T,r)\subset (G,X)$ be a 0-dimensional Shimura subdatum, and $\chi:=\Fil_r^a V_\C$ the corresponding solid character of $T$. Then for all points $x\in S_K(G,X)$ in the image of $S_{K\cap T(\bbA_f)}(T,r)$, we have
        $$|h(T,r,\chi) - h_L(x)| = O_S\Big(\log\Disc E_T + \log\big([K_T:K\cap T(\bbA_f)]\big) + \log^+ h_A(x)\Big)$$
    \end{enumerate}

The most important property is the last one, establishing a uniform comparison between the solid height function of a 0-dimensional Shimura variety, and the Weil height of a corresponding CM point when embedded in a larger Shimura variety.

In \ref{thm:solidimpliesAO} we show that the existence of a solid height function implies the validity of the Andr\'e-Oort conjecture, and in \S11 we construct a solid height function using the canonical heights from before.

\subsection{Organization of the paper}

The paper is roughly broken up into 3 parts:

\begin{enumerate}
    \item \S2-5: Using $p$--adic Hodge theory to construct canonical norms on vector bundles corresponding to Crystalline and De-Rham representations
    \item \S6-9: Endowing Shimura varieties with canonical height functions
    \item \S10-11: Proving the Andr\'e-Oort conjecture
    
\end{enumerate}

In \S\ref{S-hodge} we recall the relevant integral $p$-adic Hodge theory results from Tsuji\cite{Tsuji} and Liu-Zhu\cite{LZ} and show a compatibility between them. In \S\ref{S-metrized} we define notions of well-behaved norms on vector bundles that allows us to discuss global heights. In \S\ref{S-canheights} we work in the general context of a variety and a de Rham local system on it and enow the corresponding vector bundle with canonical norms (crystalline and de Rham). In  \S\ref{S-ramifiedexamples} we present concrete examples of what our constructions give in the case of CM elliptic curves with CM by ramified and non-maximal orders.

In \S\ref{S-SV} we recall some basics about Shimura varieties, introduce the concept of partial CM-types and show that the associated 0-dimensional Shimura varieties are sufficient to understand every
0-dimensional Shimura variety arising in a Shimura variety of adjoint type. In \S\ref{S-crystalline} we combine results of Esnault-Groechenig (\cite{EsnaultGroechenig} and generalized to our setting in the appendix) and Margulis \cite{margulis} to show that automorphic local systems on Shimura varieties are (log-)crystalline, so that we may apply the theory of relative Fontaine-Laffaille modules to them. In \S\ref{S-GR} by showing that in a given Shimura variety, all CM points are integral with respect to an appropriate integral model in the spirit of the Neron-Ogg-Shafarevich criterion. This is important for us as we need our representations to be crystalline at almost every prime for the analysis when we compare heights.
Finally, In \S\ref{S-normShimura} we combine the previous results to obtain global `canonical' heights on Shimura varieties.

In section \S\ref{S-solidheightfunctions} we define solid height functions and prove that their existence implies the truth of the Andr\'e-Oort conjecture. This section is self-contained (except for some background on Shimura varieties). Finally, in section \S\ref{S-solidheightconstruction} we construct a solid height function using the results of \S\ref{S-normShimura}.

\subsection{Acknowledgements}

The authors would like to thank Yves Andr\'e, Fabrizo Andreatta, Ben Bakker, Bhargav Bhatt, George Boxer, Antoine Chambert-Loir, H\'el\`ene Esnault, Michael Groechenig, Haoyang Guo, Mark Kisin, Ruochan Liu,  Davesh Maulik, Yong Suk Moon, Matthew Morrow, Will Sawin, Peter Scholze, Andrew Snowden, Takeshi Tsuji, Zijian Yao, Shouwu Zhang, and Roy Zhao for useful conversations. We would also like to thank Teruhisa Koshikawa for useful comments and for pointing out several inaccuracies regarding the rigid geometry and $p$-adic Hodge theory parts of the paper. We thank the referees for pointing us to the work of Brinon\cite{brinon} and many useful suggestions. 

A.S. was partially supported by NSF grant DMS-2100436.

\section{$p$-adic Hodge Theory}\label{S-hodge}

\subsection{Notation}\label{sec:padicHTnotation}

We use the following definitions and notations from \cite{brinon,S,Tsuji,LZ}\footnote{The notation in \cite{S} isn't always consistent with the notation in \cite{brinon}. In cases when the notation conflicts, we defer to \cite{S}.}:

\begin{itemize}

    \item $V$ is a complete DVR of mixed characteristic $p$ with perfect residue field $k$ and fraction field $K$. We assume $p$ is a uniformizer of $V$, so that $V\cong W(k)$, the Witt-vectors of $k$.
    
    \item $\ol V$ is the integral closure of $V$ in $\ol{K}$.
     
    \item $R$ is a smooth $V$-algebra. Let $\cR$ denote its $p$-adic completion. 
    
    \item $\ol{R}$ is the integral closure of $R$ in the maximal \'etale extension of $R[1/p]$, and define $\ol{\cR}$ analogously. Let $\hat{\ol{\cR}}$ denote the $p$-adic completion of $\ol{\cR}$.
    
    \item $f:V[s_1,s_1^{-1},\dots,s_n,s_n^{-1}]\ra R$ is an \'etale map.

    \item Let $\ol{\cR}^{\flat} $ be $\varprojlim_{x\mapsto x^p} (\ol{\cR}/p\ol{\cR})$. 
    
    \item Define $A_{\inf}(\ol{\cR})$ to be $W(\ol{\cR}^{\flat})$, the ring of Witt-vectors of $\ol{\cR}^{\flat}$. This ring has a Frobenius $\varphi$ by functoriality. For any element $x\in \ol{\cR}^{\flat}$, let $[x] \in A_{\inf}(\ol{\cR})$ denote its Teichmuller lift. 
    
    
    \item There is the following ring homomorphism $\theta: A_{\inf}(\ol{\cR})\rightarrow \hat{\ol{\cR}}$, characterized by $\theta([a]) = \lim_{n\rightarrow \infty} \widetilde{a_n}^{p^n}$, where $a\in \ol{\cR}^{\flat} = (a_n)$ and $\widetilde{a_n}\in \ol{\cR}$ is some lift of $a_n$. 
    
    \item Fix a compatible system $\beta_n \in \overline{\cR}$ of $p^{n}$th roots of $p$, with $\beta_0 = p$. Define $p^\flat \in \ol{\cR}^{\flat}$ to be the element $(\beta_n \mod p)_n$. Let $[p^\flat]$ denote the Teichmuller lift of $p^\flat.$
    
    \item Define $\xi\in A_{\inf}(\ol{\cR}) = p - [p^\flat]$. It generates the kernel of $\theta$.
    
    \item Define $\epsilon$ to be an inverse limit of primitive $p^n$th roots of unity, and $\pi=[\epsilon]-1$.

    \item Define $\Fil^r(A_{\inf}(\ol{\cR}))$ to be the ideal $\ker(\theta)^r$ if $r\geq 1$, and to be $A_{\inf}(\ol{\cR})$ if $r\leq 0$. 
    
    \item Define $A_{\cris,m}(\ol{\cR})$ to be divided power envelope of $A_{\inf}(\ol{\cR})/p^m$ with respect to the ideal $\ker(\theta) \mod p^m$. Define $A_{\cris}(\ol{\cR})$ to be the inverse limit $\varprojlim_{m} A_{\cris,m}(\ol{\cR})$ endowed with the inverse limit topology with respect to the discrete topology on $A_{\cris,m}(\ol{\cR})$. This is the same as the following construction: consider the divided power envelope $A_{0,\cris}(\ol{\cR})$ of $A_{\inf}(\ol{\cR})$ with respect to the ideal $\ker(\theta)$. Then $A_{\cris}(\ol{\cR})$ is the $p$-adic completion of $A_{0,\cris}(\ol{\cR})$.

    
    
    \item Define $\cA_{\cris,m}(\ol{\cR})$ to be the divided power envelope of $A_{\inf}(\ol{\cR}) \otimes_V \cR/p^m$ with respect to the kernel of the surjective homomorphism $\theta \mod p^m \otimes \Id \mod p^m: A_{\inf}(\ol{\cR}) \otimes \cR \mod p^m \rightarrow \ol{\cR} \mod p^m$. Define $\cA_{\cris}(\ol{\cR})$ to be the inverse limit of $\cA_{\cris,m}(\ol{\cR})$. This ring is isomorphic to $A_{\cris}(\ol{\cR})\langle v_1,\hdots v_n \rangle^{PD}$, which is the $p$-adic completion of the PD-algebra $A_{\cris}(\ol{\cR})[ v_1,\hdots v_n ]^{PD}$. Here, the elements $v_i$ map to $[s_i^{\flat}] \otimes s_i^{-1} - 1 = (1\otimes s_i^{-1})([s_i^{\flat}]\otimes 1 - 1\otimes s_i) $.
    
    \item Following Brinon,  $B_{\cris}(\ol{\cR}) := A_{\cris}(\ol{\cR})[\frac{1}{\pi}]$, $\cB_{\cris}(\ol{\cR}):= \cA_{\cris}(\ol{\cR})[\frac{1}{\pi}]$. The element $p$ is invertible in $B_{\cris}(\ol{\cR})$ and $\cB_{\cris}(\ol{\cR})$. We note that the notation for $B_{\cris}(\ol{\cR})$ (\emph{resp.} $\cB_{\cris}(\ol{\cR})$) in \cite{brinon} is $B_{\cris}^{\nabla}$ (\emph{resp.} $B_{\cris}$). 
    
    \item $B_{\inf}(\ol{\cR}):=A_{\inf}(\ol{\cR})[\frac1p]$. We use $\theta$ to refer also to the natural extension $B_{\inf}(\ol{\cR})\ra \hat{\ol{\cR}}[\frac1p]$.
    
    \item $B_{\dR}^+(\ol{\cR}):=\varprojlim B_{\inf}(\ol{\cR})/(\ker\theta)^n$ with its natural filtration $\Fil^iB_{\dR}^+(\ol{\cR})$ generated by $(\ker\theta)^i$.
    
    \item $B_{\dR}(\ol{\cR}):=B_{\dR}^+(\ol{\cR})[\frac1\xi]$. We note that this is the ring that Brinon denotes by $B_{\dR}^{\nabla}$
    
    \item $\cOB_{\inf}(\ol{\cR}):=\ol{\cR}\otimes_{V}B_{\inf}(\ol{\cR})$. Note this still admits a map $\theta$ to $\hat{\ol{\cR}}\big[\frac1p\big]$.
    
    \item $\cOB_{\dR}^+(\ol{\cR})$ is a suitable $p$-adic completion of $\varprojlim \cOB_{\inf}(\ol{\cR})/(\ker\theta)^n$ as in \cite{S-C} with its natural filtration $\Fil^i\cOB_{\dR}^+(\ol{\cR})$ generated by $(\ker\theta)^i$. It is a power-series ring over $B_{\dR}^+(\ol{\cR})$, with variables $X_i$. We note that this is the ring Brinon denotes by $B_{\dR}$.

    \item $\cOB_{\dR}(\ol{\cR}):=\cOB^+_{\dR}(\ol{\cR})[\frac1\xi]$.  
    
    \item Note that $\cA_{\cris}(\ol{\cR}), \cB_{\cris}(\ol{\cR}), \cOB_{\inf}(\ol{\cR}),\cOB_{\dr}^+(\ol{\cR}),\cOB_{\dR}(\ol{\cR})$ all admit integrable connections \mbox{$*\ra*\otimes_\cR\Omega^1_{\cR/V}$} inherited from $\cR$, as well as an action of $G =\Gal(\ol{\cR}/\cR) $. 
    
    \item There are natural embeddings $B_{\cris}(\ol{\cR})\rightarrow B_{\dr}(\ol{\cR})$ and  $\cB_{\cris}(\ol{\cR})\rightarrow \cOB_{\dR}(\ol{\cR})$, compatible with filtrations, $G$-actions and the connection. Under these embeddings, the local coordinates $v_i$ map to $(1\otimes s_i^{-1})X_i$ (\cite[Section 6.2.1, Proposition 6.2.1, Corollaire 6.2.3]{brinon}).  
    
   \item Define $t=\log([\epsilon]):=\sum_{i\geq 1} (-1)^{i+1}\frac{\pi^i}{i}$ as an element of $A_{\cris}(\ol{\cR})$. Note that $t-\pi \in \Fil^2 A_{\cris}(\ol{\cR})$, $\frac{t}{\pi}$ is a unit in $A_{\cris}(\cO_{\C_p})$ and that the Galois action on $A_{\cris}(\cO_{\C_p})$ scales $t$ via the cyclotomic character.

\end{itemize}

\subsection{The $p$-adic Riemann-Hilbert correspondence of Liu-Zhu}

We summarize here (a consequence of) one of the main results from \cite{LZ}, in the context of $X=\Spf \cR$. To set the stage, we let $\cL$ be a $\Z_p$ - local system of finite rank $r$ over $ X^{\mathrm{rig}}$, the generic fiber of $X$. We can identify this with a $G=\Gal(\ol{\cR}/\cR)$ - module $L$, free of rank $r$ over $\Z_p$. The following Theorem\footnote{Liu-Zhu work in far greater generality. Indeed, they work with local systems on rigid-analytic varieties over finite extensions of $\Q_p$. The goal of this section is to compare Liu-Zhu's construction with the relative Fontaine-Laffaile correspondence, and hence we work in a more restricted setup. However, we will use Liu-Zhu's general framework in later sections.} is \cite[Thm 3.9]{LZ}.

\begin{theorem}

Define $D^0_{\dR}(L):=(L\otimes \cOB_{\dR}(\ol{\cR}))^G$. 
Then $D^0_{\dR}(L)$ is a locally-free $\cR[\frac1p]$-module with an integrable connection. It is of rank $\leq r$ wth equality iff $L$ is a de Rham local system in the sense of Brinon and Scholze.

\end{theorem}

\begin{proof}
The definition of $D^0_{\dR}(L)$ is taken from \cite[\S3.2]{LZ} once one makes the translations from sheaves to rings. The statement about ranks follows from the pullback compatiblity \cite[Thm 3.9(ii)]{LZ} together with the corresponding statement for the classical case $R=V$, as well as the statement \cite[Thm 1.3]{LZ} showing that a local system is de Rham iff its stalk is de Rham at a single classical point in every connected component. 
\end{proof}

\subsection{Review of Tsuji's constructions}
Let $R,\cR$ be as above. We fix a Frobenius-lift $\phi$ on $V[s_i]$, by setting $\phi(s_i) = s_i^p$. There is a unique extension of $\phi$ on $\cR$ which we will also denote by the same $\phi$. We recall the category $ \FL(\cR,\phi)$. An object of $\FL(\cR,\phi)$ is given by the quadruple $(M,\Fil,\nabla,\Phi)$ where (see \cite[\S4]{Tsuji} for more details): 
\begin{enumerate}
    \item $M$ is a finitely-generated locally free $\cR$-module .
    \item A topologically nilpotent integrable connection $\nabla: M\rightarrow M\otimes \Omega^1$.
    \item A decreasing filtration $\Fil $ of $M$ by $\cR$-submodules, satisfying: 
    \begin{enumerate}
        \item[(a)] $\Fil^0 M = M$ and $\Fil^{p-1}M = 0$. 
        \item[(b)] $\Gr^{\bullet} M$ is finitely generated free $\cR$-module. 
        \item[(c)] $\nabla(\Fil^rM) \subset \Fil^{r-1}M\otimes \Omega^1$ (Griffiths transversality).
    \end{enumerate}
    \item Associated to $M$ is the filtered module $F^*(M)$ where the underlying module is simply $\phi^*M$, and the underlying filtration is
    $$\Fil^rF^*(M):=\sum_{s\geq 0}\frac{p^s}{s!}\phi^*(\Fil^{r-s}M)$$
    \item An $\cR$-linear homomorphism $\Phi:F^*(M)\ra M$ where 
        \begin{enumerate}
            \item $\Phi$ is compatible with the connections
            \item $\Phi(\Fil^r(F^*M))\subset p^rM$ for $r\in[0,p-2]$
            \item $\displaystyle\sum_{r=0}^{p-2} p^{-r}\Phi(\Fil^r(F^*(M)))=M$
        \end{enumerate}
\end{enumerate}

Tsuji in \cite[Section 5]{Tsuji} associates the following objects to an object $M$ of $\FL(\cR,\varphi)$:
\begin{itemize}
    \item An $A_{\cris}(\overline{\cR})$-module $TA_{\cris}(M):=M\otimes_{\cR} A_{\cris}(\ol\cR)$\cite[(37)]{Tsuji}
    \item An $A_{\inf}(\overline{\cR})$-module $TA_{\inf}(M)$, with a canonical isomorphism $$TA_{\inf}(M)\otimes_{A_{\inf}(\overline{\cR})}A_{\cris}(\overline{\cR})\cong TA_{\cris}(M)$$
    \item A $G$-module $T_{\cris}M$ of rank the same as $M$. This is defined as the dual of $T_{\cris}^*(M)$ where
$$T_{\cris}^*(M):=\Hom_{\cR,\Fil,\phi,\nabla}(M,\cA_{\cris}(\ol{\cR})).$$  Any $G$-module obtained this way is defined to be a crystalline local system.
\end{itemize}

\begin{rmk}\label{logfontainelaffaille}
\begin{enumerate}
    \item Let $\overline{X}/V$ denote a proper smooth scheme, let $D\subset \overline{X}$ denote a relative normal crossing divisor, and let $X = \overline{X}\setminus D$. Faltings defines the notion of a ``logarithmic Fontaine-Laffaille module" on $\overline{X}$ (see \cite[Theorem 2.6', page 43, i)]{FaltingsCrystallineCohomology}), associated to which is a local system on $X_{V[1/p]}$. Outside the appendix, we will refer to such local systems as log-crystalline in order to distinguish this notion from the notion of crystalline local systems on the generic fiber of smooth formal schemes. We note that such local systems are usually deemed crystalline in the literature (and in the appendix, are called crystalline local systems). 
    
    \item In the setting of $X, \bar{X}$ and $D$ above, let $L/X^{\an}_{V[1/p]}$ be a log-crystalline local system. Then the restriction $L|\mathscr{X}^{\rig}$ is crystalline in the usual sense, where $\mathscr{X}$ is the $p$-adic completion of $X$ and $\mathscr{X}^{\rig}$ is the rigid generic fiber of this formal scheme. Similarly, the restriction of a log Fontaine-Laffaile module to $\mathscr{X}$ is a Fontaine-Laffaile module in the usual sense. Further, the log Fontaine-Laffaile correspondence agrees with the usual Fontaine-Laffaile correspondence when restricted to $\mathscr{X}$ and $\mathscr{X}^{\rig}$. We will abuse notation and let $T_{\cris}$ and $S_{\cris}$ denote the forward and backwards direction of the log Fontaine-Laffaile correspondence.
    
\item By \cite[Theorem 2.6]{FaltingsCrystallineCohomology}, any subrepresentation and quotient-representation of a crystalline representation is also a crystalline representation. The same holds in the log-crystalline setting by \cite[Theorem 2.6', page 43, i)]{FaltingsCrystallineCohomology}.

\end{enumerate}
\end{rmk}

\subsubsection{Recovering $M$ from $T_{\cris}M$.}\label{mfromtcrism}

  Following Brinon, we mention how to recover $M\otimes \Q_p$ from $T_{\cris}(M)$ -- note that $M\otimes \Q_p$ only depends on $T_{\cris}(M)\otimes\Q_p$ and not on the actual lattice itself. We have $M\otimes \Q_p \cong (T_{\cris}(M)\otimes \cB_{\cris}(\ol{\cR}))^G$. 
 
 It is implicit though not formally stated in \cite{Tsuji} how to recover $M$ from $T_{\cris}(M)$. We summarize it here in two steps:

\begin{itemize}
    \item $TA_{\inf}(M)$ is the unique $G$-stable free $A_{\inf}(\ol{\cR})$-submodule of $T_{\cris}M\otimes A_{\inf}(\ol{\cR})$ which generates $T_{\cris}M\otimes A_{\inf}(\ol{\cR})[\frac1\pi]$ and which is \emph{trivial} modulo $\pi$ in the following sense: namely, that $TA_{\inf}(M)/\pi$ is isomorphic to $(A_{\inf}/\pi)^{\textrm{rank }M}$ as a (semi-linear) $A_{\inf}/\pi$-representation of $G$.
       \cite[Lemma 64(2) + (66)+ Prop 76]{Tsuji}
    
    \item  $M\cong \left(TA_{\inf}(M)\otimes_{A_{\inf}(\ol{\cR})}\cA_{\cris}(\ol{\cR})\right)^G$. \cite[Prop 61]{Tsuji}\footnote{Technically the proof of this proposition works with $TA_{\cris}(M)$ but the version written follows immediately from the definition of $TA_{\inf}(M)$ given above.}

\end{itemize}
We denote the above procedure by $L\ra S_{\cris}(L)$. We also have that $S_{\cris}(L)\otimes_{\Z_p} \Q_p \cong (L\otimes A_{\cris}(\ol{\cR})[1/\pi])^G$ for crystalline local systems $L$ (note that inverting $\pi$ also inverts $p$).

We shall require the following (almost certainly known) results for later:

\begin{lemma}\label{cristensor}

Let $L_1,L_2$ be crystalline local systems of weights in $[0,a]$ and $[0,b]$ such that $a+b\leq p-2$. Then $L_1\otimes L_2$ is crystalline
and there is a natural isomorphism $$S_{\cris}(L_1)\otimes S_{\cris}(L_2)\ra S_{\cris}(L_1\otimes L_2)$$. 

\end{lemma}

\begin{proof}
Define $M_i:=S_{\cris}(L_i)$ for $i=1,2$. Then $M:=M_1\otimes_\cR M_2$ is naturally an object of $MF^{\nabla}_{[0,a+b],\textrm{free}}(\cR,\phi)$.
Thus we may associate to $M$ the module
$T^*_{\cris}(M)$. Now it is clear that
$T^*_{\cris}(M)\supset T^*_{\cris}(M_1)\otimes T^*_{\cris}(M_2)$.  Therefore there is an injective isogeny
$f:T_{\cris}(M)\ra L_1\otimes L_2$. It follows that $L_1\otimes L_2$ is crystalline. 

By \cite[(39)]{Tsuji} it follows  that $$TA_{\cris}(M)\cong TA_{\cris}(M_1)\otimes_{A_{\cris}(\ol\cR)}TA_{\cris}(M_2).$$ It follows by the characterization of $TA_{\inf}(M)$ that
$$TA_{\inf}(M)\cong TA_{\inf}(M_1)\otimes_{A_{\inf}(\ol\cR)}TA_{\inf}(M_2).$$ 

It follows by \cite[Thm 70]{Tsuji} that $f\otimes A_{\inf}(\ol\cR)[\frac1\pi]$ is an isomorphism. Since $A_{\inf}(\ol\cR)[\frac1\pi]$ is a domain it is flat over $\Z_p$. Moreover $A_{\inf}(\ol\cR)[\frac1\pi]\otimes\bF_p \neq 0$ and therefore  $f$ must be an isomorphism. This shows the second part of the claim.

\end{proof}

We next prove the log-crystalline version. To do this, we recall the definition of $A_{\cris,\log}$ following Faltings: In the affine setting, we have a smooth scheme $\Spec R$ over $V$, and a relatively normal crossings divisor $D\subset \Spec R$ defined by a principal ideal $(T)\subset R$, and we define $\ol{R}$ to be the integral closure of $R$ in the maximal \'etale extension of $R[\frac1{pT}]$. The rest of the construction proceeds exactly as for $A_{\cris}$.

\begin{lemma}\label{lem:logcrystensor}
    Let $L_1,L_2$ be log crystalline local systems of weights in $[0,a]$ and $[0,b]$ such that $a+b\leq p-2$. Then $L_1\otimes L_2$ is log crystalline
and there is a natural isomorphism $$S_{\cris}(L_1)\otimes S_{\cris}(L_2)\ra S_{\cris}(L_1\otimes L_2)$$.
\end{lemma}

\begin{proof}
We first construct the map in question. Let $M_i = S_{\cris}(L_i)$ be the corresponding log Fontaine-Laffaile modules. The fact that $M:=M_1\otimes M_2$ is log Fontaine-Laffaille is obvious by definition. By Faltings' description on the top of \cite[page 37]{FaltingsCrystallineCohomology}, the Galois action on $M\otimes A_{\cris,\log} $ is determined by specifying the action of the geometric Galois group $\Delta$ where an element $\sigma$ acts on $m\otimes 1$ via 
    $e^T$ where $T:=\sum_i \nabla(\partial_i)\otimes \beta(\sigma_i)$. Note that $T$ is a derivation, hence
    \begin{align}
    e^T (m_1\otimes m_2\otimes 1)&=\sum_i \frac{1}{i!}\cdot T^i(m_1\otimes m_2\otimes 1)\\
    &= \sum_{a,b}\frac{\binom{a+b}{a}}{(a+b)!}  T^a(m_1)\otimes T^b(m_2)\otimes 1)\\
    &= e^T(m_1\otimes 1)\otimes e^T(m_2\otimes 1)
    \end{align}
    showing that that the natural isomorphism
    \begin{equation}\label{eq:crystensorisom}
    \big(M_1\otimes A_{\cris,\log}\big)\otimes \big(M_2\otimes A_{\cris,\log}\big)\ra M\otimes A_{\cris,\log}
    \end{equation} is Galois equivariant. 

    Recall that the dual of the crystalline local systems associated to a log Fontaine-Laffaile module $M$ is just $$T_{\cris}(M):=\Big(\varprojlim_{n} \Hom_{\Fil, \varphi} (M/p^n\otimes A_{\cris,\log} , A_{\cris,\log}/p^n)\Big)^*.$$ Combining with \eqref{eq:crystensorisom}, this allows us to view $(L_1\otimes L_2)^*$ as a Galois-stable sub-lattice of $(T_{\cris}(M_1\otimes M_2))^*$. Therefore, we have that $T_{\cris}(M_1 \otimes M_2)$ is a Galois stable sublattice of $L_1\otimes L_2$, and so $L_1 \otimes L_2$ is a log-crystalline local system by Remark \ref{logfontainelaffaille}(3).
    
    In order to show that $L_1 \otimes L_2 = T_{\cris}(M_1 \otimes M_2)$, it suffices to establish this equality when both local systems are restricted to the crystalline locus. But this is true by Lemma \ref{cristensor}. We have proved that $T_{\cris}(M_1\otimes M_2) = T_{\cris}(M_1) \otimes T_{\cris}(M_2)$, and therefore it follows that $S_{\cris}(L_1\otimes L_2) = S_{\cris}(L_1) \otimes S_{\cris}(L_2)$.

\end{proof}

\subsection{Compatibility of passage to fibers}

We fix a map $i:\cR\ra V$. Our goal is to show that the constructions given in the previous section are compatible with restriction to $i$.

We fix an extension $i:\ol{\cR}\ra \ol{V}$ to a geometric point, inducing a map $G_V\ra G_{\cR}$. 
Let $L$ be a crystalline $G_{\cR}$ representation, which also gives a $G_V$ representation $L_i$ via the map $G_V\ra G_{\cR}$ induced by $i$. 

By following the description in \ref{mfromtcrism}, we first form $TA_{\inf}(L)$ as the unique $G_{\cR}$-stable free $A_{\inf}(\ol{\cR})$-submodule of $L\otimes A_{\inf}(\ol{\cR})$ which generates $T_{\cris}M\otimes A_{\inf}(\ol{\cR})[\frac1\pi]$ and which is trivial modulo $\pi$.

\begin{lemma}\label{ainfcomp}
There is a functorially induced isomorphism $$TA_{\inf}(L)\otimes_{A_{\inf}(\ol{\cR})} A_{\inf}(\ol V) \ra TA_{\inf}(L_i).$$
\end{lemma}
\begin{proof}

Note that there is a natural map $A_{\inf}(\ol{\cR})\ra A_{\inf}(\ol{V})$
induced by functoriality, which is equivariant for the map of Galois groups. Now consider the map 
$$i^\# :TA_{\inf}(L)\otimes_{A_{\inf}(\ol{\cR})} A_{\inf}(\ol V) \ra L\otimes A_{\inf}(\ol{V}).$$ Note that $i^\#$ is surjective after tensoring with the fraction field $F$ of $A_{\inf}(\ol V)$ (in fact, even after inverting $\pi$). Since both sides have the same dimension as $L$, it follows that $i^\#$ is an isomorphism after tensoring with $F$. It follows that $i^\#$ is injective (here we are using that $A_{\inf}(\ol V)$ is torsion free).

Thus, the image of $i^\#$ satisfies the universal property characterizing $TA_{\inf}(L_i)$ and the claim follows. 
\end{proof}

We come to our main theorem for this subsection.

\begin{theorem}\label{fibercomp}
There is a functorially induced isomorphism $ S_{\cris}(L)\otimes_\cR V\cong S_{\cris}(L_i).$
\end{theorem}

\begin{proof}

Recall that $S_{\cris}(L)\cong \left(TA_{\inf}(L)\otimes_{A_{\inf}(\ol{\cR})}\cA_{\cris}(\ol{\cR})\right)^{G_{\cR}}$. Moreover, by \cite[(39)]{Tsuji} it follows that
$$S_{\cris}(L)\otimes_{\cR}\cA_{\cris}(\ol{\cR})\cong TA_{\inf}(L)\otimes_{A_{\inf}(\ol{\cR})}\cA_{\cris}(\ol{\cR}).$$

And so

$$(S_{\cris}(L)\otimes_{\cR} V)\otimes_V \cA_{\cris}(\ol{V})\cong TA_{\inf}(L)\otimes_{A_{\inf}(\ol{\cR})}\cA_{\cris}(\ol{V}),$$

$$(S_{\cris}(L)\otimes_{\cR} V)\otimes_V \cA_{\cris}(\ol{V})\cong (TA_{\inf}(L)\otimes_{A_{\inf}(\ol{\cR})}A_{\inf}(\ol{V}))\otimes_{A_{\inf}(\ol{V})}\cA_{\cris}(\ol{V}),$$

$$(S_{\cris}(L)\otimes_{\cR} V)\otimes_V \cA_{\cris}(\ol{V})\cong TA_{\inf}(L_i)\otimes_{A_{\inf}(\ol{V})}\cA_{\cris}(\ol{V})$$

where the last step follows by Lemma $\ref{ainfcomp}$. Finally, it follows that 
$$(S_{\cris}(L)\otimes_\cR V)\cong \left(TA_{\inf}(L_i)\otimes_{A_{\inf}(\ol{V})}\cA_{\cris}(\ol{V})\right)^{G_V}$$ from which the Theorem follows.

\end{proof}
 
We end this subsection with an elementary lemma. \begin{lemma}\label{cristwistsok}
Let $\chi$ denote the cyclotomic character\footnote{We work with the convention that the cyclotomic character has weight -1} (thought of as a $\Z_p$-representation). Let $L$ be a crystalline $G_{\cR}$-module with weights between $a$ and $ b$, with $0\leq a\leq b\leq p-2$. Then, $L\otimes \chi^{-n}$ is also crystalline for $-a\leq n \leq p-2-b $.
\end{lemma} 
\begin{proof}
The lemma follows from Lemma \ref{cristensor} if $n\geq 0$ (we remark that $\chi^{-n}$ is indeed crystalline for $0\leq n\leq p-2$). Therefore, suppose that $n<0$. Let $M = (M,\Fil,\nabla,\Phi)$ denote the object in $\MF$ associated to $L$. For any $n\in \Z$, consider the data $M(-n) :=(M,\Fil(-n),\nabla,\Phi(-n))$ where $\Fil(-n)^i = \Fil^{i-n}$, and $\Phi(-n) = p^n\cdot \Phi$. For $-a\leq n \leq p-2-b$, it's clear that $\Fil(-n)^0 = M$, $\Fil(-n)^{p-1} = \{0\}$. The compatibility of $\Phi(-n)$ with $\Fil(-n)$ follows directly from the compatibility of $\Phi$ and $\Fil$.

Let  $T_{\cris}(M(-n)) = L'$. We then have that $L'\otimes \chi^n$ is crystalline (by Lemma \ref{cristensor}), and further, that $S_{\cris}(L' \otimes \chi^n) \cong S_{\cris}(L')\otimes S_{\cris}(\chi^n)$. However, $S_{\cris}(\chi^n)$ is just the data $\cR(n) = (\cR,\Fil_{\textrm{Triv}}(n),\nabla_{\textrm{Triv}},\Phi_{\textrm{Triv}}(n) ) $ where:

\begin{itemize}
    \item $\Fil_{\textrm{Triv}}(n)^{-n} = M$ and $\Fil_{\textrm{Triv}}(n)^{1-n} = 0$,
    \item $\nabla_{\textrm{Triv}}(n) = d$, and 
    \item $\Phi_{\textrm{Triv}}(n)(1\otimes r) = p^{-n}r$, for $r\in \cR$. 
\end{itemize} 

It therefore follows that $S_{\cris}(L'\otimes \chi^n) \cong S_{\cris}(L)$, and since $S_{\cris}$ is fully faithful, we conclude that $L' \cong L\otimes \chi^{-n}$. 

\end{proof}

\begin{rmk}
    The same conclusion holds in the setting of log-crystalline local systems. 
\end{rmk}


\subsection{Compatibility of Tsuji with Liu-Zhu} 

We will need to compare the relative Fontaine-Laffaile correspondence (both in the crystalline case and the log-crystalline case) with Liu-Zhu's functor later in the paper. This follows by assembling the results of \cite{LZ}, \cite{LZ2}, \cite{FaltingsCrystallineCohomology} and \cite{Tsuji}:

\begin{theorem}\label{thm:crystalline-DeRhammap}
Let $L$ be a crystalline (respectively log-crystalline) $\Gal_{\cR}$-module. Then $L$ is de Rham in the sense of Scholze and there is a natural isomorphism of filtered vector bundles with connection $S_{\cris}(L)\otimes\Q_p\ra D^0_{\dR}(L)$.
    
\end{theorem}
\begin{proof}
    There is an embedding $\mathcal{A}_{\cris} \rightarrow \OBdr$ (see for eg. \cite[Proposition 8.2.12]{brinon}), and upto inverting $p$ the Fontaine-Laffaile module associated to a crystalline local system $\bL$ is given by $(L[1/p]\otimes \mathcal{A}_{\cris})^{\Gal_R}$. However, we are not aware of a reference that constructs a logarithmic analogue of $\mathcal{A}_{\cris}$. We will therefore deduce the crystalline compatibility using results of \cite{FaltingsCrystallineCohomology} and \cite{Tsuji} that bypasses the need for $\mathcal{A}_{\cris}$, and only uses $A_{\cris}$ and $\OBdr$ (both of which have logarithmic avatars in \cite{FaltingsCrystallineCohomology} and \cite{LZ2}).
    
    We will first consider the crystalline case to illustrate the idea. Let $M$ denote the Fontaine-Laffaile module associated to $L$. By \cite{LZ}, the filtered flat bundle $D^0_{\dR}(L)$ is given by $(L \otimes \OBdr)^{\Gal_R}$. It suffices to show that $M$ maps naturally to a subset of Galois fixed elements of $L \otimes \OBdr$. 

    Faltings associates to $M$ an $A_{\cris}$-module, which (following Tsuji) we shall denote by $T_{A_{\cris}}(M)$. This is a filtered module equipped with a Frobenius action, as well as an action of the Galois group (\cite[Page 37]{FaltingsCrystallineCohomology}). Note however that $T_{A_{\cris}}(M)$ is not equipped with a flat connection -- indeed, Faltings defines the Galois action on $T_{A_{\cris}}(M)$ using the connection on $M$. We have that $T_{A_{\cris}}(M)$ is isomorphic to $M\otimes_{\beta} A_{\cris}$, where $\beta: R\rightarrow A_{\cris}$ is constructed in page 192 of \cite{Tsuji} (see also \cite[Page 36]{FaltingsCrystallineCohomology}, where the same map is described). Note that this isomorphism is compatible with Frobenius and the filtration, but not with the Galois action. 
    
    We note that $A_{\cris}$ is equipped with a natural embedding into $\OBdr$. Therefore, there are two maps $R\rightarrow \OBdr$ -- one the natural inclusion and the other constructed using $\beta$. By \cite[Equation (40)]{Tsuji}, there is an isomorphism $M \otimes_{\beta} \OBdr \rightarrow M\otimes \OBdr$ that is compatible with all the extra structure (where the Galois action on the left hand side is defined via the one on $T_{A_{\cris}(M)}$, and the connection on the left hand side has $T_{A_{\cris}}(M)$ as its set of flat sections). The Galois action on the right hand side is solely through $\OBdr$, and the filtration and connection are defined by the tensor-product filtration and connection. The left hand side is naturally isomorphic to $L\otimes \OBdr$ (by \cite[Theorem 2.6* remark h]{FaltingsCrystallineCohomology}). This yields the required embedding of $M$ into $D^0_{\dR}(L)$. 

    To treat the log-crystalline case, we first recall Tsuji's explicit formula (\cite[Equation (42)]{Tsuji}) in the crystalline case that embeds $M$ in $M\otimes_{\beta} \OBdr$. We will use the toric chart and coordinates $s_i$ introduced in Section \ref{sec:padicHTnotation}.
    
    For $\underline{N} = (N_1,\hdots, N_n)$ a tuple of integers, define $\nabla_{\underline{N}}(m) = \Big(\prod_{1\leq i \leq n} \prod_{0\leq j \leq N_i-1} (\nabla(s\frac{d}{ds_i}) -j)\Big)(m)$, where $m\in M$. Let $X'_i = ([s_i^{\flat}]^{-1}\otimes 1) X_i$. Then, we have:  \begin{equation}\label{eqtsuji}
    m\otimes 1 \mapsto \sum_{\underline{N}} (\nabla_{\underline{N}}(m) \otimes \prod_{i}X_i^{\prime [ N_i]}).    
    \end{equation}
    Here, the exponent $^{[N_i]}$ refers to the divided power operation $\frac{X_i^{\prime N_i}}{N_i \!}$.

    We now consider the logarithmic setting, and work with the toric chart $V[s_1,s_2,\hdots s_n] \rightarrow R$, where the log structure is given by $s_1s_2\hdots s_n $. We again let $M$ denote a log-Fontaine-Laffaile module and we let $\bL$ denote the log-Crystalline local system associated to it. By \cite[Equation 2.3.14]{LZ2}, the sum in equation \eqref{eqtsuji} converges and thus gives an isomorphism of $\OBdrlog$-modules.
    $$\Phi: M\otimes\OBdrlog \ra M\otimes_\beta\OBdrlog.$$

    By \cite[Theorem 2.6*, remarks h, i]{FaltingsCrystallineCohomology}, there is a natural map $\Psi:M\otimes_{\beta} A_{\cris,\log} \rightarrow L\otimes A_{\cris,\log}$ which is an ismorphism upto inverting $t$, which yields natural morphisms;

    $$\tau: M\otimes\OBdrlog \xrightarrow{\Phi} M\otimes_\beta\OBdrlog \xrightarrow{\Psi} L\otimes \OBdrlog.$$

    Since $\Phi$ and $\Psi$ are Galois equivariant for the respective actions, $\tau(M\otimes 1)\subset ( L\otimes \OBdrlog)^{\Gal_R}=D^0_{\dR}(L)$, which completes the proof.

\end{proof}

\section{Adelically Metrized Bundles}\label{S-metrized}

Throughout we work with the norm on $\bar\Q_p$ and all of its subfields such that $|p|=\frac1p$. Given a $\bar\Q_p$ vector space $V$, we work with norms $|\cdot |$ on $V$ such that $|\alpha v|=|\alpha||v|$ for all $\alpha\in\bar\Q_p$.

\begin{definition}\label{def:acceptable}
Let $F$ be a non-Archimedean local field, and let $X/F$ denote a rigid-analytic variety with a vector bundle $V$ on it. 
Following \cite[2.7.1]{BG}, we define a norm on $V$ to be a norm $|\cdot|_x$ on every fiber $V_x, x\in X(\ol F)$. We say such a norm is \emph{acceptable} if it is Galois-invariant, and for every affinoid $U \subset X$ on which $V$ is trivial, there exist sections $s_1,\dots,s_n\in V(U)$ trivializing $V$ on $U$, and an integer $N$ such that for every $u\in E$ we have
\begin{enumerate}
    \item $\log_p|s_i(u)|_u\leq N$
    \item $p^N\cdot \{v\in V_u: |v|_u\leq 1\}\subset\bigoplus_{i=1}^n \cO_{\ol F} \cdot s_i(u)$.
\end{enumerate}


\end{definition}

\begin{lemma}\label{anysections}
If $U$ is an affinoid, $V, s_i$ is a vector bundle with a set of trivializing section and $|\cdot|$ is a norm such that the $s_i$ satisfy the conditions above, then any trivializing sections $t_1,\dots,t_n$ satisfy the two conditions above.

\end{lemma}

\begin{proof}

Let $g\in\Gamma(U,\GL(\cO_U))$ be the matrix such that $g\vec{s}=\vec{t}$. Note that $g,g^{-1}$ have entries that are elements of $\cO_U(U)$ and therefore are globally bounded. It follows that the norms of $\vec{t}$ are uniformly bounded by those of $\vec{s}$, and also that the two lattices $\cO_{\ol F} \cdot s_i(u)$ and $\cO_{\ol F} \cdot t_i(u)$ are comparable uniformly for $u\in U$. The claim follows.

\end{proof}
\begin{lemma}\label{acceptabilityafinoidcover}
Let $F$ be a non-Archimedean local field, and $X/F$ a rigid-analytic variety, $V$ a vector bundle on $X$ and $|\cdot|$ a norm on $V$.
Then $|\cdot|$ is acceptable iff it is acceptable on a finite affinoid cover.
\end{lemma}

\begin{proof}
Let $X = \bigcup_{j = 1}^JU_j$ denote an affinoid cover, such that there are trivializing sections $s_{i,j}$ of $V|_{U_j}$ which satisfy the conditions of Definition \ref{def:acceptable}. Let $U'$ denote any affinoid on which $V$ is trivial, and let $t_1\hdots t_n$ denote a set of trivializing sections of $V|_{U'}$. Define $U'_j= U'\cap U_j$. By Lemma \ref{anysections}, the $\{t_i\}|_{U'_j}$ satisfy the conditions of Definition \ref{def:acceptable} follows from the fact that the $\{s_{i,j}\}|_{U'_j}$ do. The lemma follows. 
\end{proof}

Given an integral scheme $\cY/\cO_F$ and vector bundle $\cV$ on $\cY$,  there is a natural norm on $\cV^{\rig}$ over $\cY^{\rig}$ which is easily seen to be acceptable. We call this the $\cV$-norm.

We now come to our central definition:

\begin{definition}\label{def:admissible}

Let $X$ be a proper variety over a number field $F$, $V$ a vector bundle on $X$, and $Y$ an open subscheme of $X$. We define an \emph{admissible collection of norms} on $V$ to consist of the following data:

\begin{enumerate}
    \item An integral model $\cY$ of $Y$ over $\cO_F[N^{-1}]$
    \item A vector bundle $\cV$ on $\cY$ extending $V$.
    \item For each infinite place $v$, a continuous metric $h_v$ on $V_v\mid Y_v$.
    \item For each finite place $v\mid N$, a norm $|\cdot|_v$ on $V^{\an}_v\mid Y^{\an}_v$ which extends to an acceptable norm on $X^{\an}_v$.
    \item For each finite place $v\not\mid N$, an acceptable norm  $|\cdot |_v$ on $\cV^{\rig}_v\mid \cY^{\rig}_v$ such that for almost all finite places it is the $\cV_v$-norm. 
\end{enumerate}

Moreover, we say that this collection is \emph{strongly admissible} if it also satisfies that for each infinite place $v$, the metric $h_v$ extends to a continuous metric on $V_v\mid X_v$. Abusing notation somewhat, we refer to $\left(V,(|\cdot|_v)_v\right)$ as an \emph{ (strongly) admissible normed vector bundle} on the triple $(X,Y,V)$.

\end{definition}

\begin{lemma}\label{lem:admissibleind}
For a vector bundle $V$ on a proper variety $X/F$ with open subscheme $Y$, any two admissible collections of norms on $(X,Y,V)$ agree at almost all finite places, and differ by  $O(1)$ at every finite place.
\end{lemma}

\begin{proof}
Any two integral models agree at almost all places, which implies the corresponding norms agree. As for the second claim, it suffices to prove that two acceptable norms differ by a uniform $O(1)$. 

It is sufficient to work over affinoids since ever proper variety is covered by finitely many such, on which the vector bundle is trivial. 

So let $U$ be an affinoid and $\vec{s}$ be any set of trivializing the vector bundle $V$. The claim immediately follows by property 2 of definition \ref{def:acceptable}.
\end{proof}

The following lemma is straightforward:

\begin{lemma}
The restriction of an acceptable  (resp. admissible) norm(resp. collection of norms) to a sub-vector bundle is admissible (resp. admissible). Likewise for the induced norm(s)  on a quotient bundle.
\end{lemma}

\subsection{Heights}

Note that if the vector bundle is a line bundle, then an  acceptable collection of norms is  - when $X=Y$ - nothing other than an $M$-metric \cite{BG} on our line bundle. Analogously to that context, we have an associated height function for points $P\in \cY(\cO_F[N^{-1}])$ given by taking a section $s$ of $V$ not vanishing at $P$ and defining
$$h_{V,(|\cdot|_v,v\in M_F)}(P):=-\sum_v \frac{[F_v:\mathbb{Q}_p]}{[F:\mathbb{Q}]} \log|s(P)|_v.$$ Moreover, this extends in a natural way to points in the algebraic closure $\cY(\cO_{\ol F}[N^{-1}])$.

It follows from Lemma \ref{lem:admissibleind} that up to an additive $O(1)$ the height depends only on the triple $(X,Y,V)$ and the difference of metrics at infinity. If however our metrics are strongly admissible, then this latter point is also bounded by an $O(1)$.   To make it very concrete, we have the following:

\begin{corollary}\label{cor:ourheightsareweil}
Let $X$ be a proper variety over a number field $F$, $Y\subset X$ an open subscheme, and $V$ a line bundle on $X$. Let $(V,(|\cdot|_v)_v)$ be a strongly admissible normed vector bundle on  $(X,Y,V)$. Then the associated height function on $\cY(\cO_{\ol{F}}[N^{-1}])$ is the restriction of a Weil height $h_V$ on $X(\ol \Q)$.
\end{corollary}

\begin{proof}
An integral model $\cX$ for $X$ together with a model $\cV$ of $V$ and continuous metrics at all the infinite places yields a Weil height, and also yields a strongly admissible collection of norms. The lemma now follows from the fact that all strongly admissible collections of norms give the same height up to an additive $O(1)$.
\end{proof}

\section{Canonical Heights and Admissible Metrics}\label{S-canheights}

\subsection{Metrizing $D_\HT$}

Let $K$ be a finite extension of $\Q_p$. We let $G_K$ be the absolute Galois
group of $K$. We let $\C_p$ be the completion of $\bar{K}$, and we let 
$B_\HT:=\oplus_i \C_p(i)$ where we think of $B_\HT$ as a $\C_p$-vector space with a semi-linear $G_K$ action. We metrize $\C_p(i)$ by identifying $\C_p(i)$ with $\C_p$ with the Galois action on 1 being through the $ith$ power of the cyclotomic character, and then pulling back the metric on $\C_p$ under this identification. See \cite{brinonconrad} for background on $p$-adic Hodge theory.

We let $V/\Q_p$ be a finite dimensional $G_K$-representation, and define
$D_\HT(V):=(V\otimes B_\HT)^{G_K}$. We assume that $V$ is Hodge-Tate, which means that $\dim D_\HT(V) = \dim V$. 

In integral $p$-adic Hodge theory, one typically fixes lattices in $V$ and attempts to define a lattice in $D_\HT(V)$. For our purposes it will be more convenient to work with norms, which record more information since they may not be $p^{\Z}$-valued. Thus, let assume that $|\bullet|$ is a $p$-adic norm on $V$ which is invariant under $G_K$. We call such a $(V,|\bullet|)$ a \emph{metrized $G_K$ -representation}, and say it is Hodge-Tate if $V$ is Hodge-Tate (likewise for de Rham, crystalline,...)

We shall construct a norm on $D_\HT(V)$ (in fact on all graded pieces of it) in a way which will works well in families:

\begin{definition}

Let $(V,|\bullet|)$ be a metrized $G_K$-representation which is Hodge-Tate.
For each integer $n$, let $V^\circ_n=(V\otimes \C_p(n))^{G_k}\otimes\C_p$. There is a natural map $V^\circ_n(-n)\ra V\otimes\C_p$ and we let $V_n$ be the image of this map. Finally let $V_{<n}:=\bigoplus_{m<n}V_m$. Now $V_{<n}$ has an induced norm, and thus we may equip $V_{\C_p}/V_{<n}$ with the quotient norm. This then yields a norm on 
$$V_{\C_p}/V_{<n} \otimes_{\C_p} \C_p(n).$$ Note that this latter space is isomorphic to $V^\circ_n$ and we endow it with the corresponding norm. We call it the \emph{intrinsic metric} on $V^\circ_n$ and likewise on $D_\HT(V)$. We call the corresponding norm $\leq 1$ set on $V_n^{\circ}$ the \emph{intrinsic lattice}.
\end{definition}

Recall that $\Gr B_{\dR}\cong B _{HT}$ as rings with Galois-actions. This isomorphism is unique up to an element of $\Q_p^{\times}$. We choose an isomorphism where the image of $t\mod \Fil^2\Bdr$ is sent to the element $1\in \C_p(1)$. Note that this means $\pi \mod \Fil^2(\Bdr)$ also maps to 1 in $\C_p(1)$.

Now let $\V$ be a crystalline representation in the sense of Fontaine-Laffaille. Then by Theorem \ref{thm:crystalline-DeRhammap} we obtain a map $\Gr S_{\cris}(\V)\ra \Gr D_{\dR}(V)\cong D_{HT}(V)$. The image of the crystalline lattice defines another norm, which we call the crystalline norm. We note that the intrinsic and crystalline norms are compatible with tensor and wedge powers, the former by construction as a subquotient and the latter by Lemma \ref{cristensor}. We will now prove that these two norms exactly agree for one-dimensional representations. 
\begin{lemma}\label{absolutecrisintonedim}
Let the the setting be as above, with the further condition that $\V$ is a one-dimensional representation. Then, the two norms agree. 
\end{lemma}
\begin{proof}
 After replacing $K$ with $W(\overline{\F}_p)[1/p]$, the Galois representation being crystalline must be of the form $\chi^{-a}$ with $0\leq a\leq p-2$.
Recall the element $t\in \Fil^1 \Bdr$, on which the Galois action is through the cyclotomic character ($t$ is unique up to scaling by $\Z_p^{\times}$). For brevity, we will use the term intrinsic lattice to denote the norm $\leq 1$ elements of $\Dht(\V)$. It suffices to prove that this equals the crystalline lattice. 

We first compute the intrinsic lattice. Let $t^n$ also denote the image of $t^n \mod \Fil^{n+1} \in \C_p(n)$. We have that $\Dht^n(\V) = \{e \otimes t^nz: \ z\in    \C_p\} \subset L\otimes \C_p(n)$ where $e$ is a generator of $L$. Therefore, the intrinsic lattice consists of elements $\{e\otimes t^n z: z\in W(\overline{\bF}_p)[1/p] \}$ with $\frac{t^n}{\pi^n} \cdot z \in \cO_{\C_p}$, where we identify $\frac{t^n}{\pi^n}$ with $\theta(\frac{t^n}{\pi^n}) \in \C_p$. As $\frac{t}{\pi} \mapsto 1$ under $\theta$, we see that $z\in W(\overline{\bF}_p)$.

We now compute the crystalline lattice. Note that the crystalline lattice is defined by considering the image of $S_{\cris}(L) \subset \Dht^n(L)$. To compute $S_{\cris}(L)$, observe that $TA_{\inf}$ (as in Section 2) is $L\otimes \pi^n A_{\inf}$ (this follows from \cite[Lemma 75]{Tsuji} and the defining property of $TA_{inf}$). Therefore, the crystalline lattice is the image of $(L\otimes \pi^n A_{crys})^{\Gal = 1}$ in $(L\otimes t^n\Bdr)^{Gal = 1}$. This shows that the crystalline lattice consists of elements of the form $\{e\otimes zt^n: z \in W(\overline{\bF}_p)[1/p] \}$ with $zt^n \in \pi^n A_{\cris}$. As $\frac{t}{\pi}$ is a unit in $A_{\cris}$, we see that $z\in W(\overline{\bF}_p)$, and hence the two lattices are the same. 
\end{proof}

As both norms are compatible with tensor products, the above result yields that both norms are invariant upon twisting by powers of the cyclotomic character (as long the representation remains crystalline). 
In fact, this allows us to define the crystalline norm on representations that are only crystalline up to twist.
\begin{definition}\label{crystallinenormoncrystallineuptotwist} 
Let $\V$ be a representation such that there exists some integer $a$ such that $\V(a)$ is crystalline in the sense of Fontaine-Laffaile. Then the identification of $\Gr D_{\dR}(V)$ with $\Gr D_{\dR}(V(a))$ (via multiplication by $t^a$) yields a crystalline norm on $\Gr D_{\dR}(V)$. 
\end{definition}
Lemma \ref{absolutecrisintonedim} shows that the crystalline norm on a representation that is crystalline upto twist doesn't depend on which power of the cyclotomic character we choose to twist it by. For ease of notation, we make the following definition.  
\begin{definition}\label{weightsuptotranslation}
    Let $V$ be a Hodge-Tate Galois representation. We define its weight interval to be $[b_1,b_2]$ if $b_1$ is the smallest Hodge-Tate weight and $b_2$ to be the biggest. We define its weight range to be the non-negative integer $\rng_V = b_2 - b_1$. We define the \textit{normalized weight-sum} to be the $ S - b_1 \dim V$, where $S$ is the sum of the Hodge-Tate weights\footnote{This is just the sum of the Hodge-Tate weights of $V(-b_1)$, which is the unique twist of $V$ whose minimal Hodge-Tate weight is 0.}.
\end{definition}

In the next subsection, we will define the intrinsic norm in the case of families, and show that the intrinsic and crystalline norms are comparable for local systems up to twist. One advantage to carrying out the argument in families is it allows us to deal with points defined over ramified fields by specializing, without having to directly confront the Faltings-Fontaine-Laffaille theory in that context (for example, Tsuji\cite{Tsuji} insists that $p$ is a uniformizer throughout). However, we will first show that the intrinsic and crystalline norms are comparable in the absolute case as it illuminates some of the steps needed in the case of families. This is related to results obtained by Faltings on integral comparison theorems (see \cite{FaltingsCrystallineCohomology}).

\begin{theorem}\label{absolutecrisint}
Let $K$ denote an unramified extension of $\Q_p$ and suppose that $\V$ is a $\Z_p$-representation of $\Gal_K$ that is crystalline (in the sense of Fontaine-Laffaile) up to twist. Suppose that the normalized weight-sum of $\V$ is $\leq p-2$ and that the weight range is $b$. Then the intrinsic norm on $\Gr D_{\dR}(V)$ agrees with the crystalline norm (induced by the image of $\Gr S_{\cris}$) up to a multiple of $p^{\frac{b(\dim \V-1)}{p-1}}$
\end{theorem}
\begin{proof}[Proof of Theorem \ref{absolutecrisint}]
Without loss of generality, we can and will assume that the weight interval is $[0,b]$. We have that $S_{\cris}(\V) \subset \V\otimes A_{\cris} $. Note that for $0\leq i\leq p-2$, $\Gr^i(A_{\cris}) \cong \xi^i\cO_{\C_p}$. Therefore, the image of $\Gr^i S_{\cris}(\V)$ in $\Gr^i D_{\dR}(\V)$ is contained in the Galois invariants of $\V \otimes \xi^i\cO_{\C_p}$. As we have metrized $\Gr^i B_{\dR}$ with the convention that $t^i$ has norm 1, it follows that $\Gr^i S_{\cris}(\V)$ is contained in $\theta(\xi/t)^i$ times the intrinsic lattice. As $v_p(\theta(\frac{\xi}{t})) = \frac{-1}{p-1}$ (see \cite[Page 29]{FaltingsCrystallineCohomology}), the crystalline lattice is contained in $p^{\frac{-b}{p-1}}$ times the intrinsic lattice.

To obtain the other inclusion, note that the intrinsic norm and crystalline norm are both compatible with tensor and wedge powers, the former by construction as a subquotient and the latter by Lemma \ref{cristensor}. We thus get two $\cO_{\C_p}$ lattices $A,B$ such that $A\subset p^{\frac{-b}{p-1}}B$ and $\det A=\det B$. 

Let $v_1\in B$ be primitive and such that $p^cv_1\in A$ is primitive, and complete it to a basis $p^cv_1,w_2,\dots,w_{\dim V}$ of $A$. Then since $\det A=\det B$ we have $$1\leq |p^{c} v_1|_B\ds\prod_{i=2}^{\dim V} |w_i|_B \leq p^{-c} p^{\frac{b(\dim V-1)}{p-1}}$$ from which the desired inequality $c\leq \frac{b(\dim V-1)}{p-1}$ follows.

\end{proof}


\subsection{Metrizing Hodge-Tate metrics in families}

We will borrow heavily from \cite{LZ},\cite{LZ2}, and \cite{S}. Let $X$ be a smooth geometrically connected rigid-analytic variety over $k$. In the case where one exists, we let $\ol{X}$ denote a partial compactification of $X$ such that $D:=\ol{X}-X$ is a smooth normal crossings divisor. In this setting, we equip $\ol{X}$ with the structure of a smooth log-adic space in the natural way. We will assume that $L$ is a $\Z_p$-local system on $X_{et}$ which is de Rham in the sense of \cite{S}, and such that the local geometric monodromy around the components of $D$ is unipotent. 

We shall work with the pro-\'etale site $X_{proet}$ and the pro-Kummer-\'etale site $\ol{X}_{proket}$. We shall denote by $L_{ket}$ the extension of $L$ to $\ol{X}_{ket}$, which is also a $\Z_p$-local system\footnote{To those unfamiliar, the point is that $L$ is an inverse limit of local systems with finite fibers. On each of those, monodromy is finite and can therefore be unwound by taking a sufficiently large Kummer extension.} 

Let $\widehat{\cO}_X$ and $\widehat{\cO}_{\ol X}$ be the completed structure sheaf on $X_{proet},\ol{X}_{proket}$ respectively. Let $\widehat{L}$ be the induced sheaf of local systems
on $X_{proet}$ and likewise let $\widehat{L}_{ket}$ be the sheaf on $\ol X_{proket}$. Finally we set $\bar{L}:=\widehat{L}\otimes_{\widehat{\Z_p}}\widehat{\cO}_X,$ and 
$\bar{L}_{ket}:=\widehat{L}_{ket}\otimes_{\widehat{\Z_p}}\widehat{\cO}_{\ol X}$. These are locally free sheaves on $X_{proet}, \ol{X}_{proket}$ (over their respective structure sheaves).  Let $\nu:X_{proet}\ra X_{et}, \nu_{ket}:\ol{X}_{proket}\ra\ol{X}_{et}$ be the natural maps. Note that $\bar{L},\bar{L}_{ket}$ have natural norms at every classical point of $X,\ol{X}$ which agree on $X$. 

We will now define an ascending filtration on $\bar{L},\bar{L}_{ket}$, and will use this filtration to metrize the graded pieces of $D^0_{\dr}(L)$. We work only in the Kummer case as it is more general, though it is unnecessary (and we will drop it) when we work in settings without log-structures (or, equivalently, trivial log-structures)

First, recall that by combining the second displayed equation in the proof of  \cite[Lemma 3.3.17]{LZ2}, the first equation in the proof of \cite[Cor 3.4.22]{LZ2}, and the isomorphism at the bottom of page 36 of \cite{LZ2}, we obtain

\begin{equation}\label{ds1}
\bigoplus_r \nu^*_{ket}\nu_{ket,*}(\widehat{L}_{ket}\otimes\cO\C_{\log}(r))\otimes_{\nu^*_{ket}\cO_{\bar{X}}} \cO\C_{\log}(-r) \cong \hat{L}_{ket}\otimes \cO\C_{\log}.
\end{equation}

Locally on $X$, we may pick a smooth toric chart with co-ordinates $x_1,\dots,x_m$. Recall also that we have \cite[2.3.17]{LZ2}

\begin{equation}\label{ds2}
    \cO\C_{\log}\cong \hat{\cO}_{\bar{X}}\bigg[\frac{y_1}{t},\dots,\frac{y_m}{t}\bigg]
\end{equation}

where the $y_i$ are as in \cite[2.3.6]{LZ2}.

Assume the weights for $L$ are in $[0,m]$. We have the following important structural theorem(see also \cite[\S5]{hansen}):

\begin{theorem}\label{hansenright}

Let $L$ be as above. Then:

\begin{enumerate}
    \item For $r\geq 0$ The elements of $\nu_{ket,*}(\widehat{L}_{ket}\otimes\cO\C_{\log}(r))$ have degree (in the variables $y_i$) $\leq r$ under the isomorphism \eqref{ds2}.
    \item For $r\geq 1$, the positive degree monomials of $\nu_{ket,*}(\widehat{L}_{ket}\otimes\cO\C_{\log}(r))$, lie in the $\cO\C_{\log}$ - span of $\nu_{ket,*}(\widehat{L}_{ket}\otimes\cO\C_{\log}(s))(r-s)$ for $0\leq s<r$, when pulled back via $\nu_{ket}^*$ to $\hat{L}_{ket}\otimes \cO\C_{\log}(r)$ under the identification \eqref{ds2}.
    
\end{enumerate}

\end{theorem}

\begin{proof}

\begin{enumerate}

    \item We proceed by strong induction on $r$. Consider the induced map
    $$\nabla_r:\nu_{ket,*}\bigg(\hat{L}\otimes\cO\C_{\log}(r)\bigg)\ra \nu_{ket,*}\bigg(\hat{L}\otimes\cO\C_{\log}\otimes_{\cO_{\ol X}}\Omega_{\bar {X}}^{\log}(r-1)\bigg)$$
    given by \cite[2.4.2(4)]{LZ2}. Now $\Omega_X^{\log}$ is locally free and so
    $$\nu_{ket,*}\bigg(\hat{L}\otimes\cO\C_{\log}\otimes_{\cO_{\ol X}}\Omega_X^{\log}(r-1)\bigg)\cong \nu_{ket,*}\bigg(\hat{L}\otimes\cO\C_{\log}(r-1)\bigg)\otimes_{\cO_{\ol X_{\et}}}\Omega^{\log}_{\ol X_{\et}}$$.
    
    Therefore by induction, the image consists of elements of degree $\leq r-1$. We now finish with the claim that if $\alpha\in\cO\C_{\log}$ and  $\nabla(\alpha)$ has degree $\leq r-1$ then $\alpha$ has degree $\leq r$. Indeed, writing $$\alpha = \sum_{h\in \Z^m_{\geq0}} a_h {\left(\frac{\vec{y}}t\right)}^h$$ we compute that
    $$\nabla\alpha = \sum_{h\in \Z^m_{\geq0}} a_h\sum_{1\leq i\leq m}{\left(\frac{\vec{y}}t\right)}^{h-e_i}\delta(y_i).$$ where we have used the notation in \cite[\S2.4]{LZ2}. Now since $\Omega_{\bar{X}_{\et}}^{\log}$ is locally free over $\cO_{\ol X}$ with generators $\delta(y_i)$ (see proof of \cite[2.4.2]{LZ2}) the claim follows.
    
    \item We again proceed by induction. As above, applying $\nabla_r$ to an element $f$ of $\nu_{ket,*}\bigg(\hat{L}\otimes\cO\C_{\log}(r)\bigg)$ gives an element $\nabla_r(f)$ of $\nu_{ket,*}\bigg(\hat{L}\otimes\cO\C_{\log}(r-1)\bigg)\otimes_{\cO_{\ol X_{\et}}}\Omega^{\log}_{\ol X_{\et}}$. 
    By induction, all the positive degree monomials of this element lie in the $\cO\C_{\log}$-span of $\nu_{ket,*}(\widehat{L}_{ket}\otimes\cO\C_{\log}(s))(r-1-s)\otimes \Omega^{\log}_{\ol X_{\et}}$ for $0\leq s<r-1$, and therefore the constant term is in the $\cO\C_{\log}$-span of $\nu_{ket,*}(\widehat{L}_{ket}\otimes\cO\C_{\log}(s))(r-1-s)\otimes \Omega^{\log}_{\ol X_{\et}}$ for $0\leq s\leq r-1$.

    Now by the computation in part 1, each positive-degree monomial in $\alpha$ can be extracted from the monomials of $\nabla_r(\alpha)$ by taking the projection operators for $\Omega^{\log}_{\ol X_{\et}}$ with respect to the basis $\delta(y_i)$, and the claim follows.
    
\end{enumerate}

\end{proof}

We now define $\ol{L}_{<r}\subset \ol{L}_{ket}$ to be the 
generated by all coefficients of the constant terms of 

$\nu_{ket,*}\bigg(\hat{L}\otimes\cO\C_{\log}(s)\bigg)(-s)$ for $s<r$. 

It follows from Theorem \ref{hansenright} that \begin{equation}\label{filteredisom}\ol{L}_{<r}\otimes\cO\C_{\log}\cong\bigoplus_{s<r} \nu^*_{ket}\nu_{ket,*}\bigg(\hat{L}\otimes\cO\C_{\log}(s)\bigg)(-s)\otimes_{\nu_{ket}^*\cO_{\bar X}} \cO\C_{\log}
\end{equation}

and consequently by \eqref{ds2} that $\ol{L}_{<r}$ is a locally split filtration of $\ol{L}_{ket}$. 
Finally, we define $\ol{L}_n:=\ol{L}_{<n+1}/\ol{L}_{<n}$.

\begin{proposition}\label{gradedrealization}

We have a natural isomorphism 
$$ \nu_{ket}^*\nu_{ket,*}(\hat{L}\otimes\cO\C_{\log}(n))\otimes_{\nu_{ket}^*\cO_{\bar{X}}} \ho_{\ol{X}}(-n)\cong\ol{L}_n$$

\end{proposition}

\begin{proof}

First, note there is a natural map 
$$\psi:\nu_{ket}^*\nu_{ket,*}(\hat{L}\otimes\cO\C_{\log}(n))\otimes_{\nu_{ket}^*\cO_{\bar{X}}} \ho_{\ol{X}}(-n)\ra \hat{L}\otimes\cO\C_{\log}.$$ The image of $\psi$ is actually in $\bar{L}_{<n+1}\otimes\cO\C_{\log}$ by \eqref{filteredisom} and therefore there is (\'etale locally on $X$)  a map $$\phi:\nu_{ket}^*\nu_{ket,*}(\hat{L}\otimes\cO\C_{\log}(n))\otimes_{\nu_{ket}^*\cO_{\bar{X}}} \ho_{\ol{X}}(-n)\ra \bar{L}_n\otimes\cO\C_{\log}$$ given by composing $\psi$ with the quotient map. Using the isomorphism \eqref{ds2} and part 2 of Theorem \ref{hansenright} we conclude that the image of $\phi$ is in fact simply $\bar{L}_n$, which gives us our desired map. 

It remains to show this map is an isomorphism. But it becomes one when tensored with $\cO\C_{\log}$ by \eqref{filteredisom} and so the claim follows since the $\cO\C_{\log}$ is faithfully flat (in fact free) over $\ho_{\ol X}$. 

\end{proof}

Given the theorem, we make the following definition. 

\begin{definition}\label{intrinsicnormdefn}
We endow  $\Gr^n D^0_{\dR}(L) = \nu_{ket \ *}(\widehat{L}_{ket}\otimes\cO\C_{\log}(n))$ with the norm coming from $\bar{L}_{n}$, induced from the quotient norm on $\bar{L}$, and we call this the \emph{intrinsic norm}.
\end{definition}
We remark that the norm on $\bar{L}$ and therefore $\bar{L}_n$ is Galois-invariant and therefore descends to a norm on $\Gr^n D^0_{\dR}(L)$.


\begin{lemma}\label{lem:intrinsicacceptable}

The intrinsic norms on $\nu_*(\widehat{L}\otimes\cO\C)(n))$ and  $\nu_{ket \ *}(\widehat{L}_{ket}\otimes\cO\C_{\log}(n))$ are acceptable, (according to definition \ref{def:acceptable}).

\end{lemma}

\begin{proof}

We deal with the $\ol{X}$ case, the other case being identical. By Lemma \ref{acceptabilityafinoidcover}, it suffices to check the conditions of acceptability with respect to any one finite affinoid cover. Therefore, we may assume that all the graded pieces $\Gr^i D^0_{\dR}(L)$ are trivial vector bundles. 

We work with a subaffinoid $U\subset \ol{X}$ and a perfectoid cover $\tilde{U}=\varprojlim_i U_i$ on which $L_{ket}$ trivializes. We may then pick a basis $e_j$ for $L_{ket}$ and a basis $l_i$ for $\bar{L}_{ket}$ compatible with the filtration $\bar{L}_{<n,ket}$. Let $t_{i,j}\in\hat{\cO}_{\ol X}(\tilde{U})$ be the coefficients of the $l_i$ w.r.t the $e_j$. That the norms of the $t_{i,j}$ are bounded above is trivial, hence condition 1 of definition \ref{def:acceptable}) is satisfied. 

For condition 2, note that the norm 1 lattice in the quotient is the image of the norm 1 lattice in the source which is generated by the $e_j$.
Since the $l_i$ are a basis, that means that they generate the $e_j$ over $\hat{\cO}_{\ol X}(\tilde{U})$. The demnominators of the coefficients are bounded, which immediately yields condition 2. 
\end{proof}

The intrinsic norm is completely functorial, as per the following result:

\begin{proposition}\label{intrinsifunctoriality}

Let $X$ be a smooth rigid analytic variety, and let $L$ be a $\Z_p$-local system on $X_{et}$ which is de Rham. Let $\phi:Y\ra X$ be a map of smooth rigid analytic varieties. Then the intrinsic norm on $$\nu_{X}^*\nu_{X,*}((\hat{L}\otimes\cO\C_{X,\log}(n))\otimes_{\nu_{X}^*\cO_{\bar{X}}} \ho_{\ol{X}}(-n))$$ pulls back to the intrinsic norm on $$\nu_{Y}^*\nu_{Y,*}((\widehat{\phi^{-1}L }\otimes\cO\C_{Y,\log}(n))\otimes_{\nu_{Y}^*\cO_{\bar{Y}}} \ho_{\ol{Y}}(-n)).$$ 

\end{proposition}

\begin{proof}

Note that the natural map $\phi^*(\nu_{X,*})((\hat{L}\otimes\cO\C_{X,\log}(n)) \ra ((\widehat{\phi^{-1}L}\otimes\cO\C_{Y,\log}(n))$ is an isomorphism by \cite[Thm 2.1(3)]{LZ}. It follows from comparing the direct sum decompositions \eqref{ds2} for$X$ and $Y$ that $\phi^*\ol{L}_{<n}\cong \ol{\phi^{-1}L}_{<n}$ and thus that $\phi^*\ol{L}_{n}\cong \ol{\phi^{-1}L}_{n}$. The claim follows.

\end{proof}

\subsection{Comparison of the two metrics}
Now, suppose that $X$ has a smooth integral model $\cX$ over $\cO_k$ and let $L$ be a crystalline local system on the generic fiber of the formal completion $\widehat{\cX}_{\gen}$.

Then by Theorem \ref{thm:crystalline-DeRhammap} we get a natural map $S_{\cris}(L)\ra D^0_{\dR}(L)$ which induces a map $$\Gr S_{\cris}(L)\ra \Gr D^0_{\dR}(L)\cong D_{HT}(L)=\oplus_{n} \nu_*(\widehat{L}\otimes\cO\C(n)).$$ We thus obtain a canonical lattice in $L^{\circ}_{n}$ which we call the \emph{crystalline lattice}. We define the corresponding norm the \emph{crystalline norm}. We will show that the two norms agree in two cases: the case of crystalline local systems all of whose Hodge-Tate weights are 0 and rank 1 crystalline local systems. 
\begin{lemma}\label{compareweight0}
Let $L$ denote a crystalline local system with all Hodge-Tate weights 0. Then, the crystalline and intrinsic norms agree. 
\end{lemma}
\begin{proof}
 We will first show that the crystalline lattice is contained in the norm $\leq 1$ subset of the intrinsic norm. For brevity, call the latter subset the intrinsic lattice.
 
 As 0 is the only Hodge-Tate weight, we have that $D^0_{\dR}(L) = \nu_* (\hat{L}\otimes \ho_X)$. Further, we have $S_{\cris}(L) \subset (L\otimes \Gr^0 \cA_{\cris})^{\Gal = 1} = (L\otimes \overline{\cR})^{\Gal = 1}$. It follows that the image of $ S_{\cris}(L)$ in $D^0_{\dr}(L)$ is contained in $L\otimes \overline{\cR}$, and therefore the crystalline lattice is contained in the intrinsic lattice. 
 
 To finish the lemma, it suffices to prove the other containment. In particular, it suffices to prove that the crystalline and intrinsic lattices agree for the determinant of $L$ (as both norms and therefore lattices are compatible with tensor products, see Lemma \ref{cristensor}). Therefore, we assume that $L$ is a crystalline rank-1 local system with Hodge-Tate weight 0. By \cite[Theorem 2.6*, h)]{FaltingsCrystallineCohomology}, the dual local system $L^*$ of $L$ is also crystalline. As $L^*$ (and $L\otimes L^*$, which is the trivial local system) also has Hodge-Tate weight 0, we have that the crystalline lattices for $L,L^*$ and $L\otimes L^*$ are each contained in the intrinsic lattices for each local system. As $L\otimes L^*$ is the trivial local system, the two lattices clearly agree for $L\otimes L^*$. It follows from lemma \ref{cristensor} that the two lattices therefore must agree for both $L$ and $L^*$. The lemma follows. 
 
\end{proof}

\begin{lemma}\label{compchar}
Let $L$ denote a rank 1 crystalline local system with Hodge-Tate weight $0\leq n\leq p-2$. Then the intrinsic norm is the same as the crystalline norm.
\end{lemma}
\begin{proof}
Consider the local system $L' = L\otimes \chi^n$. This is crystalline with Hodge-Tate weight 0, and so the intrinsic and crystalline norms agree for $L'$. 

We identify $L,L'$ as rank-1 $\Z_p$-modules (with the Galois action differing by a twist), and therefore also identify $L\otimes A$ with $L'\otimes A$ with $A = A_{\inf}(\overline{\cR})$ or $\cA_{\cris}(\overline{\cR})$. We choose a generator $e$ of $L$. By \cite[Lemma 75]{Tsuji}, we have that $TA_{\inf}(L) = \pi^nTA_{\inf}(L')$. We view $TA_{\inf}(L'), TA_{\inf}(L)$ as subsets of $L'\otimes \cA_{\cris}(\overline{\cR})$. We claim that $e\otimes f$ (with $f\in \cA _{\cris}(\overline{\cR})$) is an element of $S_{\cris}(L')$ if and only if $e\otimes t^nf$ is an element of $S_{\cris}(L)$. In other words, we claim that $e\otimes f$ is in the $\cA_{\cris}(\overline{\cR})$-span of $TA_{\inf}(L')$ if and only if $e\otimes t^nf$ is in the $\cA_{\cris}(\overline{\cR})$-span of $TA_{\inf}(L)$. This follows from the fact that $\frac{t}{\pi}\in \cA_{\cris}(\overline{\cR})$ is a unit. The claim follows from the easy observation that $e\otimes f \in L\otimes \cA_{\cris}(\overline{\cR})$ is Galois invariant if and only if $e\otimes t^nf \in L\otimes \cA_{\cris}(\overline{\cR})$ is. Therefore, the crystalline norm of $e\otimes f \in S_{\cris}(L')$ is the same as the crystalline norm of $e\otimes t^nf \in S_{\cris}(L)$. 

We now calculate the intrinsic norms. By Lemma \ref{compareweight0}, we have that $S_{\cris}(L')$ generates the intrinsic lattice in $D^0_{\dR}(L')$. Again, by identifying $L,L'$ as $\Z_p$-modules (with Galois actions differing by $\chi^n$), we may identify $\hat{L}\otimes \OBdr$ and $\hat{L'}\otimes \OBdr$ (with the Galois action again differing by $\chi^n$). As in the crystalline case, $e\otimes f \in \hat{L'}\otimes \OBdr$ is Galois invariant if and only if $e\otimes t^nf \in \hat{L}\otimes \OBdr$ is. We normalized the intrinsic norm so that the element $t$ has norm 1, and therefore the intrinsic norm of $e\otimes f \in \hat{L'}\otimes \OBdr$ is the same as the intrinsic norm of $e\otimes t^n f \in \hat{L}\otimes \OBdr$. The result follows.

\end{proof}

Analogous to the absolute case, we may define the crystalline norm on a local system that is crystalline up to twist. 
\begin{definition}\label{normcrystallinelocalsystemuptotwist}
Let $L$ be a local system such that there exists some integer $a$ such that $L(a)$ is crystalline. Then the identification of $\Gr D^0_{\dR}(V)$ with $\Gr D^0_{\dR}(V(a))$ (via multiplication by $t^a$) yields a crystalline norm on $\Gr D^0_{\dR}(V)$. Note by Lemmas \ref{compchar} and \ref{cristensor} that the norm is independent of $a$ (as long as $L(a)$ is crystalline).
\end{definition}

 Our goal is now to show that the crystalline and intrinsic norms  are comparable for local systems that are crystalline up to twist.

\begin{theorem}\label{crisint}
Let $\cX$ be a smooth geometrically connected scheme over $\cO_k$, and $L$ be a $\Z_p$ local system on $\widehat{\cX}_{\gen}$ that is crystalline up to twist. Suppose that the normalized weight-sum of $L$ is $\leq p-2$ and that the weight range is $b$. Then the crystalline norm and the intrinsic norm associated to $L$ are comparable up to $p^{\frac{b(\dim L -1)}{p-1}}$. 
\end{theorem}

\begin{proof}\emph{of Theorem \ref{crisint}:}

Without loss of generality, we may and do assume $X=\Spec(R)$ is affine with an \'etale map $$\cO_k[s_1,s_1^{-1}\dots,s_n,s_n^{-1}]\ra R,$$ and that the weight interval of $L$ is $[0,b]$. Recall
that $v_i=[s_i^{\flat}]\otimes s_i^{-1} -1$ are elements of $\cA_{\cris}(\ol\cR)$ such that
$\cA_{\cris}(\ol\cR)=A_{\cris}(\ol\cR)\langle v_1,\dots,v_n\rangle^{PD}$. Moreover,
$$S_{\cris}(L)\subset \cA_{\cris}(\ol\cR)\otimes L.$$ It follows that the image $S_i$ of $\Fil^i S_{\cris}(L)$ in $\Gr^iD^0_{\dR}(L)$ is contained in the image of $\Fil^iA_{\cris}(\ol\cR)\otimes L$ which is $ L\otimes \xi^i\ol\cR [\frac{v_1}\xi,\dots,\frac{v_n}\xi]^{\leq i}$ by
\cite[(5)]{Tsuji}.

Moreover, it follows by Theorem \ref{hansenright} that the positive degree components of $S_i$ are irrelevant when computing the intrinsic norm. Since we embed $\Dht^i(L)$ inside $L\otimes \ho_X$ by scaling by $\frac{1}{\pi^i}$, it follows that the crystalline lattice is contained inside $\theta(\frac{\xi}{\pi})^i$ times the intrinsic lattice. As $v_p(\theta(\frac\xi\pi))=-\frac1{p-1}$ (see for example \cite[page 29]{FaltingsCrystallineCohomology}), the crystalline lattice is contained within $p^{-\frac{b}{p-1}}$ times the intrinsic lattice.

To get the other inclusion, note first of all that by Lemma \ref{cristensor} that
$S_{\cris}(L\otimes L')=S_{\cris}(L)\otimes S_{\cris}(L')$, and by taking direct summands that
$\bigwedge^{\dim L}S_{\cris}(L)=S_{\cris}(\bigwedge^{\dim L}L) $ - note that the Hodge-Tate weights are $ \leq p-2$ by assumption. The same compatibility for the intrinsic height is immediate. 

It follows by Lemma \ref{compchar} that the top wedge powers of the crystalline lattice and of the intrinsic lattices are the same. An argument identical to one used to conclude the proof of Theorem \ref{absolutecrisint} proves this result. 

\end{proof}

\section{Some Ramified Examples}\label{S-ramifiedexamples}

\subsection{Elliptic curves with CM by a maximal order ramified at $p$}\label{sec:maximalEC}
In this subsection, we work out the comparison in the case of the Galois representation associated to an elliptic curve admitting CM by an imaginary quadratic field ramified at $p$. We will make extensive use of Colmez's paper \cite{col}, where he computes the $p$-adic valuation of certain $p$-adic periods. 

Let $p>2$ be a prime and let $\fp \in \overline{\Q}_p$ denote an element such that $\fp^2 = -p$. Let $K = \Q_p[\fp]$. 
Let $M = \Q[x] / (x^2 - p)$ -- note that $\cO_M\otimes \Z_p$ is obviously isomorphic to $\cO_K$, and there are two equally canonical choices of isomorphism which we denote by $\sigma,\tau$. Let $E$ denote an elliptic curve with CM by $\cO_M$. We note that $E$ is defined over the Hilbert class field $H$ of $M$, and the extension $H/M$ is totally split at the principal prime ideal $\fp$. Therefore, $E$ is defined over $K$, as is the CM action. 

The ring $\cO_M\otimes \Z_p$ acts $\cO_K$-linearly on $H^1_{\dR}(E_{\cO_K})$, the integral de Rham cohomology of $E$. Let $H^1_{\dR}(E_{\cO_K})^\sigma$ denote the one-dimensional subspace on which $\cO_M\otimes\Z_p$ acts via the embedding $\sigma$, and let $H^1_{\dR}(E_{\cO_K})^\tau$ denote the analogous object. The action of $\cO_M\otimes \Z_p$ preserves $\Fil^1$, and therefore $\Fil^1$ is one of two one-dimensional subspaces described above. Without loss of generality, we suppose that $\Fil^1 = H^1_{\dR}(E_{\cO_K})^\sigma$. 

Note that while $H^1_{\dR}(E_{\cO_K})^\sigma$ and $H^1_{\dR}(E_{\cO_K})^\tau$ are each saturated in $H^1_{\dR}(E_{\cO_K})$, the two subspaces don't span $H^1_{\dR}(E_{\cO_K})$ integrally. This follows because $H^1_{\dR}(E_{\cO_K})$ is free of rank 1 as an $\cO_M\otimes_{\Z} \cO_K$-module (\cite{col}). This description then gives us that $H^1_{\dR}(E_{\cO_K})/ (H^1_{\dR}(E_{\cO_K})^\tau \oplus H^1_{\dR}(E_{\cO_K})^\tau)$ is isomorphic to $\cO_K / \fp \cO_K$.

Let $T$ denote the $p$-adic Tate module of $E$. Note that $T$ is naturally a free rank-1 $\cO_M\otimes \Z_p$-module, and the Galois action is $\cO_M\otimes \Z_p$-linear. Let $L = T^{\vee}$, and let $u_1\in L$ denote an $\cO_M\otimes \Z_p$-generator, and let $u_2 = \sigma^{-1}(\fp) u_1$. In \cite{col}, Colmez computes the $p$-adic valuation of periods of Lubin-Tate groups. In particular, Colmez constructs elements $s_\sigma,s_\tau \in \Bdr$ on which the Galois action is through the Lubin-Tate characters, with the following properties (\cite[Theorem I.2.1]{col}): 
\begin{itemize}
    \item $s_\tau \in \Fil^0 \setminus \Fil^1$, with $v_p\theta(s_{\tau}) = \frac{1}{2(p-1)} + \frac{1}{2}$.
    
    \item $s_\sigma \in \Fil^1 \setminus \Fil^2$, with $v_p\theta((s_\sigma / t)) = -\frac{1}{2(p-1)} - \frac{1}{2}.$
\end{itemize}

One can now check that $D_{\dr}(L \otimes \Q_p)$ is spanned by two elements $A = u_1\otimes \fp s_\sigma + u_2 \otimes s_\sigma$ and $B = u_1 \otimes (-\fp s_\tau) + u_2 \otimes s_\tau$ -- i.e. the two elements $A,B \in L\otimes B_{\dR}$ are indeed Galois invariants.
We are now ready to compare the two norms. 

\subsubsection{The intrinsic norm}
Recall that the intrinsic norm is calculated by embedding the graded pieces (starting with the lowest piece and working upwards) of $ D_{\dR}(L)$ into quotients of $L\otimes \C_p$ and using the natural norm on $L\otimes \C_p$. In this example, the image of $\Gr^0 D_{\dR}(L)$ in $L\otimes \C_p$ is obtained by restricting $1\otimes \theta: L\otimes \Bdr \rightarrow L\otimes \C_p$ to $(L\otimes \Bdr)^{\Gal_K}$. Therefore, it follows that $B$ has intrinsic $p$-adic valuation $\frac{1}{2(p-1)} + \frac{1}{2}$. 

The intrinsic norm of $A$ is calculated as follows. Consider the quotient of $L\otimes \C_p$ by the $\C_p$-span of $(1\otimes \theta) (B)$. We endow this with the quotient norm induced by the norm on $L\otimes \C_p$. The intrinsic norm of $A$ is defined to be the norm of image of $(1\otimes \theta)(\frac{A}{t})$ in $(L\otimes \C_p) / \C_p\cdot (1\otimes \theta)(B) $. Here, $(L\otimes \C_p) / \C_p\cdot (1\otimes \theta)(B) $ has the quotient norm induced by the norm on $L\otimes \C_p$. 

We now claim that $A$ has the same intrinsic norm as the element $1\otimes \theta (u_1 \otimes \fp \frac{s_{\sigma}}{t})$. Indeed, $(1\otimes \theta) (\frac{A}{t}) - \frac{\theta(s_{\sigma}/t)}{\theta(s_{\tau})} \cdot (1\otimes \theta) (B) = u_1\otimes \theta(2\fp s_{\sigma}/t)$.   It now follows that the intrinsic valuation of $A$ is at least $\frac{-1}{2(p-1)}$. In order to show that the intrinsic valuation of $A$ is exactly $\frac{-1}{2(p-1)}$, it suffices to show the valuation of $ 1\otimes \theta (u_1\otimes \frac{2\frak{p}s_{\sigma}}{t})$ is greater than or equal to that of $A_{\alpha} = 1\otimes \theta (u_1\otimes \frac{2\frak{p}s_{\sigma}}{t}) + \alpha \cdot [1\otimes \theta (B)]$ for any $\alpha \in \C_p$. Note that we calculate the valuation in $L\otimes \C_p$ with respect to the lattice $L\otimes \cO_{\C_p}$. Suppose that there exists $\alpha \in \C_p$ such that $A_{\alpha}$ had valuation greater than $\frac{-1}{2(p-1)}$. Such an $\alpha$ must necessarily satisfy $v_p(\alpha\frak{p} \theta(s_{\tau})) = v_p(\frak{p}\theta(\frac{s_{\sigma}}{t}))$. But in that case, the valuation of $A_{\alpha}$ would then equal the valuation of $u_2 \otimes \theta(\alpha t_{\tau})$, which is strictly smaller than $\frac{-1}{2(p-1)}$.

\subsubsection{The crystalline norm}
In order to compute the crystalline norms of $A$ and $B$, we proceed as follows. We have the canonical identification $H^1_{\dR}(E_K) \tilde{\rightarrow} D_{\dR}(L)$. The image of $H^1_{\dR}(E_{\cO_K})$ under this map defines the set of points having crystalline norm $\leq 1$. Colmez's construction of the periods $s_\sigma, s_\tau$ yields that the elements $A,B$ are integral generators of the two subspaces $H^1_{\dR}(E_{\cO_K})^\sigma$ and $H^1_{\dR}(E_{\cO_K})^\tau$. Therefore, it follows that $A$ is an integral generator of $\Fil^1 \subset H^1(E_{\cO_K})$, and so has valuation 0 under the crystalline norm. 
    
As remarked above, the element $B$ is not an integral generator of $\Gr^0H^1_{\dR}(E_{\cO_K})$, and but $\frac{B}{\fp}$ is integral in $\Gr^0H^1_{\dR}(E_{\cO_K})$, and indeed generates it. It follows that $B$ has crystalline valuation $\frac{1}{2}$. Therefore, the crystalline and intrinsic norm agree up to a factor of $p^{\frac{1}{2(p-1)}}$, which is (stronger than) what is claimed by Theorem \ref{crisint}.


\subsection{Elliptic curves with CM by a non-maximal order}
Let $E,K$ be as in Subsection \ref{sec:maximalEC}. There is a unique degree $p$ cyclic extension $K'/K$ such that every order $p$ subgroup of $E_K$ is defined over $K'$. Let $E'$ denote an elliptic curve with CM by the unique order having index $p$ inside the maximal order. There exists a degree-$p$ isogeny $E\rightarrow E'$ defined over $K'$. The $p$-adic comparison isomorphism relating $H_{\dR}^1(E)$ with $D_{\dR}(T_p^\vee(E))$ is compatible with isogenies. Therefore, in order to compare the intrinsic and crystalline norms for $E'$, it suffices to work inside $D_{\dR}(T_p^{\vee}(E))$. 

We let $L$, $s_{\sigma}, s_{\tau}, A,B,$ etc to mean the same objects as in Section \ref{sec:maximalEC}. Setting $L' = T_p^{\vee}(E')$, we may assume without loss of generality that $L'\subset L$ is the $\Z_p$-sublattice spanned by $u_1,pu_2$. 
\subsubsection{Intrinsic norm}
We claim that $B,A$ have intrinsic valuation equalling $-\frac{1}{2}+\frac{1}{2(p-1)}$ and $-\frac{1}{2(p-1)}$ respectively. Indeed, the intrinsic valuation for $A$ equals the intrinsic valuation of the element $u_1 \otimes (-\frak{p}s_{\sigma)})$ as the term $u_2\otimes s_{\tau}$ is rendered irrelevant while computing the intrinsic valuation as in the maximal case. 

\subsubsection*{Crystalline norm}

It now remains to compute the crystalline valuations of $A,B$. As $ E\rightarrow E'$ is defined by a degree $p$ isogeny, it follows that $H^1_{\dR}(E'_{\cO_{K'}}) \subset H^1_{\dR}(E_{\cO_{K'}})$ is a sublattice, whose co-kernel has cardinality the same as $\cO_{K'} / p \cO_{K'}$. Therefore, as $A,B$ forms a basis of $H^1_{\dR}(E'_K)$, the changes in the crystalline valuations of $A$ and $B$ will sum to $-1$,and therefore the crystalline valuations of $A$ and $B$ sum to $-\frac{1}{2}$. Therefore, it suffices to calculate the crystalline norm of $B$. 

We have that $\Fil^1(H^1_{\dR}(E'_{\cO_{K'}})) = \Fil^1(H^1_{\dR}(E_{\cO_{K'}})) \cap H^1_{\dR}(E_{\cO_{K'}})$. Therefore, the change in the crystalline norm of $B$ is just the valuation of $\alpha \in \cO_{K'}$ where $\alpha$ is any element that satisfies $\Fil^1(H^1_{\dR}(E'_{\cO_{K'}})) =\alpha \Fil^1(H^1_{\dR}(E_{\cO_{K'}}))$, equivalently $\Omega^1_{E'/\cO_{K'}} = \alpha \Omega^1_{E/\cO_{K'}}$. 

We calculate this using a global argument as follows. Recall that $E,E'$ have CM by the imaginary quadratic field $M$. Let $H$  denote the Hilbert class field of $M$, and let $H'$ denote the cyclic degree $p$ extension of $H$ which is the field of definition of the order $p$ subgroups of $E$ and therefore of $E'$. First, we suppose that $H = M$, i.e. $\cO_M$ has class number 1. The change in the global Faltings heights of $E,E'$ is well understood. Indeed, we have that $h_F(E') - h_F(E) = \frac{p-1}{2p} \log p$ (see for instance, \cite{lucia}, or \cite{YZ}, or \cite[Proposition 5.1.6]{ShankarTang} and \cite{Autissier}). On the other hand, by the Faltings isogeny theorem, this change in height equals $\frac{1}{2}\log p - \frac{1}{[H':\Q]} \log (|\Omega^1_{E/\cO_{K'}}/\Omega^1_{E'/\cO_{K'}} |) = \frac{1}{2}\log p - \frac{1}{2p} \log (|\Omega^1_{E/\cO_{K'}} / \Omega^1_{E'/\cO_{K'}}|$. This is because the extension $H'/\Q$ is totally ramified at $p$, and so there is only one local place that contributes to the difference in Faltings heights. Comparing these quantities, we see that
$\frac{1}{2p}\log (|\Omega^1_{E/\cO_{K'}} / \Omega^1_{E'/\cO_{K'}}| = \frac{1}{2} \log p - \frac{(p-1)}{2p} \log p$, whence $|\Omega^1_{E/\cO_{K'}} / \Omega^1_{E'/\cO_{K'}}| = p$. It then follows that $v_p(\alpha) = \frac{1}{2p}$, and therefore that the crystalline valuation of $A$ is $ - \frac{1}{2p}$. 

If $M$ does not have class number 1, then the increase in the degree $[H':\Q]$ would be offset exactly by the number of places of $H$ above $p$ (note that the extension $H/M$ is totally split above the unique prime ideal of $M$ dividing $p$, as this ideal is principal), and so we would still obtain that the crystalline valuation of $A$ is $-\frac{1}{2p}$. 

It therefore follows that the crystalline valuation of $B$ must be $\frac{1}{2} - 1 + \frac{1}{2p} = -\frac{1}{2} + \frac{1}{2p}$. 

Therefore, the two norms are indeed seen to be comparable by the above explicit calculation.

\section{Shimura Varieties}\label{S-SV}

See \cite{milne-intro} for further background on this section.

\subsection{Basic definitions}

Let $G$ be a connected reductive group over $\Q$, $\Ss=\Res_{\C/\R}\G_m $ the Deligne torus, $X$ be a $G(\R)$-conjugacy class of 
homomorphisms $f:\Ss\ra G_\R$ satisfying the Shimura axioms in \cite{milne-models}. Note in particular that this includes the axiom that the central torus $Z(G)^o$ is split over a CM field. We call $(G,X)$ a \textit{Shimura datum}

Finally, let $K=\prod_p K_p\subset G(\A_f)$ be a split-neat compact subgroup. Let $E(G,X)\subset \C$ be the reflex field, and $E$ be a field over which $G$ splits. Associated to this data we get 
\begin{itemize}
    \item A complex projective variety $\check{X}_{\C}$ with a transitive action of $G$, such that $X$ injects into $\check{X}_{\C}(\C)$ in a $G(\R)$-equivariant way.
    \item An algebraic variety $S_K(G,X)$ over $E(G,X)$ called a \textit{Shimura variety} whose complex points can be identified with $G(\Q)\backslash\left(X\times G(\A_f)\right)/K$.
    \item A canonical model of $\check{X}$ over $E(G,X)$ satisfying certain natural properties.
\end{itemize}

Henceforth we write $\check{X}$ and $S_K(G,X)$ to mean the canonical models, which are algebraic varieties defined over $E(G,X)$. \medskip

If $G$ is a simple adjoint group over $\Q$, then $G=\Res_{F/\Q}G'$ for some totally real field $F$ and a geometrically simple adjoint group $G'$ \cite[Thm 3.13]{milne-moduli}. Thus $\displaystyle G_\R\cong\prod_{\sigma:F\ra \R} G'_\sigma$
where $G'_{\sigma} = G'\times_{F,\sigma} \R$. Let $I_{nc}$ be the set of real places of $\R$ such that $G'_{\sigma}(\R)$ is non-compact for each $\sigma \in I_{nc}$. We then have a corresponding splitting of hermitian symmetric domains $\displaystyle X\cong \prod_{\sigma \in I_{nc}} X_\sigma$. 

\subsection{Special points}

Following \cite{milne-moduli} we say that a point $x\in X$ is \textit{special} if the corresponding $h:\Ss\ra G$ factors through a $\Q$-torus $T$. We define $T_h$ to be the smallest such $\Q$-torus.

We say that a point of $S_K(G,X)$ is Special if some (and hence, any) representative $(x,g)\subset X\times G(\bbA_f)$ has the property that $x$ is special.

\begin{lemma}\label{lem:specialptCMsplit}
Given a special point $h\in X$, the torus $T_h$ splits over a CM field. Consequently, $h$ lies in a $0$-dimensional Shimura sub-datum $(T,h)$. 
\end{lemma}

\begin{proof}

If $G$ is adjoint, then this follows from \cite[A3]{milne-aut}. Therefore, the torus $T_{h^{\ad}} \subset G^{\ad}$ splits over a CM field. But $T_{h^{\ad}}$ is the quotient of $T_h$ by a subtorus of $Z(G)$, which also splits over a CM field. The claim follows.
    
\end{proof}

As a result, we see that every special point $h\in X$ is contained in a $0$-dimensional Shimura datum $(T_h,h)\subset (G,X)$.

\subsection{Useful group-theoretic reductions for Shimura varieties}
We will record two results that will be used at multiple points in the future. 
\begin{proposition}\label{prop:shimuracover}
    Let $(G,X)$ be a Shimura datum. Then there is another Shimura datum $(G',X')$ and a morphism $f: (G',X')\rightarrow (G,X)$ satisfying the following properties. 
    \begin{enumerate}
        \item $G'\rightarrow G$ is surjective. 
        \item $f$ is an isomorphism at the level of adjoint Shimura data.
        \item $G'^{\der}$ is simply connected.
    \end{enumerate}
\end{proposition}

\begin{proof}

We first assume that $G^{\der}$ is adjoint. In this case, $G\cong G^{\der} \times  Z(G)$. Let $G^s\ra G^{\der}$ be the simply connected cover, and set $G_1:=G^s\times Z(G)$ Now let $\mu:\Ss\ra G$ be a cocharacter in $X$ corresponding to a special point $h$, so that $\mu$ factors through the $\Q$-torus $T_h$. By Lemma \ref{lem:specialptCMsplit} $T_h$ splits over a CM field. Let $T'_h\subset G_1$ denote the connected component of the inverse image of $T_h$. Finally, let $K:=T'_h\cap \left(Z(G^s)\times\{1\}\right)$.
We define $G':= (G_1\times T'_h)/K_\Delta$, where $K_\Delta$ is the diagonal embedding of $K$. 

Note that there is a diagonal embedding of $T'_h$ into $G_1$ which induces an embedding $\phi:T_h\hookrightarrow G'$. We define the conjugacy class of the composition $\Ss\ra T_h\ra G'$ to be $X_1$. Since $T_h$ is split over a CM field the Shimura axioms are immediate, and so it remains to show that $G'^{\der}$ is simply connected. But $K_\Delta\cap \left(G_1\times \{1\}\right) = \{1\}$, and so $G_1$ injects into $G'$, and thus $G^s\cong G'^\der$, as desired.

Finally we handle the general case. Suppose now that $G$ is arbitrary. Consider $G_0:=G/Z(G^{\der})$ so that $G_0^{\der}$ is adjoint. Thus by what we proved above there is a Shimura variety $(G_2,X_2)\ra (G_0,X_0)$ inducing an isomorphism on the adjoint groups. We define $G'$ to be the connected component of $G\times_{G_0}G_2$, and likewise $X_1$ to be $X\times _{X_0} X_2$. Since $G'$ surjects onto $G_2$ its derived group must be simply connected, completing the proof.
    
\end{proof}

\begin{proposition}\label{prop: realrankatleast2}
    Let $(G,X)$ be a Shimura datum such that $G = G^c$. Then, there exists another Shimura datum $(H,Y)$ that $(G,X)$ embeds in, such that $H = H^{c}$ and such that every $\Q$-simple factor of $H^{\der}$ has real rank at least 2. Further, we may choose $H^{\der}$ to be $\Q$-simple if $G^{\der}$ is.
\end{proposition}
\begin{proof}
    Let $F$ denote a real quadratic field that is linearly disjoint from some finite Galois extension that splits $G$. Then, $H_1 = \Res_{F/\Q}G_F$ is a reductive group such that every $\Q$-simple factor of $H^{\prime, \der}$ has real rank at least 2. The center $Z_{H_1}$ of $H_1$ contains the center $Z_G$ as a subtorus. We have the canonical surjective morphism that we call $\det: H_1 \rightarrow H_1/H_1^{\der}$. This map $\det|_{Z_{H_1}}$ is an isogeny. Let $Z'_G = \det(Z_G)$. Set $H = \det^{-1}(Z'_G)$. This group has the same derived group as $H_1$, and satisfies $H^{ c} = H$, as $Z_G$ and $Z_H$ are isogenous tori. The group $G$ embeds in $H$ by construction, and $(G,X)$ induces a Shimura data on $H$ in which it embeds as a sub-Shimura datum. By construction, $H^{\der}$ is $\Q$-simple if $G^{\der}$ is. The result follows.

\end{proof}

\subsection{Principle bundles}

For an algebraic group $G$ defined over $\Q$, let $Z(G)^o$ denote its central torus. Assuming $Z(G)$ splits over a CM field - which is true for all our Shimura varieties - we may define $Z(G)_s< Z(G)^o$ to be the maximal anisotropic $\Q$-subtorus which splits over $\R$. We define $G^c:=G/Z(G)_s$.

 There is an algebraic variety $P=P_K(G,X)$ over $E(G,X)$ such that 
 $P(\C)\cong G(\Q)\backslash \left(X\times G(\C)\times G(\A_f)\right)/K$
 and maps $$S_K(G,X)\xleftarrow[]{\pi} P\xrightarrow[]{\psi} \check{X}$$ which over $\C$ are naturally identified with the maps $[(x,g,k)]\ra [(x,k)]$ and $[(x,g,k)]\ra g^{-1}x$. Moreover, there is a left action of $G$ on $P$ such that $g'\circ [(x,g,k)] = [(x,gg'^{-1},k])$ and $\psi([x,gg',k])=g'\psi([x,g,k])$.
 
 $P$ has a natural $G^c_{E(G,X)}$-action under which $\psi$ is $G^c_{E(G,X)}$-equivariant, and $P$ is a $G^c_{E(G,X)}$-torsor under $\pi$.

\subsection{Automorphic vector bundles}\label{secthreeautomorphicvb}

A $G^c$-vector bundle $V$ on $\check{X}$ is a vector bundle equipped with an equivariant $G^c$ action over $\check{X}$. 
We may pull back $V$ to obtain a $G^c$-equivariant vector bundle $\psi^*V$ on $P$. This may then in turn be descended to obtain a vector bundle W on $S_K(G,X)$. We call this the \emph{automorphic vector bundle}.
If $V$ is defined over a field $F\supset E(G,X)$ then $W$ has a canonical model over $F$ as well.

\subsection{Rational representations}\label{sec:rationalrepresentations}

Let $\rho$ be a representation of $G^c$ acting on a  vector space $V$ defined over $\Q$. This in turn induces a rational local system $_BV_{\Q}$ on $S_K(G,X)(\C)$ as well as \'etale local systems $_{\et}V_\ell$ for each finite prime $\ell$. There are canonical isomorphisms $_{\et}V_{\ell} \rightarrow _BV_{\Q}\otimes \Q_{\ell}$. The \'etale local systems $_{\et}V_\ell$ descend to the reflex field (see for example \cite[III, Remark 6.1]{milne-models}). 

Fixing a $K$-stable lattice $\V\subset V$ yields a $\Z$-local system on $S_K(G,X)(\C)$, which in turn (using the canonical isomorphism $_{\et}V_{\ell} \rightarrow _BV_{\Q}\otimes \Q_{\ell}$) yields integral models $_{\et}\V_{\ell}\subset \ _{\et}V_{\ell}$. The local systems $_{\et}\V_\ell$ also canonically descend to the reflex field. 

\begin{proposition}\label{prop: milnepullback}
Suppose $f: (G,X) \rightarrow (H,Y)$ is a morphism of Shimura data, and let $f$ also denote the morphism of Shimura varieties thus induced. Suppose $(\rho, V)$ is a rational representation of $H^c$. Let $f^*V$ denote representation\footnote{Which factors through $G^c$.} of $G$ induced by $f:G\rightarrow H$. Then, $_{B}(f^*V)_{\Q}$ and $_{\et}(f^*V)_\ell$ are canonically isomorphic to $f^*(_{B}V_\Q)$ and $f^*(_{\et}V_\ell)$ respectively. 
\end{proposition}
\begin{proof}
The statement about Betti local systems is immediate. The statement for the \'etale local systems follow from the following two facts. Firstly, that the Betti-\'etale isomorphism is canonical, and secondly, the descent of the \'etale local system to the reflex field is also canonical (see \cite[Remark 6.1]{milne-models} and the preceding discussion). 
\end{proof}


 The Riemann-Hilbert correspondence yields a vector bundle with connection $_{\dR}V$ on $S_K(G,X)$, which has a filtration $\Fil^{\bullet}$ - induced pointwise by the action of $\Ss$ - with each filtered piece being an automorphic vector bundle. Further, the data of $_{\dR}V,\Fil^{\bullet}$ descends to the reflex field. 
 Moreover, by the work of Diao-Lan-Liu-Zhu\cite[\S4.1]{LZ2}, for each finite place $v\mid p$ of $E(G,X)$ we obtain algebraic vector bundles with connection $_{p-\dR}V_{E_v}$ equipped with a flat connection and filtration. By \cite[Thm 5.3.1]{LZ2} these are naturally isomorphic in a way which is compatible with morphisms of Shimura data.


\subsection{Canonical metrics at infinity}\label{canhodge}

Let $S_K(G,X)$ be a Shimura variety. Suppose that $\rho:G\ra GL(V)$ is an irreducible rational representation of $G$. We pick a lattice $\V$ preserved by $K$. It follows by \cite[II,Prop 3.3]{milne-moduli}that $\rho\circ f_x$ is a Hodge structure $V_x$ on $V$ for any $x\in X$. By \cite[Prop 3.2]{milne-models} it follows that
$V$ admits a $G$-equivariant bilinear form $\psi:V\times V\ra\Q_\psi$ for a rational character $\psi$ of $G$, which becomes a polarization on $V_x$ for all $x\in X$. Corresponding to such a polarization, we obtain a canonical metric on $_{\dR}V_{\C}$ on $S_K(G,X)$ by taking the Hodge norm (see \cite[6.6']{Schmidt}), and this is compatible with pullbacks under maps of Shimura varieties.

Now, following \cite[II.4.2]{milne-models} for any element $\tau\in\Gal(\Qbar/\Q)$ one may associate to $\tau S_K(G,X)$ the structure of a Shimura variety $S_{^\tau K}(^{\tau}G,^{\tau}X)$ in a manner unique up to a canonical isomorphism. Following \cite[III.6.2]{milne-models} one obtains 
canonical representation $V_{\tau}$ of $^{\tau}G$ in a way compatible with the passage to the de-Rham vector bundles or the $\ell$-adic local systems, such that $_{\dR}V_{\tau}\cong \tau_{\dR}V$. Moreover, conjugating $\psi$ gives us a polarization $\psi_\tau$.  Thus we obtain canonical Hodge metrics on $_{\dR} V_{\C_{\tau}}$ as well.


\subsection{Automorphic line bundles on tori}

We shall require the following lemma, which tells us that for tori, there is essentially only one way given an automorphic line bundle to find it as a sub-bundle of the associated vector bundle corresopnding to a local system. To see the relation recall that for tori, automorphic bundles
are in bijection with complex representations of the torus, whereas local systems are in bijection with rational representations.

\begin{lemma}\label{lem:torcharorb}
The irreducible representations of a Torus defined over $k$ are in bijection with $G_k$-orbits on its character lattice $X^*(k^{sep})$.  
\end{lemma}
\begin{proof}
This is \cite[Theorem 14.22]{milne-alg}
\end{proof}

\subsection{Partial CM types}\label{S-SV-partialCMtypes}

Let $K/F$ be a quadratic CM extension of a totally real field. We define a torus $R_K:=\Res_{K/\Q}\G_m / \Res_{F/\Q}\G_m$. Note that the cocharacter lattice $X_*(R_K)$ is isomorphic
to $\Z\langle e_{\sigma},\ \sigma:K\ra \C\rangle /(e_{\sigma}+e_{\ol\sigma})$ as a Galois module. Similarly, the character lattice $X^*(R_K)\subset X^*(\Res_{K/\Q} \G_m)$ as a Galois module is isomorphic to $\Z \langle e^*_\sigma - e^*_{\overline{\sigma}}\rangle$.

We define a partial CM type \cite[4.2]{Harris} to be a subset $\Phi$ of $\Hom(K,\C)$ such that
$\Phi\cap\ol\Phi=\emptyset$. We associate a homomorphism
$r_\Phi:\Ss\ra R_{K,\R}$ as follows: We first note that
$$R_{K}(\C) \cong \prod_{\sigma\in\Hom(K,\C)} \C_{\sigma}^\times/\sim$$ where $\sim$ identifies $\C_{\sigma}$ and $\C_{\ol\sigma}$ via inversion. Thus, identifying $\Ss(\C)\cong \C^\times\times \C^{\times}$, the map from the first factor to $\C^{\times}_\sigma$
is $$\begin{cases} z\ra z & \sigma\in\Phi\\ z\ra 1 & \sigma\in\ol\Phi\\ z\ra 1& \textrm{ else} \end{cases}.$$
The map from the second factor is determined by the conjugation action.

The data $(R_K,r_{\Phi})$ defines a Shimura datum, which we shall show to be sufficient to discuss CM points in adjoint Shimura varieties.
We finally define the (complex) character $\chi_\Phi$ of $R_{K,\C}$ to be 
$$\chi_\phi:=\sum_{\sigma\in\Phi}e^*_\sigma - e^*_{\overline{\sigma}}.$$

\subsection{CM points in adjoint Shimura varieties}

The purpose of this subsection is to prove the following:

\begin{theorem}\label{pcmreduction}
Let $(G,X)$ be a Shimura datum with $G$ an adjoint group over $\Q$, and let $(T_x,h_x)$ denote a $0$-dimensional Shimura sub-datum such that $T_x$ splits over a CM field $K$. Then there are $B$ maps $a_i:(T_x,h_x)\ra (R_K,r_{\Phi_i})$ to Shimura data corresponding to partial CM types, such that 
$$(\det F^1_x\Lie G\mid T_x)^N = \ds\prod_{i=1}^B(\chi_{\Phi_i}\circ a_i)^{c_i}$$ where $N,B,c_i$ are positive integers bounded in terms of $\dim G$.

\end{theorem}

The rest of this section is devoted to the proof of 
Theorem \ref{pcmreduction}. 
\vspace{0.1in}

\textit{Step 1: Constructing the maps of tori}
\vspace{0.1in}

As in the beginning of the section, we may write
$G_\R\cong\prod_{j\in J}G_j$ where each $G_i$ is absolutely simple, and the Galois action on $J$ factors through a totally real field $F$ \cite[Thm 3.13]{milne-moduli}. There is then a corresponding splitting of hermitian symmetric domains 
$X\cong \prod_{j\in J_{nc}}X_j$ where $J_{nc}$ denotes the elements $j$ where $G_j(\R)$ is non-compact.

 Let $K$ be the splitting field of $T_x$. Even though $T_x$ needn't be a maximal torus of $G$, we will still use the term `root' to refer to a character of $T_x$ that occurs in the adjoint representation of $G$ restricted to $T_x$. For each compact factor $G_j$ the map $\pi_j\circ f_x$ is trivial, whereas for each noncompact factor we denote the corresponding co-root by $\mu_j:=\pi_j\circ f_x(z,1) $. Let $\mu=(\mu_j)_j$ be the corresponding cocharacter of $T_x$. Then $\mu$ is special, in the sense that $\langle \mu,\beta\rangle\in\{-1,0,1\}$ for any root $\beta$ of $T_x$\cite[p.20]{milne-intro}.

The Galois group $\Gal_\Q$ acts on cocharacters of $T_x$ and the action on factors through a Galois CM field $K_x$ which contains $K$, with complex conjugation acting as negation \cite[12.4.e]{milne-intro}. Let $F_x$ be the totally real subfield of $K_x$.  Let $\tau_0$ be a complex place of $K_x$.

Let $\beta$ be a root of $T_x$ occurring in $\mathrm{Lie}\ G$. Since
$\beta$ is fixed by $\Gal(\Qbar/K_x)$ and acted on by complex conjugation as negation, there is a Galois-equivariant map $X^*(R_{K_x}) \ra X^*(T_x)$ sending $e_{\tau_0}^* - e_{\overline{\tau_0}}^*$ to $\beta$. Thus, there is a map
$f_{\beta}:T_x\ra R_{K_x}$
such that $(e_{\tau_0}^*-e_{\overline{\tau_0}}^*)\circ f_\beta=\beta$. The maps $f_\beta$ will be our $a_i$ above.

\vspace{0.1in}
{\it Step 2: Establishing the Map of Shimura varieties}
\vspace{0.1in}

We now check that the maps $f_\beta$ induce maps on Shimura varieties from $(T_x,h_x)$ to $(R_K,r_{\Phi_\beta})$ for an appropriately chosen partial CM type $\Phi_\beta$.

For $g \in \Gal_{\Q}$, we have $$(e^*_{g\tau_0} - e^*_{g\overline{\tau}_0})\circ f_{\beta}=g\beta.$$ Since $\mu_i$ is special for each $i$ we have that $\langle \mu,g\beta\rangle \in\{-1,0,1\}$. Let $\Phi_\beta$ denote the embeddings $g\tau_0$ of $K_x$ for which the inner product is $1$. Then evidently $\Phi_\beta$ is a partial CM-type. 

Moreover, $r_{\Phi_\beta}$ is characterized by the property that $$\forall\sigma\in\Hom(K,\C), (e_\sigma^*-e^*_{\bar\sigma})\circ r_{\Phi_\beta}(z,1)=\begin{cases}z & \sigma\in\Phi_\beta\\ z^{-1} & \sigma\in\ol\Phi_\beta\\ 1& \textrm{ else} \end{cases},$$ and
$$\forall g\in\Gal_{\Q}, (e^*_{g\tau_0} - e^*_{g\overline{\tau}_0})\circ f_{\beta}\circ h_x(z,1)=\left((g\beta)\circ \mu\right)(z) = \begin{cases}z & g\tau_0\in\Phi_\beta\\ z^{-1} & g\tau_0\in\ol\Phi_\beta\\ 1& \textrm{ else} \end{cases}$$

We thus conclude that $ f_{\beta}\circ h_x = r_{\Phi_\beta}$ as desired.

\vspace{0.1in}
{\it Step 2: Establishing the identity.}
\vspace{0.1in}

We now prove the main identity. Note that by definition
$$f_{\beta}^*\chi_{\Phi_\beta}=\ds\sum_{g\in\Gal(K_x/\Q)\mid \langle \mu, g\beta\rangle=1} g\beta.$$ We define $n_\beta$ to be the order of the stabilizer of $\beta$ in $\Gal(K_x/\Q)$.

The filtration on $\Lie_G$ at $x$ is determine by the co-character $\mu$, and thus $F^1_x\Lie_G$ consists of the span of the $\beta$-eigenspaces for all roots $\beta$ satisfying $\langle \mu,\beta\rangle =1$. Let $m_\beta$ be the multiplicity of that eigenspace. Then 

$$\det F^1_{\dR}\Lie G\mid T_x = \sum_{\beta\mid\langle \mu,\beta\rangle=1}m_\beta \beta$$ where $m_\beta$ is the multiplicity with which $\beta$ occurs. Let $B$ denote the number of $\Gal(K_x/\Q)$-orbits of roots occurring in $\Lie G$, counting with multiplicity and $\beta_1,\dots,\beta_B$ a representative from each.

Finally, we let $N:=\ds\prod_{i=1}^B n_{\beta_i},a_i=f_{\beta_i},c_i=N\cdot \frac{m_{\beta_i}}{n_{\beta_i}}$. It is clear that all these numbers are bounded in terms of $\dim G$.

\subsection{Motives associated to special points}

\begin{definition}
Following Milne\cite{milne-cm} we define a rational Hodge structure to be be a finite dimensional $\Q$-vector space $W$, together with a map $h:\Ss\ra \GL(W_{\R})$ such that the corresponding weight map $h^{-1}\mid \G_m$ is defined over $\Q$. 
\end{definition}

Let $(T,h)$ denote a Shimura datum with $T=T^c$ a torus. Let $K = \prod_{p}K_p \subset T(\A_f)$ denote a neat compact open subgroup. Let $V$ denote a $\Q$-representation of $T$, and let $\V\subset V$ denote a lattice stabilized by $K$. Recall from the discussion in Section \ref{sec:rationalrepresentations} the \'etale local systems $_{\et}\V_\ell$ on $S_K(T,h)$ defined over the reflex field of $(T,h)$. Note also that the map $h:\Ss\ra T_{\R}$ makes $V$ a rational hodge structure.

\begin{theorem}\label{specialpointisabelian}
The rational Hodge structure associated to $V$ is in the Tannakian category generated by the rational Hodge structures corresponding to complex CM abelian varieties. 
\end{theorem}
\begin{proof}
This is the content of \cite[Proposition 4.6]{milne-cm}.
\end{proof}

Let $x\in S_K(T,h)$ denote some point, and let $E$ denote the field of definition of $x$. The fiber of the local system $_{\et}\V_p$ at $x$ yields a Galois representation $\rho_p: \Gal_E\rightarrow \GL(\V_p)$. The following result asserts that the representations $\rho_p$ are potentially crystalline for large enough primes $p$.

\begin{proposition}\label{specialiscrystalline}
Let the setup be as above, and let $a$ be the lowest Hodge weight occurring in $V$.  
Then, there exists a finite extension $E'/E$ such that for $p\gg 1$, $\rho_p(-a)|_{\Gal_{E'}}$ is crystalline at all places of $E'$ dividing $p$. 

\end{proposition}

Note that here we use crystalline in the same sense as \S2, so that the weights all lie in $[0,p-2]$.

\begin{proof}

Without loss of generality, we assume that there is no subtorus $T''\subset T$ defined over $\mathbb{Q}$ whose base-change to $\R$ contains the image of $h$ (otherwise, we may just replace $T$ by $T''$). We already have that the rational Hodge structure associated to $V$ is in the Tannakian category generated by some CM complex abelian variety $A$. Said more explicitly, $V$ is a direct summand  of some tensor power $W:=H^1(A(\C),\Q)^{\otimes m}\otimes (H^1(A(\C),\Q)^{\vee})^{\otimes n}$ in the category of rational Hodge structures. Let $\W \subset W$ denote the lattice induced by $H^1(A(\C),\Z) \subset H^1(A(\C),\Q)$.

Let $(T',r')$ be the Shimura variety corresponding to $H^1(A)$. Then $W$ is naturally a representation of $T'$, and the projection map $W\ra V$ naturally induces a map $r'(\Ss)\ra h(\Ss)$ and taking $\Q$-zariski closures a surjection $T'\ra T$. Thus we have a map of Shimura varieties $f:(T',r')\ra (T,h)$ such that $f^*V$ is a direct summand of $W$. By Proposition \ref{prop: milnepullback}, the associated Galois representations don't change, and thus we may reduce to proving the claim for $(T,h)=(T,r')$ and $V=f^*V$. 

We have that $V$ is a direct summand of $W$, and hence $\V\otimes \Z_p$ is a direct summand of $\W\otimes \Z_p$ for $p$ large enough. The discussion in Section \ref{sec:rationalrepresentations} (specifically, the canonical isomorphism between the Betti and \'etale realizations) implies that $_{\et}\V_p$ is a direct summand of $_{\et}\W_p$, and therefore it suffices to prove the claim for $W$. By Lemma \ref{cristensor}, tensor products of crystalline representations are crystalline as long as the weights stay within $[0,p-2]$. We are thus reduced to proving the claim for $W=H^1(A)$. 

Let $E'/E$ denote a number field over which $A$ is defined, and has CM (and therefore has good reduction everywhere). The crystallinity over $E'$ for large primes now follows from the fact that $A$ has good reduction modulo primes above $p$.

\end{proof}


\section{$p$-adic Local Systems on Shimura Varieties are Crystalline}\label{S-crystalline}
In this section, we use work of Esnault-Groechenig to prove results pertaining to the crystallinity of $p$-adic local systems on Shimura varieties. Let $(G,X)$ be a Shimura datum, with reflex field $E(G,X)$.
We also fix a neat compact open subgroup $K\subset G(\mathbb{A}_f)$. For any positive integer $m$, it is well known that there are only finitely many rigid flat vector bundles of rank $\leq m$ on $\ShimK_\C$, and these must all be defined over some number field $E'$ which we may assume contains $E(G,X)$. We fix a log-smooth compactification $\ol{\ShimK}$ of $S_K(G,X)$ defined over $E(G,X)$, which spreads out to such a compactification $\ol{\integralShimK}$ of $\integralShimK$ away from finitely many places $v$. This allows us to discuss the notion of log-crystalline local systems on $\ShimK$ at almost all places. Throughout this section, the term log-crystalline will always be relative to the log smooth compactification $\ol{\integralShimK}$. 



    

    


\begin{theorem}\label{crystallinelocalsystems}
Let $(\rho,V)$ denote a degree $n$ rational representation of $G$ that factors through $G^c$, and let $\V$ be a lattice stable under $K$. Recall that $_{\et}\V_p$ denotes the corresponding $\Z_p$-local system on $S_K(G,X)$. There exists a positive integer $N'$ such that the following holds: Let $p\not\mid N'$ denote a prime, and let $v$ be a place of $E'$ dividing $p$. Then $_{\et}\V_p/\ShimK_{E_v}$ is a log-crystalline local system up to cyclotomic twist.

\end{theorem}

\subsection{Rigidity and crystalline descent}
We remind the reader that local systems called crystalline in the appendix are called log-crystalline in this section. As a first step, we use Esnault-Groechenig's Theorem \ref{thm:betti} in conjunction with Margulis super-rigidity to show that a self-direct sum of $_{\et}\V_p|_{\ShimK_{\C}}$ admits a log-crystalline descent to $\ShimK_{W(\F_q)}$, where $q$ is a power of a large prime $p$. More precisely, let $_{\et}\V_{p^f}$ denote $_{\et}\V_p \otimes W(\F_{p^f})$. 

\begin{corollary}[Esnault-Groechenig, Margulis]\label{EsnaultGroechenig}
    Let $(\rho,V)$ be as in Theorem \ref{crystallinelocalsystems}, and $_{\dR} V_{\rho}/\integralShimK$ denote the vector bundle with connection associated to $\rho$. There exist integers $n$ and $N$ such that the following holds:
 
For every prime $p\nmid N$, and a prime $v$ of $E(G,X)$ above $p$, there exists a positive integer $f$ such that $_{\et}\V^{\oplus n} _{p^f}\mid_{\ShimK_{\C}} $ admits a descent to $\ShimK_{W(\F_{q}[1/p])}$ that is log-crystalline. Here, $q $ is a power of $p$. 
\end{corollary}

\begin{proof}
We can and will replace $G$ by $G^c$. 
We first deduce the corollary under the assumption that $G^{\der}$ is a $\Q$-simple group. By Proposition \ref{prop: realrankatleast2}, there exists an embedding of Shimura data $(G,X)\rightarrow (H,Y)$ such that $H^c = H$, $H^{\der}$ is $\Q$-simple, and has real rank at least 2. Let $W$ be a faithful representation of $H$. By applying Margulis super-rigidity (\cite[Thm.IX.6.15.ii]{margulis}), we deduce that every local system on the Shimura variety associated to $(H,Y)$ is strongly cohomologically rigid, and we may therefore apply \ref{thm:betti} to $S_K(H,Y)$ to deduce that $_{\et}W_{S_K(H,Y)_{\C}}$ admits a log-crystalline descent. We now restrict to $\ShimK$ to deduce that $_{\et}W|_{\ShimK_{\C}}$ admits a log-crystalline descent. By replacing $W$ by some tensor power, we may assume that $V$ is a sub $G$-representation of $W$. As any sub-local system of a log-crystalline local system is log-crystalline, any sub local system of $_{\et}W|_{\ShimK_{\C}}$ that descends to $\ShimK_{W(\F_q)[1/p]}$ also admits a log-crystalline descent. The isotypic component of $V$ in $W$ induces such a local system. This concludes the case when $G^{\der}$ is $\Q$-simple. 


We will now deduce the corollary for arbitrary Shimura varieties. Let $(G',X')\rightarrow (G,X)$ be as in Proposition \ref{prop:shimuracover}. The induced map $S_{K'}(G',X') \rightarrow S(G,X)$ is a finite \'etale map. Replacing $N$ by a larger integer if necessary, we may assume that the map is finite \'etale at the level of integral models. By the remark immediately following \cite[Theorem 2.6]{FaltingsCrystallineCohomology}, it suffices to check that $_{\et}\V_p^{\oplus n}$ pulled back to $S_{K'}(G',X')$ admits a descent that is log-crystalline. This reduces the corollary to the case where $G^{\der}$ is simply connected, and we shall henceforth assume that this is the case.

    We may therefore write $G^{\der} = G_1 \times G_2 \hdots G_k$, where each $G_i$ is $\Q$-simple. Define $G^{(i)} = G/\prod_{i\neq j} G_j$. The Shimura datum $(G,X)$ induces canonical Shimura data $(G^{(i)},X_i)$ for $1\leq i \leq k$, and there is a natural inclusion of Shimura data $(G,X)\rightarrow \prod_i (G^{(i)},X_i)$. Note that the induced map $\ShimK \rightarrow \prod_i S_{K_i}(G^{(i)},X_i)$ is a finite \'etale map, as the morphism of Shimura data is an isomorphism at the level of adjoint Shimura data. We have already dealt with the case of $\Q$-simple groups, and therefore there exist faithful representations $V_i$ of $G^{(i)}$ such that the associated local systems on $S_{K_i}(G^{(i)},X_i)$ admit  crystalline descents. There exists an integer $m$ such that $V$ is a direct summand of $(V_1 \boxtimes V_2 \hdots \boxtimes V_k)^{\otimes {m}}|_G =(V_1^{\otimes m} \boxtimes \hdots \boxtimes V_k^{\otimes m})|_G $. By replacing $V_i$ by $V_i^{\otimes m}$, we assume that $m = 1$. The representations $V_i$ are fixed independent of $p$, and so by replacing $N$ by a larger integer if necessary, we may assume that the difference between the maximum and minimum Hodge weights of $_{\dr}(V_1 \boxtimes V_2 \hdots \boxtimes V_k)/S(\prod G^{(i)},X_i)$ are smaller than $p-1$. So, $\boxtimes_i {_\et\V_{i,p^f}}\mid_{\ShimK_{\C}}$ admits a log-crystalline descent to $\ShimK_{W(\F_q)[1/p]}$. 
    As any direct summand of a log-crystalline local system is log-crystalline, any sub local system of $(\boxtimes_i {_\et}\V_{i,p^f})\mid_{\ShimK_{\C}}$ that descends to $\ShimK_{W(\F_{q})}$ is also log-crystalline. The isotypic component of $V$ in $(V_1 \boxtimes V_2 \hdots \boxtimes V_k)|_G$ induces such a local system. We note in passing that the multiplicity $n$ of $V$ in this isotypic component only depends on the initial data of $G$ and $V$, and not on the prime $p$. The corollary now follows. 
\end{proof}


\subsection{Crystallinity of $p$-adic local systems}

\begin{proof}[Proof of Theorem \ref{crystallinelocalsystems}] The log-crystallinity of $_{\et}\V_p|_{\ShimK}$ would follow from that of $_{\et}\V^{\oplus n}_{p}$ as it is a sub local-system. Therefore, by letting $V$ denote $V^{\oplus n}$ with $n$ as in Corollary \ref{EsnaultGroechenig}, we may assume (by Corollary \ref{EsnaultGroechenig}) that $_{\et}\V_{p^f}|_{\ShimK_{\C}}$ admits a descent to $\ShimK_{W(\F_q)[1/p]}$ which is log-crystalline. We call this local system $\W$.

Consider the $p$-adic local system $\bL / \ShimK$ given by $\pi^*\pi_*\Hom_{W(\F_{p^f})}(\W^{\vee}, _\et \V_{p^f}^{\vee})$, where $\pi: \ShimK \rightarrow \Spec W(\F_q)[1/p]$ is the structure map. By adjointness, we have a map of arithmetic local systems $\bL \rightarrow \Hom_{W(\F_{p^f})}(\W^{\vee}, _\et \V_{p^f}^{\vee})$, and therefore $_\et \V_{p^f} \rightarrow \bL^{\vee}\otimes_{W(\F_{p^f})} \W$. This map is injective, since $_{\et}\V_{p^f}$ and $\W$ are isomorphic on $\ShimK_{\C_p}$. 

Recall that if a local system is crystalline up to cyclotomic twist, so is its dual (by \cite[Theorem 2.6* h]{FaltingsCrystallineCohomology}). We now claim that $\bL$, and therefore $\bL^{\vee}$, are crystalline up to cyclotomic twist. Let $x\in\ShimK(W(\F_{q^m}))$ be a special point (such a point exists as we have assumed that $p$ is a large prime). It suffices to check that $\bL_x$ is crystalline as $\bL$ is pulled back from a point. The Galois representation $\W_x$ is crystalline, and therefore  $\W^{\vee}_x$ is crystalline up to cyclotomic twist. Moreover, $_{\et}\V_{p^f,x}$ is crystalline up to cyclotomic twist (by Lemma \ref{specialiscrystalline}), therefore so is $_{\et}\V^{\vee}_{p^f,x}$. For sufficiently large $p$ the Hodge weights of $\Hom(\W^{\vee}, _\et \V_{p^f}^{\vee})$ are in the Fontaine-Laffaile range up to cyclotomic twist. Therefore, we have that $\bL$ and $\bL^{\vee}$ are crystalline up to cyclotomic twist by Lemma \ref{cristensor}. We may also assume that the Hodge weights of $\bL^{\vee}\otimes \W$ are in the Fontaine-Laffaile range up to twist, and therefore $\bL^{\vee}\otimes \W$ is log-crystalline up to cyclotomic twist by Lemma \ref{lem:logcrystensor}. Finally, as $_{\et}\V_{p^f}$ is a sub-local system of a local system that is log-crystalline up to cyclotomic twist, we have that it must be log-crystalline up to cyclotomic twist. The theorem follows as $_{\et}\V_p$ is a $\Z_p$-local subsystem of $\V_{p^f}$.

\end{proof}

\section{Good Reduction of CM Points}\label{S-GR}

The purpose of this section is to show that all CM points in a Shimura variety are integral with respect to some fixed model. Ideally, we would show this for all primes, even the finitely many `bad ones'. However, we have been unable to show this, settling instead for handling almost all primes.

\vspace{0.1in}
\emph{Question: Given a Shimura variety, is there an integral model of it over some $\cO_F$ with respect to which all CM points are integral?}
\vspace{0.1in}

As such, our main theorem is as follows:

\begin{theorem}\label{CMgoodreduction}

Let $S=S_K(G,X)$ be a Shimura variety defined over a number field $F$. For some positive integer $N$, there exists an integral model $\cS$ over $\cO_F[N^{-1}]$ such that every CM point of $S(\Qbar)$ extends to a point of $\cS(\ol{\Z}[N^{-1}])$. 

\end{theorem}

\subsection{Punctured formal neighborhoods}

\begin{lemma}\label{puncture}
Let $R$ be a regular local ring of mixed characteristic $p$, and consider\newline
$S=R[[x_1,\dots,x_n,y_1,\dots,y_m]]\left[\frac1y_1,\dots\frac1y_n\right]$. Let $R_0, S_0$ denote the maximal prime-to-$p$ Galois \'etale extension of $R,S$ respectively. Then $S_0$ is generated by $R_0$ and the prime-to-$p$ roots of the $y_i$.

\end{lemma}

\begin{proof}
Without loss of generality we may assume $R=R_0$ by base-changing. Let $W$ be a Galois \'etale extension of $S$ of degree prime-to-$p$. Let $T$ denote the normal closure of $R[[x_1,\dots,x_n,y_1,\dots,y_m]]$ in $W$. Now for each $j\in\{1,\hdots,m\}$ we let $R_j$ and $ T_j$ denote the localization of $R[[x_1,\dots,x_n,y_1,\dots,y_m]]$ and $ T$ at the prime ideals $(y_j)$ and $\mathfrak{y}_j$ respectively, where $\mathfrak{y}_j$ is some prime ideal of $T_j$ sitting above $(y_j)$. Now $T_j,R_j$ are discrete valuation rings. Let $e_j$ denote the ramification degree and let $e=\prod_j e_j$. Now let $W'$ denote the compositum of $W$ and the $e$'th roots of all the $y_i$, and let $T', T'_j$ be as before. Then by Abhyankar's lemma \cite[\href{https://stacks.math.columbia.edu/tag/0BRM}{Tag 0BRM}]{stacks-project}, $T'_j$ is unramified over $R[[x_i,\dots,x_n,y_1^{\frac1{e}},\dots,y_m^{\frac1{e}}]]_j$ for all $j$. By the purity of the branch locus \cite[\href{https://stacks.math.columbia.edu/tag/0BMB}{Tag 0BMB}]{stacks-project} it follows that $T'$ is unramified over $R[[x_1,\dots,x_n,y_1^{\frac1e},\dots,y_m^{\frac1e}]$. Finally, \'etale covers of $R$ correspond bijectively to \'etale covers of $R[[x_1,\dots,x_n,y_1^{\frac1e},\dots,y_m^{\frac1e}]]$ and thus $T'=R[[x_1,\dots,x_n,y_1^{\frac1e},\dots,y_m^{\frac1e}]]$. The claim is thus proven.
\end{proof}

\subsection{Specialization homomorphisms and setup}

Let $S=S_K(G,X)$ be a Shimura variety, and let $\ol S$ be a log-smooth compactification of $S$ with $D=\ol S-S$. By blowing up further, we may and do ensure that the irreducible components of $D$ are smooth, i.e. they have no self-intersections. By spreading out, we may form a model
$(\ol\cS,\cS,\cD)$ over $\cO_E[\frac1N]$ for some large even integer $N$ such that $\ol\cS$ is proper smooth and
$\cD$ is a normal crossings divisor of $\ol\cS$ over $\cO_E[\frac1N]$. 

We fix a faithful $G_\Q$-representation $V$ with a lattice $\V$ invariant under $K$. By shrinking\footnote{Note that Theorem \ref{CMgoodreduction} is invariant under finite covers.} $K$ we may assume that the action of $K_2$ on $\V_2 = \V\otimes \Z_2$ is trivial mod $4$. We fix a CM point $x\in S(\Qbar)$ and
by possibly enlarging $N$ assume that $x$ is induced by an integral point of $\cS$, unramified outside of primes dividing $N$.

The following lemma is essentially \cite[Prop 3.1]{EG2}: 

\begin{lemma}\label{localsystemextend}

For all primes $p\not\mid N$, the local system $_{\et}\V_2$ extends to $\cS_{W(\ol{\bF}_p)}$. 

\end{lemma}

\begin{proof}
For a profinite group $H$, denote by $H'$ denote the prime-to-$p$ quotient of $H$. By \cite[Thm A.7]{LO} we have that
 $$\pi'_1(\cS_{W(\ol{\bF}_p)})\cong\pi'_1(S_{\ol{\Q}_p}).$$  

 We claim we have a canonical direct product decomposition

\begin{equation}\label{pi1splitting}
 \pi'_1(S_{\Q_p^{\ur}}) \cong \pi'_1(S_{\ol{\Q}_p})\times \pi'_1(\Q_p^{\ur}).
 \end{equation}

 Indeed, we have the natural maps 
 $$\pi'_1(S_{\ol{\Q}_p}) \xrightarrow{f_1} \pi'_1(S_{\Q_p^{\ur}}) \xrightarrow{f_2} \pi'_1(\cS_{W(\ol{\bF}_p)}).$$ Since the composition is an isomorphism, it follows that $f_1$ is injective and $f_2$ determines a canonical left-inverse $\eta$ to $f_1$. Thus the natural right-exact sequence 

 $$  \pi'_1(S_{\ol{\Q}_p}) \ra \pi'_1(S_{\Q_p^{\ur}})\ra \pi'_1(\Q_p^{\ur})\ra 1$$ is in actually exact and admits a natural splitting given by $\eta$, proving the claim.

 
 
 Finally, via this decomposition, we see that a prime-to-$p$ local system on $S_{\Q_p^{\ur}}$ extends to $\cS_{W(\ol{\bF}_p)}$ if and only if the induced action of $\pi'_1(\Q_p^{\ur})$ is trivial, which can be checked on any $\Q_p^{\ur}$ point of $S$ which extends to a $\Z_p^{\ur}$ point of $\cS$. Indeed, this follows from $\Z_p^{\ur}$ having no \'etale extensions. 
 
Now, by assumption, $x$ is a point of $S$ which extends to $\cS$, so it is enough to check that $_{\et}\V_{2,x}$ is unramified at $p$. However, since $x$ is a CM point and $2\neq p$ we see that the image of inertia on  $_{\et}\V_{2,\ol{x}}$ is finite, and therefore is trivial since the monodromy of $_{\et}\V_2$ is contained in $K_2$ which is torsion-free by construction. This completes the proof.
 
\end{proof}

The following lemma completes the proof of Theorem \ref{CMgoodreduction}.

\begin{lemma}\label{lem:CMsmoothreduction}

With the above notation, for every place $v\not\mid N$, every CM point of $S_v$ extends to a point of $\cS_v$.

\end{lemma}

\begin{proof}

Let $y$ be such a CM point with field of definition $E'\supset E$ and consider its extension $y_0$ to $\ol\cS(\ol E_v)=\ol\cS(\ol{\cO_{E_v}})$ corresponding to some place $w\mid v$ of $E'$. Let $M=E_w^{\prime \ur}$.

Suppose for the sake of contradiction that $y_{0,\ol\bF_p}\in D$. Then we may pick a regular system of parameters such that $\widehat{\cO}_{S,y_{0,\ol\bF_p}}\cong \cO_{E,v}[[t_1,\dots,t_n]][[s_1,\dots,s_m]]$ with $D$ cut out by the $t_i$. Since $y_0$ reduces to $D$ we see that the $t_i$ pull back to elements of the maximal ideal of $\cO_M$. 

This induces a map
$$f:\cO_M[[t_1,\dots,t_n,s_1,\dots,s_m]]\left[\frac1{t_i},\dots\frac1{t_n}\right]\ra M.$$ 

Now by lemma \ref{puncture} $f$ induces a map of prime-to-$p$ Galois Groups
$G^{(p)}_M\ra \prod_{\ell\neq p} \Z^n_\ell(1)$. Moreover, by the same lemma the image of this map is completely determined by the (positive) valuations of $f(t_i)$. 

Now, by choosing an embedding $\iota:\cO_M\ra\C$ we obtain an induced map
$$\cO_M[[t_1,\dots,t_n,s_1,\dots,s_m]][\frac1{t_i},\dots\frac1{t_n}] \ra \C[[t_1,\dots,t_n,s_1,\dots,s_m]]\left[\frac1{t_i},\dots\frac1{t_n}\right]$$ inducing an isomorphism of prime-to-$p$ Galois groups. Moreover, we may consider a map $$F:\C[[t_1,\dots,t_n,s_1,\dots,s_m]]\left[\frac1{t_i},\dots\frac1{t_n}\right]\ra \C[[t]]\left[\frac1t\right]$$ defined by $F(s_i)=0,F(t_i)=t^{v_M(f(t_i))}$ which thus has the same fundamental group image as $f$.

Since $_{\et}\V_{2,y_M}$ is trivial, it follows that $F^*_{\et}\V_2$ is also trivial. Now, since the $t_i,s_j$ are regular parameters, the map $F$ is the completion of an holomorphic map $F':\Delta^*\ra S(\C)$ where $t$ is identified with a coordinate of $\Delta$ vanishing at the origin, and since the (profinite completion of the) topological fundamental group of $\Delta^*$ is naturally identified with the \'etale fundamental group of $\Spec \C[[t]]\left[\frac1t\right]$, it follows that $F'^*_{\et}{}_BL_\C$ is also trivial. However, $_{B}V_\C$ underlies a variation of Hodge structures with maximal domain of definition $S_\C$, and thus $F^*_{\et}{}_BV_\C$ must have infinite monodromy around $0$ \cite{Schmid}. This is our desired contradiction.

\end{proof}

\section{Canonical Norm on Local Systems of  Shimura Varieties}\label{S-normShimura}

Let $S=S_K(G,X)$ be a Shimura variety, and let $\ol{S}$ be a log-smooth compactification. We assume as always that $K$ is neat. Let $\rho$ be a rational representation of $G$ on a polarized vector space $(V,q)$ that factors through $G^c$, and let $\V\subset V$ denote a lattice. By Theorem \ref{crystallinenormoncrystallineuptotwist} for large $p$ we have that $_{\et}\V_p$ is log-crystalline up to cyclotomic twist. Let $(\cS,\cV)$ be a smooth integral model of $(S, {}_{\dR}V)$ over $\cO_E[N^{-1}]$ as in Theorem \ref{CMgoodreduction}. We assign an admissible norm to $(\Gr{_{\dR}V,\ol S,S)}$ as follows: 

\begin{itemize}
    \item For each Archimedean place, we simply use the norm on the graded pieces induced by the Hodge norm as described in section \ref{canhodge}.
    \item For each non-Archimedean place $v$ (dividing $p\in \mathbb{Z}$) at which $_{\et}\V_p$ is crystalline and which doesn't divide $N$ from Theorem \ref{CMgoodreduction}, and with $p > \rng_V\cdot \dim V + 2$ : we identify $_{\dR}V_{E_v}$  with $_{p-\dR}V_{E_v}$ via \cite[Thm 5.3.1]{LZ2}, and use the crystalline norm on the graded pieces of $_{p-\dR}V_{E_v,x}\cong D_{HT}(_{\et}V_{E_v,x})$. 
    \item For the (finitely many) other non-Archimedean placed $v$, we identify $_{\dR}V_{E_v}$  with $_{p-\dR}V_{E_v}$ via \cite[Thm 5.3.1]{LZ2}, and use the intrinsic norm on the graded pieces of $_{p-\dR}V_{E_v,x}\cong D_{HT}(_{\et}V_{E_v,x})$.
\end{itemize}

\begin{theorem}\label{shhtadmis}
With the notation above, the normed vector bundle $\Gr _{\dR}V$ is admissible. 
\end{theorem}

\begin{proof}

For every finite place $v$, the intrinsic norm on the graded pieces of $D_{HT}(_{\et}V_{E_v,x})$ extends to an acceptable norm on $\ol{S}_v$ .

The fact that the norms at infinity are continuous is immediate, and the acceptability of the norms at each finite place follows from Theorem \ref{crisint} and Lemma \ref{lem:intrinsicacceptable}. 

The Theorem is then reduced to the proof of the following proposition:

\begin{proposition}\label{integralcomp}
For almost all places $v$, the image of the integral model $S_{\cris}(_{\et}\V_{E_v})$ in $_{\dR}V_{E_v}$ under the canonical isomorphism agrees with $\cV_v$. 
\end{proposition}

\begin{proof}

We first reduce to the case where $G^{\der}$ is $\Q$-simple and $G=G^c$. By Proposition \ref{prop:shimuracover}
we may find a cover $f:(G',X')\ra (G,X)$ such that $G'^{\der}$ is simply-connected. Since the map on Shimura varieties is finite \'etale, is is enough to check the proposition on $G'$ after pullback. We therefore assume that $G^{\der}$ is simply connected.

We write $G^{\der} = G_1 \times G_2 \hdots G_k$, where each $G_i$ is $\Q$-simple. Define $G^{(i)} = G/\prod_{i\neq j} G_j$. The Shimura datum $(G,X)$ induces canonical Shimura data $(G^{(i)},X_i)$ for $1\leq i \leq k$, and there is a natural inclusion of Shimura data $(G,X)\rightarrow \prod_i (G^{(i)},X^i)$. Note that the induced map $(G,X)\rightarrow \prod_i (G^{(i)},X_i)$ induces a finite \'etale map of Shimura varieties, as the morphism of Shimura data is an isomorphism at the level of adjoint Shimura data. There exist $\Q$-representations $V_i$ of $G^{(i)}$ such that $V$ is a direct summand of $(V_1 \boxtimes V_2 \hdots \boxtimes V_k)|_G$. Since this is true rationally, it follows for almost all places that the same is true for $\V_v$. Thus the claim for $V$ follows from the claim for each of the $V_i$. We have therefore reduced to the case where $G^{\der}$ is $\Q$-simple. 

Finally, we easily reduce to the case where $G=G^c$, noting that our local systems by definition comes form a rational representation that factors through $G^c$, therefore all our associated data is canonically pulled back along the map $(G,X)\ra (G^c,X^c)$. 

Next, we will now reduce to the case where $G^{\der}$ has real rank at least 2, so that we may apply the results of the appendix. First, we note that the Tannakian argument above reduces to proving the result for some faithful representation of $G$. By Proposition \ref{prop: realrankatleast2}, there exists an embedding of Shimura data $(G,X) \rightarrow (H,Y)$ where $H^{\der}$ is $\Q$-simple and has real rank at least 2, and $H = H^c$. Any faithful representation of $H$ induces a faithful representation of $G$, and so we have reduced to the case when the real rank of the derived group is at least 2.  

Henceforth we assume that $G^{\der}$ is $\Q$-simple, has real rank at least 2, and that $G = G^c$ in what follows. We first note that both $S_{\cris}(_{\et}\V_{E_v})$ and $\cV_v$ are equipped with flat log-connections. We claim that they are \emph{abstractly} isomorphic after an unramified based change base. Indeed, $V$ splits as a direct sum of representations after some finite extension of $\Q$, and hence $\V_v$ does as well for almost all $v$. For $v$ sufficiently large this field extension is unramified at $v$.
Hence, after such a base change, both  $S_{\cris}(_{\et}\V_{E_v})$ and $\cV_v$ are direct sums of integral models of irreducible flat log-bundles on the generic fiber. The claim now follows by Proposition \ref{prop:model}, which applies because every local system on $\ShimK$ is strongly cohomologically rigid. 

It follows that the image of two models in $_{\der}V$ are the same up to a global automorphism of the generic fiber $_{\dR}V_{E_v}\otimes \Q_p$, which corresponds to a global automorphism of the associated $\Q_p$-local system. We may thus check that these two structures agree at a single point. By the crystalline compatibility on passage to fibers in Theorem \ref{fibercomp} we are thus reduced to checking the agreement for a 0-dimensional Shimura datum $(T,r)$ that is a sub of $(G,X)$. We pick $T$ as the smallest $\Q$-subtorus containing the image of $r$.

Since the $\R$-split center of $G=G^c$ is $\Q$-split, it follows that the weight character of $(G,X)$ is rational, and hence the same is true for $(T,r)$. Thus, by \cite[Prop 4.1]{milne-cm} it follows that $(T,r)$ arises from a rational Hodge structure $V$  of CM type. Moreover, by \cite[Prop 4.7]{milne-cm} there is a CM abelian variety $A$ such that $V$ is in the Tannakian category generated by the rational Hodge structure  $H^1(A)$. It follows that $(T,r)$ is a quotient of the Shimura variety $(T',r')$ corresponding to $H^1(A)$. By Proposition \ref{prop: milnepullback}, it is sufficient to prove the proposition for the Shimura variety $(T',r')$.

Now $(T',r')$ admits a faithful representation on $\V_0:=H^1(A)$. The canonical isomorphism between $_{p-\dR} \V_0$ and $_{\dR} \V_0$ constructed in \cite{LZ2} is just the usual $p$-adic comparison isomorphism constructed by Faltings (note that the former object is $D_{\dR}(_{\et}\V_{0,p}[1/p]) = D_{\dR}(H^1_{\et}(A,\Q_p))$ and the latter object is canonically $H_{\dR}^1(A)$). Further, for this particular representation, our statement follows from the integral version of this theorem \cite[Thm 5.3]{FaltingsCrystallineCohomology}. By \cite[Theorem 5.3.1]{LZ2}, the canonical isomorphisms between the $p$-adic and usual Rieman Hilbert correspondences are compatible under Tannakian operations. As $\V\otimes\Q$ is in the Tannakian category generated by $\V_0\otimes\Q$, it is a direct summand of some tensor power of $\V_0\otimes\Q$. This remains true integrally for almost all places. Furthermore, The functor $S_{\cris}$ is also compatible under tannakian operations as long as $p$ is large enough relative to the Hodge-Tate weights of $\V$ (Lemma \ref{cristensor}), and therefore the claim follows.

\end{proof}

\begin{corollary}\label{shhtweil}
For any integer $a$, let $\cL=\det\Gr^a_{F^\cdot}{}_{\dR}V$. Let $h$ denote the height corresponding to $\cL$ with its normed vector bundle structure as above. Let  $h_\cL$ denote a Weil height corresponding to $\cL$ considered as a line bundle on $\ol S$ via the Deligne extension. The height $h$ differs from $h_\cL$ by at most $O(\max(1,\log h_A))$ where $A$ is an ample line bundle on $\ol S$. In particular, if $\cL$ is ample then $h$ is comparable to $h_\cL$. 
\end{corollary}

\begin{proof}

The key is that the Hodge metric has logarithmic singularities by \cite[Thm 6.6']{Schmid}. 

By Theorem \ref{shhtadmis} and corollary \ref{cor:ourheightsareweil} the difference $h-h_{\cL}$ is $O(1)$ plus the difference at the Archimedean places coming from the fact that our normed line bundle is not \emph{strongly} admissible. By Schmid's work\cite[Thm 6.6']{Schmid} this is bounded by the logarithm of a power of the absolute value of logarithm of the distance to the boundary $D=\ol S-S$. Since the absolute value of the logarithm of the distance to the boundary is bounded above by $h_A$, the result follows.

\end{proof}

\end{proof}

\begin{definition}
We keep the setup of this section, and suppose that the weights of $_{\dR}V$ are in $[a,b]$, such that $\Gr^b _{\dR}V$ is 1-dimensional. Call the height constructed using the above norms on $\Gr^b _{\dR}V$ the \emph{canonical height}. 
\end{definition}

\section{Solid Height Functions}\label{S-solidheightfunctions}

\subsection{Notation}

Given two functions $f,g$ we write $f\prec_{\vec{x}} g$ ($g\succ_{\vec{x}}  f$) if there exist functions $A,B>0$ depending only on $\vec{x}$ such that $f\leq Ag^B$. Given a torus $T$, we define $E_T$ to be the splitting field of $T$, and we define $K_T\subset T(\bbA_f)$ to be the maximal compact subgroup.

\subsection{Height function requirements}
Following previous notation, we make the following definition:

\begin{definition}
Let $V$ denote a representation of $\G_m$. Define $\rng_V$ to be the difference between the highest and lowest weight of $\G_m$ on $V$.
\end{definition}

\begin{definition}\label{def:solidcharacter}
    Let $(T,r)$ be a Shimura datum with $T$ a torus, and let $\chi$ be a character of $T^c$. We say that $\chi$ is \textit{solid} if there exists an irreducible $\Q$-representation $V_\chi$ of $T^c$, such that under the filtration $\Fil^\cdot V_{\chi,\C}$ induced by $r$,  we have that the highest weight piece $\Fil^a V_{\chi,\C}$ is 1-dimensional and is isomorphic to $\chi$. By Lemma \ref{lem:torcharorb} the representation $V_\chi$ is unique and we define $\rng_\chi$ to be $\rng_{V_\chi}$.
\end{definition}

Observe that if $\chi_1,\chi_2$ are both solid characters then so is $\chi_1\chi_2$ and $\rng_{\chi_1\chi_2}\leq \rng_{\chi_1}+\rng_{\chi_2}$. 

\begin{lemma}
    Let $f:(T_1,r_1)\ra (T_2,r_2)$ be a map of Shimura vaieties with $T_1,T_2$ Tori, and let $\chi$ be a solid character of $T_2$. Then $\chi\circ f$ is a solid character of $T_1$ of the same weight.
 \end{lemma}

\begin{proof}
    Let $V_\chi$ be the corresponding representation of $T_2$, and consider $V_\chi$ as a representation of $T_1$ via $f$. The filtration induced by $r_1$ is the same as that induced by $r_2$ since $r_2\circ f=r_1$ and thus the highest weight piece $\Fil^a V_{\chi,\C}$ is also the heighest weight piece for $T_1$. The claim follows.
\end{proof}

\begin{definition}\label{def:solidtriple}
    We consider \textit{Toric Triple} to be a triple $(T,r,\chi)$ where 

\begin{enumerate}
    \item $T$ is a Torus
    \item $(T,r)$ is a Shimura variety
    \item $\chi$ is a solid character of $T^c$.
\end{enumerate}
\end{definition}

The example that we will be most concerned with is the setting of CM fields and partial CM types. We have the following lemma: 
\begin{lemma}\label{lem:partialCMtypesolid}
Let $E$ be a CM field with totally real subfield $F\subset E$. Let $\Phi$ be a partial CM type on $E$ with $|\Phi| = g_1$. Then $(R_E,r_\Phi,\chi_\Phi)$ is a toric triple.     
\end{lemma}
\begin{proof}
    We need to produce a representation of $R_E$ whose highest weight (with respect to the cocharacter $\rho_\Phi$) is one-dimensional and the highest weight is $\chi_\Phi$. We have the map $h: \Res_{E/\Q} \G_m \rightarrow \Res_{E/\Q}\G_m$ given by $\alpha\mapsto \frac{\alpha}{\bar{\alpha}}$. This map factors through $R_K$, and therefore we obtain a homomorphism $f: R_K \rightarrow \Res_{E/\Q} \G_m $. We have that $V' =  \Res_{E/\Q}\G_a$ is a representation of $R_K$ via $f$. Base-changing to $\C$, we see that the weight spaces of $r_\Phi$ are $V'_1 = \bigoplus_{\sigma \in \Phi} \C_{\sigma}$, $V'_{-1} = \bigoplus_{\sigma \in \bar{\Phi}}$, and $V'_0 = \bigoplus_{\sigma \notin \Phi \cup \bar{\Phi}}$. The characters of $R_E$ on $V'_1$ are precisely $e^*_{\sigma} - e^*_{\bar{\sigma}}$, for each $\sigma \in \Phi$. Further, each such $e^*_{\sigma} - e^*_{\bar{\sigma}}$ appears with multiplicity one. We may thus define $V_{\chi_{\Phi}}$ to be the irreducible subspace of $\bigwedge^{g_1} V$ containing the line $\bigwedge^{g_1}V'_1$.

\end{proof}

\begin{definition}\label{def:solidheight}

    A \textit{solid height} is a real-valued $h$ function on the set of solid triples, such that:

    \begin{enumerate}
        \item If $\chi_1,\chi_2$ are both solid characters then
        $$h(T,r,\chi_1\chi_2)=h(T,r,\chi_1)+h(T,r,\chi_2)+O_{\rng_{\chi_1}+\rng_{\chi_2}+\dim T}(\log\Disc E_T)$$

        \item Given\footnote{As we explain in the next section, property (2) follows from (1) and (3) formally, but we still decided to include it for clarity as it is used in the proofs.} a solid triple $(T,r,\chi)$ and a positive integer $m$ we have 
        $$h(T,r^m,\chi)=mh(T,r,\chi) + O_{m+\rng_\chi+\dim T}(\log\Disc E_T)$$

        \item If $(T_2,r_2,\chi)$ is solid, and $f:(T_1,r_1)\ra (T_2,r_2)$ is a morphisms of Shimura data with $T_1,T_2$ Tori, then 
        $$h(T_1,r_1,\chi\circ f)=h(T_2,r_2,\chi) + O_{\rng_\chi+\dim T_1+\dim T_2}(\log \Disc E_{T_1}+\log \Disc E_{T_2}).$$

        \item Let $S=S_K(G,X)$ be a Shimura variety, and $V$ an irreducible representation of $G$ with a sublattice $\V$ fixed by $K$. Assume that the highest weight piece $L:=\Fil^a _{\dR}V_\C$ is 1-dimensional, and let $h_L$ be a Weil height on a Toroidal compactification $\bar S$ of $S$ corresponding to the Deligne extension of $L$. Let $h_A$ be a Weil height corresponding to any ample bundle on a Toroidal compactification of $S$. Finally, let $(T,r)\subset (G,X)$ be a 0-dimensional Shimura subdatum, and $\chi:=\Fil_r^a V_\C$ the corresponding Solid character of $T$. Then for all points $x\in S_K(G,X)$ in the image of $S_{K\cap T(\bbA_f)}(T,r)$, we have
        $$|h(T,r,\chi) - h_L(x)| = O_S\Big(\log\Disc E_T + \log\big([K_T:K\cap T(\bbA_f)]\big) + \log^+ h_A(x)\Big)$$
    \end{enumerate}

\end{definition}

We shall prove the following theorem:

\begin{theorem}\label{thm:solidimpliesAO}
    Suppose a solid height function exists. Then the Andr\'e-oort conjecture holds.
\end{theorem}

\subsection{Properties of heights of total CM types}
In this section, we will prove the following proposition. 
\begin{proposition}\label{lem:boundsolidheightoftotalCMtype}
   Let $E$ be a CM field of degree $2g$ with totally real field $F$, and let $\Phi$ be a complete CM type on $E$. Then, $h(R_E,r_\Phi,\chi_\Phi) = O_g(\Disc E^{o_g(1)})$.
\end{proposition}
\begin{proof}
By properties (1) and (3) we may assume $\Phi$ is a primitive CM type.

Consider the norm map $$\Nm_{E/F:}\Res_{E/\Q}\G_m\ra \Res_{F/\Q}\G_m,$$ and define $S_E:=\Nm^{-1}_{E/\Q}(\G_m).$ Now
$\Res_{E/Q}\G_m(\R)\cong \bigoplus_{\sigma\in\Hom(F,\C)}\C_{\sigma}$.  
There is a natural morphism $\Ss\ra (\Res_{E/Q}\G_m)_\R$ which on the $\sigma$ co-ordinate induces the isomorphism corresponding to the element of $\Phi$ above $\sigma$, and this morphism factors through $S_E$.  We denote by $s_\phi$ the corresponding map $\Ss\ra S_{E,\R}$.

There is a natural quotient map $i:(S_E,s_\Phi)\ra (R_E,r_\Phi)$, so by property (3), it suffices to prove that $h(S_E, s_\Phi, i^*(\chi_{\Phi})) = O_g(\Disc E^{o_g(1)})$ 

Let $V':=\Res_{E/\Q}\G_a$ denote the standard representation of $S_E$, and let $\Nm_{E/\Q}$ denote the Norm character, which is solid as its a rational character. The highest weight space (with respect to $s_{\Phi}$) in $\bigwedge^g V'$ is one-dimensional. Let $\chi_{\std}$ denote the the character of this weight space. A direct computation shows that $\chi^2_{\std}=i^*\chi_{\Phi} \cdot \Nm$. By property (1), it suffices to prove that $$\max\big(h(S_E,s_{\Phi},\Nm_{E/\Q}),h(S_E,s_\Phi,\chi_{\std})\big) = O_g((\Disc E)^{o_g(1)}).$$ 

Now $\Nm_{E/\Q}: (S_E,s_\Phi)\ra (\G_m,\Nm_{E/\Q})$ is a Shimura morphism, and if we let $\textbf{1}$ denote the identity ccharacter of $\G_m$ then $\textbf{1}\circ \Nm_{E/\Q} = \Nm_{E/\Q}$, and so by property (3) we have that $h(S_E,s_{\Phi},\Nm_{E/\Q}) = O_g(\log\Disc E).$

The representation $V'$ of $S_E$ is equipped with a canonical (isomorphism class of) symplectic form, and $S_E$ acts on $\std$ by symplectic similitudes. By Zarhin's trick, the representation $\std^{\oplus 8}$ contains a full rank lattice $L$ such that:
\begin{itemize}
    \item $L$ is self-dual for the symplectic form on $\std^{\oplus 8}$.
    \item $L\otimes \hat{\Z}$ is stable by the maximal compact subgroup of $S_E(\A^{f})$ (indeed, by the maximal compact subgroup of $\Res_{E/\Q} \G_m$). 
\end{itemize}
This induces an embedding of the Shimura variety $S_{K_{S_E}}(S_E,s_{\Phi})$ into $\mathcal{A}_{8g}$. By construction, the highest weight on $\bigwedge^{8g} (V')^{\oplus 8}$ respect to $s_{\Phi}$ is $\chi^8_{\std}$. By Properties (1) and (4) it suffices to prove that for any fixed Weil height $h_A$ on $\mathcal{A}_{8g}$ and any $x\in\mathcal{A}_{8g}$ in the image of $(S_E,s_\Phi)$ that 
$h_A(x) = O_g(\Disc E^{o_g(1)}).$

Any Weil height $h_A$ induced by the highest weight piece of $\bigwedge^{8g} \std^{\oplus 8}$ on all of $\mathcal{A}_{8g}$ satisfies  $h_A = h_{\textrm{Fal}} + O_g(\log h_A) $ by \cite[Prop 4.4-4.5]{chaifaltings}. Therefore, it suffices to prove that $h_{\textrm{Fal}}(x) = O_g(\Disc E^{o_g(1)})$. 

By construction of the self-dual lattice in $\std^{\oplus 8}$, the abelian variety underlying $x$ has the form $B = A^4 \times (A^{\vee})^4$, where $A$ is a $g$-dimensional abelian variety with CM by $\cO_E$ and having CM type $\Phi$.

We have that $h_{\textrm{Fal}}(B) = 8 h_{\textrm{Fal}}(A)$. The claim is now exactly \cite[Cor 3.3]{T}, which uses the average form of Colmez's conjecture (proved in \cite{AGHM} and \cite{YZ}) to bound the heights of such CM abelian varieties. 

The completes the proof of this proposition.

\end{proof}

\subsection{Main height bound}

\begin{theorem}\label{thm:mainhtbound}
Assume a solid height function exists. Fix a Shimura variety $S=S_K(G,X)$ with $G$ of adjoint type, and $K$ neat of split-adjoint type. Let $(T,r)\subset (G,X)$ be a 0-dimensional Shimura sub-datum, and let $L$ be an ample bundle on $S$, with a corresponding Weil height function $h_L$. Then for any point $x\in S_K(G,X)(\Qbar)$ in the image of $S_{K\cap T(\bbA_f)}(T,r)$, we have

$$h_L(x) = O_S\left(\bigg(\Disc E_T\cdot [K_T:K\cap T(\bbA_f)]\bigg)^{o_S(1)}\right).$$

\end{theorem}

\begin{proof}
Since all Weil heights associated to ample bundles  are comparable, it is sufficient to prove the theorem for any choice of  Weil height.

\vspace{0.1in}
Let $V$ denote the adjoint representation of $G$ so that $\Fil^{-1} _{\dR}V = _{\dR}V$ and $\Fil^2_{\dR}V = 0$, and fix a lattice $\V$ stabilized by $K$. Then $_{\dR}V$ is naturally equipped with an admissible collection of norms via the canonical norm. Then $L:=\det\Fil^1_{\dR}V$ is an ample bundle on $S$. By property (4) of Definition \ref{def:solidtriple}, it is sufficient to show that 
$$h(T,r,\chi) = O_S\left(\bigg(\Disc E_T\cdot [K_T:K\cap T(\bbA_f)]\bigg)^{o_S(1)}\right)$$

where $\chi$ is the character $\det\Fil_r^1 V_\C\mid T$. 

By Theorem \ref{pcmreduction} there there are $B$ maps $a_i:(T_x,h_x)\ra (R_{E_T},r_{\Phi_i})$ to Shimura data corresponding to partial CM types, such that 
$$\chi^N = \prod_i (\chi_{\Phi_i}\circ a_i)^{c_i}$$ where $N,B,c_i$ are positive integers bounded in terms of $\dim G$.

By properties (2) and (3) of Definition \ref{def:solidtriple} we see that 

$$h(T,r,\chi)^N= \ds\sum_{i=1}^N c_i h(R_K,r_{\Phi_i},\chi_{\Phi_i}) + O_S(\log \Disc E_T).$$ The result therefore follows from the Theorem below.

\end{proof}

\begin{theorem}\label{delignesidea}
Let $E/F$ be a CM field of degree $2d$,  $\Psi$ a partial CM type of $E$. Then
$h(R_E,r_\Psi,\chi_\Psi)=O_d(\Disc E^{o_d(1)}).$
\end{theorem}

The proof of this will heavily use an idea of Deligne\cite[Prop. 2.3.10]{D} for combining partial CM types to get a complete CM type for a larger CM fields, which he used to classify Shimura varieties of abelian type.

\begin{proof}\textrm{ }

\vspace{0.2in}
\noindent\emph{Step 1: The full CM type}
\vspace{0.2in}

This is Proposition \ref{lem:boundsolidheightoftotalCMtype}.

\vspace{0.2in}
\noindent\emph{Step 2: Reduction to $E/\Q$ a Galois extension}
\vspace{0.2in}

We may reduce to the case where $E$ is Galois as follows: Let $E'$ denotes the Galois closure of $E$, set $m=[E':E]$, and define $\Psi'$ to be the pullback of $\Psi$ to a partial CM type of $E'$.  Then the Norm map $\Nm_{E'/E}:R_{E'}\ra R_E$  gives a map of Shimura Data
$(R_{E'},r_{\Psi'})\ra (R_E,r_{\Psi}^m)$, and satisfies $\Nm_{E'/E}\circ \chi_\Psi = \chi_{\Psi'}$. The reduction now follows from properties (2) and (3) of Definition \ref{def:solidtriple}.

\vspace{0.2in}
\noindent\emph{Step 3: $|\Phi|=1$}
\vspace{0.2in}

We next handle the case where $\Phi$ consists of a single place. In that case, note that the isomorphism class of the toric triple $(R_E,r_{\Phi},\chi_{\Phi})$ is independent of $\Phi$, as $\Gal(E/\Q)$ acts transitively on $\Hom(E,\C)$. Therefore we refer to $h(R_E,r_{\Phi},\chi_{\Phi})$ simply by $h_1(E)$.

For $i=1,\dots,d$ let $E_i=F(\sqrt{p_i})$ for distinct primes $p_1,\dots,p_d$. One may pick these primes of size $O_d(1)$ such that $E,E_1,\dots,E_d$ are disjoint extensions of $F$. We set $E=E_0$. 

Next, fix $0\leq j\leq d$.  We choose $\Phi_1,\dots,\Phi_d$ to be partial CM types of $E_0,\dots,\hat{E_j},\dots,E_d$ consisting of a single place, such that all of these places lie above a distinct place of $F$. Let $E^{(j)}$ denote the compositum
$E^{(j)}=E_0\dots \hat{E_j}\dots E_d$, and define $\Phi^{(j)}_i$ to be the pullback of $\Phi_i$ to $E^{(j)}$. Then $\Phi^{(j)}:=\displaystyle\bigcup_{\substack{i=0\\i\neq j}}^d\Phi_i'$ is a full CM type on $E'$.

For $i\neq j$, the norm map $\Nm_{E'/E_i}$ induces a map of Shimura data $(R_{E^{(j)}},\rho_{\Phi^{(j)}})\rightarrow (R_{E_i},\rho_{\Phi_i}^{2^{d-1}})$. Moreover, $\chi_{\Phi^{(j)}}=\displaystyle\prod_{i=0}^d \chi_{\Phi_i}\circ\Nm_{E^{(j)}/E_i}$. It follows from properties (1) and (3) of definition \ref{def:solidtriple}  that

$$ h(R_{E^{(j)}}, r_{\Phi^{(j)}},\chi_{\Phi^{(j)}})=2^{d-1}\displaystyle\sum_{\substack{i=0\\i\neq j}}^dh_1(E_i) + O_{\dim E}(\log\Disc E).$$

and therefore that
$$h_1(E_0) = \frac{1}{d2^{d-1}}\cdot \Bigg(\displaystyle\sum_{j=1}^d h(R_{E^{(j)}}, r_{\Phi^{(j)}},\chi_{\Phi^{(j)}}) - (d-1)h(R_{E^{(0)}}, r_{\Phi^{(0)}},\chi_{\Phi^{(0)}})\Bigg)  + O_{\dim E}(\log\Disc E).$$

The bound now follows Step 1.

\vspace{0.2in}
\noindent\emph{Step 4: The general case}
\vspace{0.2in}

Finally, we handle the general case. Let $E,F,\Phi$ be arbitrary and we set $e=d-|\Phi|$. For $i=1,\dots,e$ we pick $E_i=F(\sqrt{p_i})$ for primes numbers $p_i=O_{[E:\Q]}(1)$ so that $E,E_1,\dots,E_d$ are disjoint over $F$. We set $E=E_0$.

For $i=1,\dots,e$ we pick singleton sets $\Phi_i$ consisting of one place of $E_i$ such that each place of $F$ is either below an element of $\Phi$, or below the element of $\Phi_i$ for $1\leq i\leq e$. Now we set $E_{tot}$ to be the compositum $E_{tot}:=E_0\dots E_e$ and $\Phi_{tot}$ the union of the pullbacks of the $\Phi_i$ and $\Phi$ to $E_{tot}$ under $\Nm_{E_{tot}/E_i}$, which is a complete CM type. Then
as above, we obtain:

\begin{equation}
    h(E_{tot},r_{\Phi_{tot}},\chi_{\Phi_{tot}})=2^eh(R_E,r_\Phi,\chi_\Phi) + 2^{e}\sum_{i=1}^eh_1(E_i) + O_{\dim E}(\log\Disc E).
\end{equation}

The result now follows from steps 1 and 3. 

\end{proof}

\subsection{Using the height bound to prove Andr\'e-Oort}

We finally obtain Andr\'e-Oort conjecture for Shimura varieties:

\begin{theorem}\label{mainao2}
Let $S=S_K(G,X)$ be a Shimura variety. Then the Andr\'e-Oort conjecture holds for $S_K(G,x)$.
\end{theorem}

\begin{proof}

First, by \cite[Prop 2.1-2.2]{EY} we may reduce to the case where $G$ is of adjoint type and $K$ is neat and split-adjoint. Let $L$ be an ample bundle on $S$  and let $h$ be a corresponding Weil height.

Next, using \cite[Theorem 2]{BSY} it is sufficient to prove that for any Shimura subdatum $(T,r)\subset (G,X)$ and any points in $S_K(G,X)$ in the image of $S_{K\cap T(\A_f)}(T,r)$, that
 \begin{equation}\label{heightneed}
 h(w) = (\Disc(E_T) \cdot  [K_T:K\cap T(\A_f)])^{o_S(1)}
 \end{equation}.
 
But this is precisely Theorem \ref{thm:mainhtbound}

\end{proof}

\section{Construction of a Solid Height Function}\label{S-solidheightconstruction}

\subsection{Norms on 0-dimensional Shimura varieties}

\begin{definition}
    Let $S_K(T,r)$ be a Shimura variety such that $K$ is split, $V$ a $T^c$ representation of weights $[a,b]$.  We let $q:V\times V\ra \Q_{\psi}$ be a polarization of $V$, for a $\Q$-character $\psi$ of $T$, and $\V\subset V$ be a lattice. We define an admissible collection of norms  $_{K,\V,q}|\cdot |$ on $(S_K(T,r),S_K(T,r),\Gr^b _{\dR} V)$ as follows:

\begin{itemize}
    \item For each Archimedean place, we simply use the norm induced by the Hodge norm as described in section \ref{canhodge} using the polarization $q$ above.
    \item For each non-Archimedean place $v\mid p$ such that
    \begin{itemize}
        \item $T$ is unramified at $p$,
        \item $K_p$ is maximal, and 
        \item $p\geq \rng_V\cdot \dim V+2$ \footnote{This condition allows us to compare this norm with the intrinsic norm, by Theorem \ref{crisint}.}.
    \end{itemize}
    we identify $_{\dR}V_{E_v}$  with $_{{p-\dR}}V_{E_v}$ via \cite[Thm 5.3.1]{LZ2}, and use the crystalline norm associated to $\V$.
    \item For the (finitely many) other non-Archimedean places $v$, we identify $_{\dR}V_{E_v}$  with $_{{p-\dR}}V_{E_v}$ via \cite[Thm 5.3.1]{LZ2}, and use the intrinsic norm corresponding to $_{{p-\dR}}\V_{E_v}$ .
\end{itemize}

This collection is admissible by Theorem \ref{shhtadmis}. We write $h_{K,\V,q}$ for the corresponding height function on $S_K(T,X)$. If the level $K=K_T$ is maximal, we suppress it from the notation and simply write $_{\V,q}|\cdot |$ or $h_{\V,q}$.
\end{definition}

\begin{lemma}\label{lem:cnstht}
Assume the setup above. Then $h_{K,\V,q}$ is constant on $S_K(T,X)(\Qbar)$.
\end{lemma}

\begin{proof}
This proof is morally similar to the proof that principally polarized, CM abelian varieties with the same CM type and the same \emph{maximal} endomorphism ring have the same height\footnote{Note that Abelian varieties arising in the same 0-dimensional Shimura variety have locally isomorphic endomorphism rings: i.e, isomorphic over $\Z_p$ for each prime $p$. In particular, maximal and non-maximal endomorphism rings do not occur in the same 0-dimensional Shimura variety. On the level of groups, these correspond to distinct embeddings of the torus in $\GSp_{2g}$}. 

Recall that the complex points of $S_K(T,r)$ can be described as $T(\Q)\backslash T(\bbA_f)/K$. Let $x_1,x_2$ be two distinct points. Then there are many elements $t$ in $T(\bbA_f$) such that $tx_1=x_2$. We define $S(t)$ to be the set of primes $p$ such that $t_p\not\in K_p$. Each such element defines a map
$f_t:S_K(T,r)\rightarrow S_K(T,r)$. Moreover, the pullback of the automorphic local system $_{\et}\V_p$ is naturally isomorphic to the local system $_{\et}(t^{-1}\V)_p$ where
$$t^{-1}\V:=t^{-1}(\V\otimes\hat{\Z})\cap \V_\Q.$$
The intrinsic and crystalline norms are functorial by Theorem \ref{fibercomp} and Proposition \ref{intrinsifunctoriality}. Therefore 
$h_{K,\V,q}(x_2)=h_{K,t^{-1}\V,t^{-1}q}(x_1)$.

Recall the polarization $q$ is a $T$-equivariant map $V\times V\ra \Q_\psi$ for a $\Q$-character $\psi$ of $T$, giving the induced map 
$$ T(\Q)\backslash (V\times V)\times T(\bbA_f)/K \ra T(\Q)\backslash \Q_\psi\times T(\bbA_f)/K \ra \Z, $$ wherethe second map sends $(a,t')\ra a|\psi(t')| $ where we use the adelic norm on $\Q^{\times}\backslash \bbA_f^{\times}/\widehat{\Z}^{\times}$. It therefore follows that $t^{-1}q=|\psi(t^{-1})| \cdot q$, 

Moreover, the lattices corresponding to $t^{-1}\V,\V$ agree at places $v$ above primes $p\not\in S(t)$, and so at for such places  $_{K,\V,q}|\cdot|_v=_{K,t^{-1}\V,t^{-1}q}|\cdot|_v$.

Thus the difference $h_{K,\V,q}(x_2)- 
 h_{K,t^{-1}\V,t^{-1}q}(x_1)$ is a sum of rational multiples of $\log p,p\in S(t)$. By weak approximation we may pick $t$ so as to make $S(t)$ exclude any given prime. Since the logarithms of primes are $\Q$-linearly independent, we see that $h_{K,\V,q}(x_2)=h_{K,t^{-1}\V,t^{-1}q}(x_1)$ and the claim follows.
\end{proof}

By the above lemma, we write $h_{K,\V,q}$ for the constant value taken by the height function.

\subsection{Definition of the function}

Let $(T,r,\chi)$ be a Toric triple, and let $V_\chi$ be the corresponding irreducible $\Q$-representation. We set $[a,b]$ to be the smallest interval containing the weights of $V_\chi$.

We define $\Q_\chi$ to be the the commutative algebra $\End_T(V)$, and $R_\chi\subset\Q_\chi$ to be the maximal order. Note that $\Q_\chi$ is a subfield of $E_T$ and $R_\chi$ is the ring of integers. We set $\V_\chi\subset V_\chi$ to be a free $R_\chi$ sub-module, which we may therefore trivialize. 

All $\Q$-polarizations of $V_\chi$ are of the form $q(a,b):=\tr_{\Q_\chi/\Q}(\alpha a\bar{b})$ where $\alpha\in \Q_\chi$ is an element which is either totally imaginary or totally real (depending on whether the weight is odd of even), and of a pre-specified sign in each complex embedding. The discriminant of this polarization is $\Disc \Q_\chi\cdot  \Nm_{\Q_\chi/\Q}(\alpha)$, so we pick an $\alpha\in R_\chi$ with as small a norm as possible. This can be accomplished with the following elementary lemma:

\begin{lemma}
    Let $E/F$ be a CM field, with $d=[F:\Q]$, and let $\Phi$ be a CM type. For each real embedding $\sigma$ of $F$, fix a sign $s_\sigma\in \{\pm 1\}$. There exists an element $\alpha\in\Oo_F$ with $s_\sigma\sigma(\alpha)>0$ for all $\sigma$ with $\Nm_{E/\Q}(\alpha)=O_d((\Disc E)^d)$.

    Likewise, there exists a totally imaginary $\alpha\in\Oo_E$ with 
    $is_\sigma\sigma(\alpha)>0$ for all $\sigma$ with $\Nm_{E/\Q}(\alpha)=O_d((\Disc E)^d)$.
\end{lemma}

\begin{proof}

Consider the lattice $\Oo_E\subset \C^d$. By \cite[Lecture X]{Siegel}, there exists a Minkowski reduced basis of $\Oo_E$ with elements all of size at most $O_d((\Disc E)^{\frac 12})$. Thus we may pick $\beta\in \Oo_E$ within at most 
$O_d((\Disc E)^{\frac 12})$ of any element in $\C^d$, and so in particular within any hyper-quadrant. The first claim follows by taking $\alpha=\beta+\bar\beta$, and the second by taking $\alpha=\beta-\bar\beta$. 
    
\end{proof}

We now define $h(T,r,\chi)$ to be the constant value taken by the function $h_{\V,q}$ where $q$ corresponds to a polarization of Discriminant $O_d((\Disc E)^d).$




\subsection{Establishing Properties (1)-(3)}

\textbf{Property (1):}
\vspace{0.1in}
Let $\chi_1,\chi_2$ be solid characters of $(T,r)$, and set $\chi_3=\chi_1\chi_2$. Define $\V_{\chi_i}, i=1,2,3$ as above, together with the polarizations $q_1,q_2,q_3$. 

Consider $W=V_{\chi_1}\otimes V_{\chi_2}$ and $\W=\V_{\chi_1}\otimes\V_{\chi_2}$. There is a projection map $f:W\ra V_{\chi_3}$ and $f(\W)$ is stable under the image of $R_{\chi_1}\otimes R_{\chi_2}\ra R_{\chi_3}$. We may therefore take $\V_{\chi_3}:=R_{\chi_3}f(\W)$, and note that $$[\V_{\chi_3}:f(\W)]\prec_{\dim T}\Disc E_T.$$

It follows that we may write $\W$ as a superlattice containing $\V_{\chi_3}\oplus\V'$ with index $\prec_{\dim T}\Disc E_T$.  

Let $q_W:=q_1\otimes q_2$ be the induced polarization on $W$. Now the Crystalline norm is used for all of $_{\W,q_W}|\cdot|,_{\V_1,q_1}|\cdot|,_{\V_2,q_2}|\cdot|, $ for all places above primes $p\gg_{\dim T+\rng_{\chi_1}+\rng_{\chi_2}} 1$ which don't divide $\Disc E_T$, so that by the comparison Theorem \ref{crisint}, $$h(T,r,\chi_1)+h(T,r,\chi_2)=h_{\W,q_W}+O_{\dim T+\rng_{\chi_1}+\rng_{\chi_2}}(\log\Disc E_T).$$

Passing to a sublattice of index $M$ changes all the intrinsic and crystalline norms at a place $v$ by at most $|M|_v$, and so restricting the collection of admissible norms $_{\W,q_W}|\cdot|$ to $\V_{\chi_3}\oplus\V'$ we see that $$h(T,r,\chi_1)+h(T,r,\chi_2)=h_{\V_{\chi_3}\oplus\V',q_W}+O_{\dim T+\rng_{\chi_1}+\rng_{\chi_2})}(\log\Disc E_T).$$ Since the highest filtered piece of $\V_{\chi_3}\oplus\V'$ lies within $\V_{\chi_3}$, we have $$h_{\V_{\chi_3}\oplus\V',q_W}=h_{\V_{\chi_3},q_W} + O_{\dim T+\rng_{\chi_1}+\rng_{\chi_2}}(1).$$

It remains to compare the polarizations $q_W\mid\V_{\chi_3} $ and $q_3$. They are both of discriminant $\prec_d \Disc E_T$, and so  trivializing $\V_{\chi_3}$ we may write them as $(x,y)\ra \tr_{\Q_{\chi_3}/\Q}(x\alpha\bar y)$ for elements $\alpha_W,\alpha_3\in R_{\Q_{\chi_3}}$ of norms $\prec_d \Disc E_T$. It follows that at an archimedean places $w$ the induced norms differ by a multiplicative factor of $w\circ \chi(\alpha_W\alpha_3^{-1})$ on $\Gr^b V_{\chi_3,\C}$, and hence the difference $h_{\V_{\chi_3},q_W} - h_{\V_{\chi_3},q_3}$ is equal to $\log|\Nm_{\Q_{\chi_3}/\Q}(\alpha_W\alpha_3^{-1})|$ which is $O_{\dim T+\rng_{\chi_1}+\rng_{\chi_2}}(\log\Disc E_T)$ as desired.

\vspace{0.1in}
\textbf{Property (3):}
\vspace{0.1in}

Let $(T_2,r_2,\chi_2)$ is solid, and $f:(T_1,r_1)\ra (T_2,r_2)$ is a morphisms of Shimura data with $T_1,T_2$ Tori. Let $\chi_1:=\chi_2\circ f$, and we identify $V_{\chi_1}$ with $V_{\chi_2}$.  We claim that $\Q_{\chi_1}=\Q_{\chi_2}$. Indeed, both may be described as the elements of $\End_\Q(V_{\chi_1})$ which preserve the Hodge decomposition. Thus we may take $\V=\V_{\chi_2}$.

Now as in the proof of property (1), the archimedean contributions to the solid heights of the two polarizations $q_1,q_2$ differ by $O_{\dim T_1+\dim T_2+\rng\chi}(\log\Disc E_{T_1}+\log\Disc E_{T_2})$, and the intrinsic and crystalline norms agree whenever they are both defined. Now the crystalline norm is used for $_{\V_\chi,q_i}|\cdot|_v$ for $i=1,2$ for all places $v$ lying above primes $p$ with $p\gg_{\rng_\chi+\dim T_1+\dim T_2} 1$ which don't divide $\log\Disc E_T$,  so by another application of Theorem \ref{crisint} the result follows.

\vspace{0.1in}
\textbf{Property (2):}
\vspace{0.1in}

We claim that property (2) follows formally from (1) and (3). Indeed, the $m$'th power map gives a morphism $f:(T,r)\ra (T,r^m)$ such that $\chi\circ f=\chi^m$. Therefore (3) says that $$h(T,r^m,\chi)-h(T,r,\chi^m) = O_{\dim T+\rng_\chi+m}(\log\Disc E_T).$$ Finally, by applying property (1) we have that 
$$h(T,r,\chi^m)-mh(T,r,\chi) = O_{\dim T+\rng_\chi+m}(\log\Disc E_T)$$ and the result follows.

\subsection{Establishing Property (4)}

Let $S=S_K(G,X)$ be a Shimura variety, and $V$ an irreducible representation of $G$ with a sublattice $\V$ fixed by $K$, and a polarization $q$. Assume that the highest weight piece $L:=\Fil^a _{\dR}V_\C$ is 1-dimensional, and let $h_L$ be a Weil height on a Toroidal compactification $\bar S$ of $S$ corresponding to the Deligne extension of $L$. The result of (4) is evidently invariant under changing $K$, so we assume that $K$ is split. We let $h_0$ denote the canonical norm. By Corollary \ref{shhtweil} the different $|h_L-h_0|$ is bounded by $O_S(\log^+ h_A(x))$ so it is sufficient for us to prove (4) with $h_0$ replacing $h_L$.

Now for a Shimura subdatum $i:(T,r)\hookrightarrow (G,X)$, let $K':=i^*K$, and let $\chi$ denote the character $\Fil^a_rV_\C\mid T$.  Let $_0|\cdot|$ denote the collection of norms on the line bundle $\Gr^a_{\dR}(V\mid T)$ on $S_{K'}(T,r)$ pulled back from those on $S_K(G,X)$. Then for any $x\in S_K(T,r)(\Qbar)$ the value $h_0(i(x))$ can be computed on $S_K(T,r)$ using the bundle  $\Gr^a_{\dR}(V\mid T)$ with the collection $_0|\cdot |$. Now these norms agree with the norms $_{K',\V,q}|\cdot|$ at the archimedean places, and at all finite places $v$ lying above primes $p\gg_G 1$ not dividing $[K_T:K']\disc E_T$. Thus, by the comparison Theorem \ref{crisint} it is sufficient to prove (4) replacing $h_0(i(x))$ by $h_{K',\V,q}(x)$.

Let $E$ denote the center of the endomorphism algebra $\End_T(V)$, and $\Oo_E$ denote the maximal order. By \cite[Theorem 4.1]{DO}, there is a subring $R\subset \Oo_E$ preserving $\V$ with $[\Oo_E:R]\prec_G [K_T:K']\disc E_T$.  There is therefore a sublattice 
$\V'\subset \V$ of index $\prec_G [K_T:K']\disc E_T$ which is stable under $\Oo_E$.  

Since $V$ admits $V_\chi$ as a summand of multiplicity 1, $\Oo_E$ admits $R_\chi$ as a direct summand, and therefore $\V'$ admits an $R_\chi$ sublattice $U$ as a direct summand. $U$ may not be free, but it is isomorphic to an ideal class, and so has a free submodule of index bounded by $O(\Disc(E_T)^{\frac12}).$ Thus, we may pick $\V'$ to admit $\V_\chi$ as a direct summand. 
Again using Theorem \ref{crisint} and arguing as above, we have that $$\max(|h_{K',\V,q} - h_{K',\V',q}|,|h_{K',\V',q}-h_{K',\V_{\chi},q}|) = O_G\big(\log \disc E_T +\log[K_T:K']\big),$$ so it sufficient to prove (4) replacing $h_{K',\V,q}$ by $h_{K',\V_{\chi},q}$. 

Now the only difference between $h_{K,\V_{\chi},q}(x)$ and $h_{K_T,\V_{\chi},q_\chi}(x)$ occur at the archimedean places corresponding to the distinct polarizations, and at the finite places at which we use the intrinsic norms vs. crystalline norms. The polarizations are both of discriminant $\prec_G [K_T:K']\disc E_T$ and so are handled as in the proof of property (3), and the norms agree for primes not dividing $[K_T:K']$ and so are handled by Theorem \ref{crisint}. 
This completes the proof.

\bibliographystyle{alpha}

\begin{thebibliography}{9}

\bibitem{Andre}
Y. Andr\'e, 
\textbf{Finitude des couples d'invariants modulaires singuliers sur une courbe alg\'ebrique plane non modulaire.} 
{\it J. Reine Angew. Math.} 505 \textbf{1998},p. 203–208

\bibitem{AGHM}
 F. Andreatta, E. Goren, B. Howard, and K. Madapusi, \textbf{Faltings heights of abelian varieties with complex multiplication}
 {\it Ann. of Math.} (2) 187 \textbf{2018}, no. 2, 391–531

\bibitem{Autissier} P. Autissier, 
\textbf{Hauteur moyenne de variétés abéliennes isogènes}, \textit{Manuscripta Math.} 117 (2005), no. 1.

\bibitem{BT-hodgeaxs}
B. Bakker and J. Tsimerman, \textbf{The Ax-Schanuel conjecture for variations of Hodge structures}, {\it Invent. Math.} 217 \textbf{2019}, no. 1, p. 77–94

\bibitem{Binyamini}
G. Binyamini,
\textbf{Point counting for foliations over number fields},
{\it https://arxiv.org/abs/2009.00892}

\bibitem{BSY}
G. Binyamini, H. Schmidt, and  A. Yafaev,
\textbf{Lower bounds for Galois orbits of special points on Shimura varieties: a point-counting approach}
{\it Math. Ann.} 385 \textbf{2023} no. 1-2, p. 961--973. 

\bibitem{BG}
E. Bombieri and W. Gubler,
\textbf{Heights in Diophantine Geometry}
{\it Cambridge University Press} \textbf{2009}

\bibitem{brinon}
O. Brinon, \textbf{Representations $p$-adiques cristallines et de de Rham dans le cas relatif}, {\it Mem. Soc. Math. Fr. (N.S.)}, no. 112 (2008). 

\bibitem{brinonconrad}
O. Brinon and B. Conrad,
\textbf{CMI summer school notes on $p$-adic Hodge theory}
https://math.stanford.edu/~conrad/papers/notes.pdf
\bibitem{col}
P.Colmez, 
\textbf{P\'{e}riodes des vari\'{e}t\'{e}s ab\'{e}liennes \`a multiplication complexe}, {\it Ann. of Math. (2)}, no.138, (1993)




\bibitem{chaifaltings}
C.Chai, G.Faltings, \textbf{Degeneration of Abelian Varieties},  {\it Springer Science and Business Media}, \textbf{Vol.22.}

\bibitem{chiu}
K. Chiu, \textbf{Ax-Schanuel for variations of mixed Hodge structures},
{\it https://arxiv.org/abs/2101.10968}

\bibitem{conradgrunwaldwang} B. Conrad, \textbf{Lifting global representations with local properties} http://virtualmath1.stanford.edu/~conrad/papers/locchar.pdf



\bibitem{DO}
C. Daw and M. Orr,
\textbf{Heights of pre-special points of Shimura varieties}
{\it Math. Ann. 365} \textbf{2016}, no. 3-4, p. 1305–1357

\bibitem{D}
P. Deligne,
\textbf{Vari\'et\'es de Shimura: interpr\'etation modulaire, et techniques de construction de
mod\'eles canoniques}
{\it Automorphic forms, representations and L-functions (Proc.
Sympos. Pure Math., Oregon State Univ., Corvallis, Ore., 1977)} Part 2, Proc. Sympos. Pure Math.,
XXXIII. Amer. Math. Soc., Providence, R.I \textbf{1979} p. 247–289

\bibitem{Deligne-1971c}
P.Deligne,
\textbf{Travaux de Shimura},{\it info S\'em. Bourbaki} F\'ev 1971, Expos\'e 389, \textbf{Lect.
Notes Math.} vol. 244, \textit{Springer}, Heidelberg




\bibitem{LZ2}
H. Diao, K.-W. Lan, R. Liu, and X. Zhu,
\textbf{Logarithmic Riemann-Hilbert correspondences for rigid varieties}
{\it https://arxiv.org/abs/1803.05786}

\bibitem{Edixhoven}
B. Edixhoven,
\textbf{Special points on products of modular curves} 
{\it Duke Math. J.} 126 \textbf{2005}, no. 2, p. 325–348

\bibitem{EY}
B. Edixhoven, A, Yafaev, \textbf{Subvarieties of Shimura varieties.} {\it Annals of Mathematics} Volume 157 (2003).

\bibitem{EsnaultGroechenig} H. Esnault, M. Groechenig, {\bf Rigid connections and $F$-isocrystals,} {\it Acta Mathematica} 225, no. 1 {\bf 2020} p. 103--158. 

\bibitem{EG2}
H. Esnault, M. Groechenig,
{\bf Cohomologically rigid local systems and
integrality}, {\it Selecta Mathematica}, 
Volume 24, No 5, {\bf 2018} p.4279--4292

\bibitem{FaltingsCrystallineCohomology}
G. Faltings,
\textbf{Crystalline Cohomology and p-Adic Galois Representations}
{\it Algebraic analysis, geometry, and number theory , Johns Hopkins Univ. Press}  \textbf{1989} p.25-80

\bibitem{Gao-reduction}
Z. Gao, 
{\bf About the mixed Andr\`{e}-Oort conjecture: reduction to a lower bound for the pure case}
{\it Comptes rendus Math\'{e}matiques} vol. 354, \textbf{2016} p. 659-663

\bibitem{GK}
Z. Gao and B. Klingler,
\textbf{Ax-Schanuel for variations of mixed Hodge structures},
{\it https://arxiv.org/abs/2101.10938}

\bibitem{smallreps}
R.M. Guralnick 
\textbf{Small representations are completely reducible}
{\it J. Algebra} 220 \textbf{1999}, no. 2, 531-541. 



\bibitem{Harris}
M. Harris,
\textbf{Period Invariants of Hilbert Modular Forms, II},
{\it Compositio Mathematica}, tome 94, no 2,\textbf{1994}, p. 201-226

\bibitem{hansen}
D. Hansen,
\textbf{Period morphisms and variations of p-adic Hodge structure (preliminary draft)},
{\it http://www.davidrenshawhansen.com/periodmapmod.pdf}


\bibitem{Jantzen}
J.C. Jantzen,
\textbf{Representations of algebraic groups}
Pure and applied mathematics, 131 {\it Academic Press, Inc., Boston, MA.\textbf{1987}}

\bibitem{Kato}
K. Kato,
\textbf{Heights of motives}
{\it Proc. Japac Acad. Ser. A Math. Sci.} 90 (2014), no.3, 49--53. 

\bibitem{Koshikawa}
T. Koshikawa,
\textbf{On heights of motives with semistable reduction},
{\it https://arxiv.org/abs/1505.01873v3}

\bibitem{UY}
B. Klingler, E. Ullmo, A. Yafaev,
\textbf{The hyperbolic Ax-Lindemann-Weierstrass conjecture}
{\it Publ. Math. Inst. Hautes \'Etudes} Sci. 123 \textbf{2016}, p. 333–360


\bibitem{Kiehl}
R.Kiehl,
\textbf{Die de Rham Kohomologie algebraischer Mannigfaltigkeiten \"uber
einem bewerteten K\"orper}
{\it Inst. Hautes Etudes Sci. Publ. Math.}  33, \textbf{1967}, p.5-20

\bibitem{KY}
B. Klingler and  A. Yafaev, 
\textbf{The Andr\'{e}-Oort conjecture}
{\it Ann. of Math.} (2) 180 \textbf{2014}, no. 3, p. 867–925

\bibitem{lanshimura}
K. Lan,
\textbf{An example-based introduction to Shimura varieties}
 https://www-users.cse.umn.edu/~kwlan/articles/intro-sh-ex.pdf, {\it preprint}




\bibitem{LO}
M. Lieblich, M. Olsson, 
\textbf{Martin Generators and relations for the \'etale fundamental group}
{\it Pure Appl. Math. Q.} 6 \textbf{2010}, no. 1, Special Issue: In honor of John Tate. Part 2, p. 209–243

\bibitem{LZ}
R. Liu and X. Zhu,
\textbf{Rigidity and a Riemann-Hilbert correspondence for $p$-adic local systems.}
{\it Invent. Math. 207} \textbf{2017}, no. 1, p.291–343

\bibitem{Lovering}
T. Lovering, 
\textbf{Integral canonical models for automorphic vector bundles of abelian type.}
{\it Algebra Number Theory} 11 \textbf{2017}, no. 8, p. 1837–1890.

\bibitem{margulis}
G. Margulis, \textbf{Discrete subgroups of semisimple lie groups}, vol. 17, Springer Science \& Business Media, \textbf{1991}




\bibitem{milne-models}
J. Milne,
\textbf{Canonical Models of
(Mixed) Shimura Varieties and
Automorphic Vector Bundles}
{\it Automorphic forms, Shimura varieties, and L-functions, Vol. I (Ann Arbor, MI, 1988),} 283-414, Perspect. Math., 10, {\it Academic Press, Boston, MA,} 1990.

\bibitem{milne-cm}
J.Milne, \textbf{Motives over finite fields.} {\it Motives (Seattle, WA, 1991),}  401–459, Proc. Sympos. Pure Math., 55, Part 1, {\it Amer. Math. Soc., Providence, RI, 1994}

\bibitem{milne-moduli}
J. Milne,
\textbf{Shimura Varieties and Moduli} 
{\it https://www.jmilne.org/math/xnotes/svh.pdf}

\bibitem{milne-intro}
J. Milne,
\textbf{Introduction to Shimura Varieties}
{\it https://www.jmilne.org/math/xnotes/svi.pdf}

\bibitem{milne-alg}
J. Milne,
\textbf{Algebraic Groups}
{\it https://www.jmilne.org/math/CourseNotes/iAG200.pdf}

\bibitem{milne-aut}
J.Milne,
\textbf{Automorphic vector bundles on connected Shimura varieties}, {\it Invent.math.} 92, \textbf{1988}, pp. 91–128.

\bibitem{lucia} L. Mocz,
\textbf{A new Northcott property for Faltings height}, {\it Ph.D. dissertation, Princeton University, Princeton, 2017}. 

\bibitem{MPT}
N. Mok, J. Pila, and J. Tsimerman,
\textbf{Ax-Schanuel for Shimura varieties}
{\it Ann. of Math.} (2) 189 \textbf{2019}, no. 3, p.945–978

\bibitem{PT-axlag}
J. Pila and J. Tsimerman,
\textbf{Ax-Lindemann for $\cA_g$},
{\it Ann. of Math.} (2) 179 \textbf{2014}, no. 2, p. 659–681

\bibitem{PT-surfaces}
 J. Pila and J. Tsimerman, 
 \textbf{The Andr\'e-Oort conjecture for the moduli space of abelian surfaces}
 {\it Compos. Math.} 149 \textbf{2013}, no. 2, p.204–216

\bibitem{Pila}
J. Pila,  \textbf{O-minimality and the Andr\'e-Oort conjecture for $\C^n$}
{\it Ann. of Math.} (2) 173 \textbf{2011}, no. 3,p. 1779–1840

\bibitem{PW}
J. Pila and A. Wilkie,
\textbf{The rational points of a definable set.} {\it Duke Math. J. 133} \textbf{2006}, no. 3,p. 591–616


\bibitem{PZ}
J. Pila and U. Zannier,
\textbf{Rational points in periodic analytic sets and the Manin-Mumford conjecture.}
{\it Atti Accad. Naz. Lincei Rend. Lincei Mat. Appl.} 19 \textbf{2008}, no. 2, p.149–162

\bibitem{Schmid}
W. Schmid, 
\textbf{Variation of Hodge structure: the singularities of the period mapping}
\emph{Invent. Math.} 22 \textbf{1973}, p.211–319

\bibitem{Schmidt}
H. Schmidt,
\textbf{Counting rational points and lower bounds for Galois orbits},
{\it Rend. Acad. Lincei} 30 \textbf{2019}, p. 497--509

\bibitem{S}
P. Scholze,
\textbf{$p$-adic Hodge theory for rigid-analytic varieties}
{\it Forum Math. Pi} 1 \textbf{2013} p.1-77

\bibitem{S-C}\
P. Scholze,
\textbf{$p$-adic Hodge theory for rigid-analytic varieties-corrigendum}
{\it Forum Math. Pi} 1 \textbf{2016}

\bibitem{ShankarTang}\ A.N. Shankar and Y. Tang,
\textbf{Exceptional splitting of reductions of abelian surfaces}, {\it Duke Math. J.} 169 (2020), no. 3. 

\bibitem{Siegel}
C. L. Siegel,
\textbf{Lectures on the Geometry of Numbers}, {\it Springer-Verlag}, Berlin, 1989.

\bibitem{T}
J. Tsimerman, \textbf{The Andr\'e-Oort conjecture for $\cA_g$},
{\it Ann. of Math.} (2) 187 \textbf{2018}, no. 2, 379–390


\bibitem{Tsuji}
T. Tsuji,
\textbf{Crystalline $\Z_p$-Representations and $A_{\inf}$-Representations with Frobenius}
{\it P-adic Hodge Theory, Simons Symposia, Springer}, \textbf{2020}, p.161-319

\bibitem{UYA}
E. Ullmo and A. Yafaev,
\textbf{Galois orbits and equidistribution of special subvarieties: towards
the Andr\'e-Oort conjecture},
{\it Ann. of Math.} (2) 180 \textbf{2014}, no. 3, 823-865


\bibitem{YZ}
X. Yuan, S.-W. Zhang,
\textbf{On the averaged Colmez conjecture}
{\it Ann. of Math.} (2) 187 \textbf{2018}, no. 2, p. 533–638

\bibitem{stacks-project} The Stacks project authors, The Stacks project, \url{https://stacks.math.columbia.edu}, \textbf{2021}.



\end{thebibliography}

\begin{thebibliography}{DK09-2}

\bibitem[Del70]{Del70} Deligne, P.: {\it \'Equations Diff\'erentielles \`a Points Singuliers R\'eguliers}, Lecture Notes in Mathematics {\bf 163} (1974), Springer Verlag. 
\bibitem[EG18]{EG18} Esnault, H., Groechenig, M.: {\it Cohomologically rigid local systems and integrality}, Selecta Mathematica {\bf 24} 5 (2017), 4279--4292.
\bibitem[EG20]{EG20} Esnault, H., Groechenig, M.: {\it Rigid connections and F-isocrystals}, Acta Math. {\bf 225} 1 (2020),  103--158.
\bibitem[EV86]{EV} Esnault, H., Viehweg, E.: {\it Logarithmic de Rham complexes and vanishing theorems}, Inventiones  math. {\bf 86} (1986), 161--194.
\bibitem[Fal88]{Fal88} Faltings, G.: {\it Crystalline cohomology and $p$-adic Galois-representations}, Algebraic analysis, geometry, and number theory (Baltimore, MD, 1988), 25-80, Johns Hopkins Univ. Press, Baltimore, MD, 1989.

\bibitem[Kat96]{Kat96} Katz, N.: {\it Rigid local systems}, Princeton University Press (1996). 
\bibitem[LSZ13]{LSZ13} Lan, G., Sheng, M., Zuo, K.: {\it Semistable Higgs bundles, periodic Higgs bundles and representations of  algebraic fundamental groups},  J. Eur. Math. Soc. (JEMS) \textbf{21} (2019), no. 10, 3053--3112.
\bibitem[LSYZ19]{LSYZ19} Lan, G., Sheng, M., Yang, Y., Zuo, K.: {\it Uniformization of p-adic curves via Higgs-de Rham flows},  Journal f\"ur die reine und angewandte Mathematik (Crelles Journal) {\bf 747}  (2019),  63--108.
\bibitem[Lan14]{Lan14} Langer, A.: {\it Semistable modules over Lie algebroids in positive characteristic},  Doc. Math. {\bf 19} (2014), 509--540. 
\bibitem[Mo06]{mochizuki} Mochizuki, T.: {\it Kobayashi-Hitchin correspondence for tame harmonic bundles and an application}, Ast\'erisque {\bf 309} (2006), viii+117.
\bibitem[Nit93]{nitsure} Nitsure, N.: {\it Moduli of semistable logarithmic connections}, Jour. Amer. Math. Soc. {\bf 6} (1993) no 3, 597--609.
\bibitem[Sch08]{schepler} Schepler, D.: {\it Logarithmic nonabelian Hodge theory in characteristic $p$}, \url{https://arxiv.org/abs/0802.1977} (2008). 
\bibitem[SYZ22]{SYZ} Sun, R., Yang, J., Zuo, K.: {\it Projective crystalline representations of \'etale fundamental groups and twisted periodic Higgs-de Rham flow}, J. Eur. Math. Soc. {\bf 24} (2022) no 6, 1991--2076.
\bibitem[Sim90]{Sim90} Simpson, C.: {\it Harmonic bundles on non-compact curves}, Jour. Amer. Math. Soc. {\bf 3} (1990) no 3, 713--770. 
\bibitem[Sim94]{Sim94} Simpson, C.: {\it Moduli of representations of the fundamental group of a smooth projective variety. I}, Publ. math. Inst. Hautes \'Etudes Sci. {\bf 79} (1994), 47--129. 
{\bibitem[Xu19]{Xu19} Xu, D.: {\it Rel\`evement de la transform\'ee de Cartier d'Ogus-Vologodsky modulo $p^n$}, M\'emoires de la SMF {\bf 163} (2019), 6--144.}












\end{thebibliography}

\appendix

\section{Frobenius structures and unipotent monodromy at infinity\\by H\'el\`ene Esnault and Michael Groechenig}

We fix an irreducible affine base scheme $S$ which is of finite type over a universally Japanese ring. For the purpose of this appendix, $S$ will either be $\Spec \Cb$, $\Spec \Fb_q$ or $\Spec R$, where $R$ is a finite type algebra. Let us denote by $\bar{X}_{S}$ a smooth and projective $S$-scheme with a relatively very ample line bundle $\Oo_{\bar{X}_S}(1)$. Let $X_{S} \subset \bar{X}_{S}$ be an open subscheme such that $\bar{X}_{S} \setminus X_{S}$ is a strict normal crossings  divisor (snc) $D_{S} = \bigcup_{\mu=1}^cD_{S}^\mu$.
The sheaf of degree $n$ K\"ahler differentials with log-poles along $D$ will be denoted by $\Omega^n_{\bar{X}_{S}/S}\langle D \rangle$. For $\mu=1,\dots,c$ we write $\res_\mu\colon \Omega^1_{\bar{X}_{S}/S}\langle D \rangle \to {\mathcal O}_{D^\mu_{S}/S}$ for the residue map.

\begin{definition}\label{definition}
\begin{itemize}
\item[(a)] A \emph{log-dR local system} on $\bar{X}_{S}$ is a pair $(E_{S},\nabla_{S})$ where $E_{S}$ is a vector bundle of rank $r$ on $\bar{X}_{S}$ and 
$$\nabla\colon E_{S} \to E_{S}\otimes \Omega_{\bar{X}_{S}/S}^1\langle D \rangle$$
is a flat logarithmic connection such that $\res_{\mu}(\nabla)\in H^0(D^\mu_S, \End(E_S|_{D^\mu_S}))$ is \emph{nilpotent} for all $\mu=1,\dots,c$.
\item[(b)] We say that $(E_{S},\nabla_{S})$ is \emph{strongly cohomologically rigid}, if 
\begin{equation}\label{eqn:logdRcomplex}
\mathbb{H}^1\big(\bar{X}_{S},[\End(E_S) \xrightarrow{ \End(\nabla_S)} \End(E_S) \otimes \Omega^1\langle D \rangle \xrightarrow{\End(\nabla_S)} \cdots]\big) = 0.\end{equation}
\item[(c)] A \emph{log-Higgs bundle} on $\bar{X}_S$ is a pair $(V_S,\theta_S)$, where $V_S$ is a vector bundle of rank $r$ on $\bar{X}_S$ and $\theta_S$ is an $\Oo$-linear morphism
$V \to V \otimes \Omega_{\bar{X}}^1\langle D \rangle$ satisfying $\theta_S \wedge \theta_S = 0$.
\item[(d)] A log-Higgs bundle $(V_S,\theta_S)$ is called \emph{strongly cohomologically rigid}, if 
\begin{equation}\label{eqn:logHiggscomplex}
\mathbb{H}^1\big(\bar{X}_{S},[\End(V_S) \xrightarrow{ \End(\theta_S)} \End(V_S) \otimes \Omega^1\langle D \rangle \xrightarrow{\End( \theta_S)} \cdots]\big) = 0.\end{equation}
\end{itemize}
\end{definition}

\begin{rmk}\label{rmk:chern}
\begin{itemize}
\item[(a)] If $S=\Spec \Cb$, the underlying vector bundle $E$ of a log-dR local system has vanishing Chern classes. This follows from the formula for the Atiyah class of $E$ given in \cite[Proposition B.1]{EV}. In addition, the left-hand side of \eqref{eqn:logdRcomplex} computes 
$H^1(X({\mathbb C}), {\End}(E_{\mathbb{C}, {\rm an}})^{\End(\nabla_{\mathbb C})})$. Indeed, as ${\rm res}_\mu(\nabla_{\mathbb C})$ is nilpotent for  $\mu=1,\ldots, c$,  so  is  ${\rm res}_\mu( \End(\nabla_{\mathbb C}))$, thus $\End(E_{\mathbb C})$ is Deligne's extension the cohomology of which computes analytically $Rj_*$ where $j: X_{\mathbb{C}, {\rm an}}\to \bar X_{\mathbb{C}, {\rm an}}$, see \cite[II, Proposition~3.13, Corollaire~3.14]{Del70}. 
\item[(b)] The notion of \emph{strong cohomological rigidity} is more restrictive than the one of \emph{cohomological rigidity} used in \cite{Kat96,EG18}. A cohomological rigid local system, in the traditional sense, does not have any non-trivial infinitesimal deformations which leave the monodromies at infinity invariant. A strongly cohomologically rigid log-dR local system does not have any non-trivial infinitesimal deformations, independently of any constraints at the boundary.
\end{itemize}
\end{rmk}

\begin{definition}[Arithmetic models]
Let $(\bar{X}_{\Cb},D_{\Cb},\Oo_{\bar{X}}(1))$ be a triple consisting of a smooth projective complex variety $\bar{X}_{\Cb}$, an snc divisor $D_{\Cb}$, and a very ample line bundle $\Oo_{\bar{X}_{\Cb}}(1)$. 
\begin{itemize}
\item[(a)] An \emph{arithmetic model} for $(\bar{X}_{\Cb},D_{\Cb},\Oo_{\bar{X}_{\Cb}}(1))$ is given by an affine scheme $S$ where $\Gamma(S,\Oo_S)$ is a finite type subring $R \subset \Cb$, a smooth projective $S$-scheme $\bar{X}_S$ together with an snc divisor $D_S$ such that
$$\bar{X}_{\Cb} = \bar{X}_S \times_S \Spec \Cb\text{ and }D_{\Cb}=D_{S} \times_S \Spec \Cb,$$
and a relatively very ample line bundle $\Oo_{\bar{X}_S}(1)$ pulling back to $\Oo_{\bar{X}_{\Cb}}(1)$. 

\item[(b)] Let $\{(E^i_{\Cb},\nabla^i_{\Cb})\}_{i \in I}$ be a family of log-dR local systems on $X_{\Cb}$. An arithmetic model for $(\bar{X}_{\Cb},D_{\Cb},\Oo_{\bar{X}_{\Cb}},\{E^i_{\Cb},\nabla_{\Cb}^i\}_{i \in I})$ is given by an arithmetic model for $(\bar{X}_{\Cb},D_{\Cb},\Oo_{\bar{X}_{\Cb}}(1))$ as in (a), and log-dR local systems $\{(E^i_{S},\nabla^i_{S})\}_{i \in I}$ on $X/S$ satisfying
$$(E^i_{\Cb},\nabla_{\Cb}^i)=(E^i_S,\nabla^i_S)|_{\bar{X}_{\Cb}} \text{ for all }i \in I.$$
\end{itemize}
\end{definition}

\begin{theorem}\label{thm:main}
Suppose that every stable log-dR local system $(E_{\Cb},\nabla_{\Cb})$ of rank $r$ on $(\bar{X}_{\Cb},D_{\Cb})$ is strongly cohomologically rigid. Then, there exists a finite type subalgebra $R \subset \Cb$ and a model of $(\bar{X}_{\Cb},X_{\Cb},D_{\Cb})$ over $S=\Spec R$ such that every stable log-dR local system of rank $r$ on $(\bar{X}_{\Cb},D_{\Cb})$ has an $S$-model $(E_S,\nabla_S)$ such that for every finite field $k$ and every morphism $R \to W(k)$ the formal flat connection
$$(\widehat{E}_W,\widehat{\nabla}_W)$$
is endowed with the structure of a \emph{torsionfree Fontaine-Laffaille module} on $X_W=\bar{X}_W \setminus D_W$.
\end{theorem}
\begin{rmk}
In \cite{EG20} we prove a stronger result for the case where $D_{\Cb}=\emptyset$. The assumptions of \emph{loc. cit.} are less stringent, as they apply more generally to arbitrary rigid dR local system, i.e. isolated points of the moduli space $\sM_{dR}$. The additional assumptions above allow one to simplify the argument significantly.
\end{rmk}

\subsection{Construction of a suitable arithmetic model}

Moduli spaces of logarithmic flat connections on complex varieties were constructed by Nitsure in \cite{nitsure}. Using Langer's boundedness (see \cite{Lan14}), this construction was extended to more general base schemes (\cite[Theorem 1.1]{Lan14}):

\begin{theorem}[Langer]\label{thm:langer}
For a fixed polynomial $P$ there exists a quasi-projective $S$-scheme $\sM_{dR}(\bar{X}_S,D_S)$ of stable flat logarithmic connections on $\bar{X}_S$ with Hilbert polynomial $P$.
\end{theorem}

More generally, Langer constructs moduli spaces for semistable $\Lambda$-modules, where $\Lambda$ is a ring of operators in the sense of \cite{Sim94}. It is explained on p. 87 of \emph{loc. cit.} that flat logarithmic connections are a special case of the general theory of $\Lambda$-modules. 
We are interested in moduli spaces of flat logarithmic connections with \emph{vanishing Chern classes} (see Remark \ref{rmk:chern}). The corresponding Hilbert polynomial satisfies
$$P_0(n)=\int r\cdot{} \mathrm{td}_{\bar{X}_{\Cb}} \mathrm{ch}(\Oo_{\bar{X}_{\Cb}}(n))\text{ for all }n \in \mathbb{N}.$$

\begin{corollary}\label{cor:langer}
There exists a closed subscheme $\sM_{log{\mbox -}dR}(\bar{X}_S,D_S) \subset \sM_{dR}(\bar{X}_S,D_S)$, which is the moduli space of stable log-dR local systems with Hilbert polynomial $P_0$.
\end{corollary}
\begin{proof}
There is an \'etale covering $\big(U_i \to \sM_{dR}(\bar{X}_S,D_S)\big)_{i \in I}$ 
such that we have a universal family $(\mathcal{E}_{U_i},\nabla_{U_i})$ on $U_i \times_S \bar{X}_S$. By stability, such a universal log-dR $U_i$-family is well-defined up to tensoring by a line bundle on $U_i$. 
By construction, the characteristic polynomial  $\chi_{i,\mu}(T)$ of $\res_\mu(\nabla_{U_i}) 
  $  is a section of a locally free sheaf on $U_i$.
We let $Z_i \hookrightarrow U_i$ be the closed immersion corresponding to the vanishing locus of $(\chi_{i,\mu}(T)-T^r), \mu=1,\ldots, c$. This closed immersion is independent of the choice of a $U_i$-universal family, since tensoring by a line bundle on $U_i$ leaves $\chi_{i,\mu}$ invariant. We may thus apply faithfully flat descent theory to glue those closed immersions to a closed embedding 
$$Z \hookrightarrow \sM_{dR}(\bar{X}_S,D_S).$$
The scheme $Z$ is the sought-for moduli space $\sM_{log{\mbox -}dR}(\bar{X}_S,D_S)$.
\end{proof}

We record the following consequence of non-abelian Hodge theory for later reference.
\begin{theorem} \label{thm:mochi}
For every strongly cohomologically rigid log-dR local system $(E_{\Cb},\nabla_{\Cb})$ on $\bar{X}_{\Cb}$ there exists an $F$-filtration $\cdots \subset F^i\subset F^{i-1} \subset \cdots F^0=E$ satisfying Griffiths transversality $\nabla\colon F^i\to F^{i-1} \otimes_{\Oo_X}  \Omega^1_X\langle D \rangle$ and with the associated graded sheaves $\mathrm{gr}_F^i E = F^i/F^{i+1}$ being locally free. The associated Higgs bundle is denoted by
$$\big(\mathrm{gr}_F E, \mathrm{KS}\big),$$
 where $\mathrm{KS}$ stands for \emph{Kodaira-Spencer} and is defined by the linear maps $$\mathrm{gr}_F \nabla\colon  \mathrm{gr}_F^i E\to \mathrm{gr}_F^{i-1} E \otimes_{\Oo_X} \Omega^{1}_X\langle D \rangle.$$ 
\end{theorem}
\begin{proof}
Mochizuki proved in \cite[Theorem 10.5]{mochizuki} that every log-dR local system on $\bar{X}_{\Cb}$ can be complex analytically deformed to a polarised variation of Hodge structures, which implies the existence of the requisite $F$-filtration on the rigid $(E_{\Cb},\nabla_{\Cb})$. In \emph{loc. cit.}, this is stated in terms of Betti local systems on $X_{\Cb} = \bar{X}_{\Cb} \setminus D_{\Cb}$. This is an equivalent perspective, by virtue of the Riemann-Hilbert correspondence which is complex analytic. Due to strong cohomological rigidity, $(E_{\Cb},\nabla_{\Cb})$ cannot be deformed in a non-trivial manner. We conclude that $(E_{\Cb},\nabla_{\Cb})$ underlies a polarised variation of Hodge structures.
\end{proof}

\begin{rmk}
Stability of the log-Higgs bundle $\left(\bigoplus_j\mathrm{gr}_{F}^{ij}E^i_{\mathbb{C}},\mathrm{KS}(\nabla^i_{\mathbb{C}})\right)$ is implied by Mochizuki's parabolic Simpson correspondence \cite{mochizuki}. We remark that the parabolic structure is trivial in the case at hand, since we assume that the monodromies around the divisor at infinity are unipotent and therefore in this case, parabolic stability amounts to stability in the usual sense of log-Higgs bundles. See \cite[p. 722]{Sim90} where the triviality of the parabolic structure is justified for the curve case. The argument given there generalises directly to higher dimensional varieties.
\end{rmk}
{Subsequently, for every strongly cohomologically rigid log-dR local system $(E_{\Cb},\nabla_{\Cb})$ on $\bar{X}_{\Cb}$ 
we fix the $F$-filtration constructed in Theorem~\ref{thm:mochi}.}
\begin{proposition}\label{prop:model}
We keep the assumptions of Theorem \ref{thm:main}. There exists an arithmetic model $(S,\bar{X}_S,D_S,\Oo_{X_S}(1))$ of $(X_{\Cb},D_{\Cb},\Oo_{X_{\Cb}}(1))$ such that
\begin{itemize}
\item[(a)] all rank $r$ log-dR local systems $(E^i_{\Cb},\nabla^i_{\Cb})_{i\in I}$ have a locally free model $(E^i_S,\nabla^i_S)_{i\in I}$ over $S$,
\item[(b)] the models $(E^i_{S},\nabla^i_{S})_{i\in I}$ are also strongly cohomologically rigid,
\item[(c)] the filtrations $(F^{ij}_{\Cb} \subset E^i_{\Cb})$ are defined over $S$ such that the $S$-relative filtrations $F^{ij}_{S}\subset E^i_S$ satisfy the \emph{Griffiths-transversality condition},
\item[(d)] for every $i\in I$ the associated graded 
$$\left(\bigoplus_j\mathrm{gr}_{F}^{ij}E^i_S,\mathrm{KS}(\nabla^i_S)\right)$$ is a \emph{stable} logarithmic Higgs bundle which is also locally free. 
\end{itemize}
Furthermore, if $s\colon \Spec \bar{k} \to S$ is a geometric point of $S$, then
\begin{itemize}
\item[(e)] $(E_{s},\nabla_{s})$ is a log-dR local system on $\bar{X}_{s} = \bar{X}_S \times_S \Spec \bar{k}$, 
then there exists $i\in I$ such that $(E_{s},\nabla_{s})=(E^i_{S},\nabla^i_S)|_{X_{s}}$, 
\item[(f)] $p=\mathrm{char}(\bar{k}) > 2r + 2$, and
\item[(g)] $S \to \Spec \Zb$ is smooth.
\end{itemize}
\end{proposition}
\begin{proof}
The proof is analogous to the one of \cite[Proposition 3.3]{EG20} and will therefore only be sketched. Consider the set  $\mathcal{R}$ of all finite type subrings $R \subset \Cb$. Since $\Cb = \bigcup_{R \in \mathcal{R}} R$ and $(\bar{X},D_{\Cb})$ are defined in terms of finitely many homogenous equations, there exists $\widetilde{R} \in \mathcal{R}$ such that $(\bar{X}_{\Cb},D_{\Cb})$ are obtained by base change from a pair of projective schemes $(\bar{X}_{\widetilde{S}},D_{\widetilde{S}}) \subset \mathbb{P}^N_{\widetilde{S}}$, where we write $\widetilde{S}$ for $\Spec \widetilde{R}$. We may assume that $D_{\widetilde{S}}$ is an snc divisor and that $\bar{X}_{\widetilde{S}}$ is smooth.

We now consider the moduli space $\sM_{log{\mbox -}dR}(\bar{X}_{\widetilde{S}}/{\widetilde{S}})$. Since 
$$\sM_{log{\mbox -}dR}(\bar{X}_{\Cb}/\Cb) \simeq \sM_{log{\mbox -}dR}(\bar{X}_{\widetilde{S}}/{\widetilde{S}}) \times_{\widetilde{S}} \Spec \Cb$$
is finite and flat over $\Spec \Cb$, there exists a finite type algebra $ \tilde{R} \subset 
R$, such that the base change (we denote $\Spec R$ by $S$)
$$\sM_{log{\mbox-}dR}(\bar{X}_{S}/S) = \sM_{log{\mbox -}dR}(\bar{X}_{\widetilde{S}}/{\widetilde{S}}) \times_{\widetilde{S}} S$$
is finite and flat over $S$. 

Since there are only finitely many log-dR local systems $(E^i_{\Cb},\nabla^i_{\Cb})_{i\in I}$ over $\Cb$, we may assume that they have stable and locally free models $(E^i_{S},\nabla^i_{S})_{i\in I}$ over $S$. This amounts to property (a) above. By further enlarging $R$ we obtain strong cohomological rigidity (property (b)), and properties (c,d) about the $F$-filtrations and the associated graded log-Higgs bundles. 

The $S$-models above give rise to sections
\begin{equation}\label{eqn:sections}
[(E^i_{S},\nabla^i_{S})]_{i\in I}\colon S\colon \sM_{log{\mbox -}dR}(\bar{X}_{\widetilde{S}}/{\widetilde{S}}).
\end{equation}
Since the structural morphism $\sM_{log{\mbox -}dR}(\bar{X}_{\widetilde{S}}/{\widetilde{S}}) \to S$ is finite and flat, we infer that the sections of \eqref{eqn:sections} are jointly surjective. This implies (e). By inverting $(2r+2)!$ we can achieve (f). And, property (g) can be arranged by passing to the maximal open subset of $S$ which is smooth over $\Spec \Zb$.
\end{proof}

\subsection{Applications of the Higgs-de Rham flow}

In this subsection, we apply the \emph{logarithmic Higgs-de Rham flow} from \cite{LSYZ19} (the smooth and proper case is due to \cite{LSZ13}).

We fix an arithmetic model as in Proposition \ref{prop:model}. Let $\bar{k}$ be an algebraic closure of a finite field and let $s\colon \Spec \bar{k} \to S$ be a geometric point of $S$.

\begin{definition}[\cite{LSZ13,LSYZ19}]
An $f$-periodic Higgs-de Rham flow on $X_{s}$  is a tuple 
$$(E_0,\nabla_0,F_0,\phi_0,E_1,\nabla_1,F_1,\dots,E_{f-1},\nabla_{f-1},F_{f-1},\phi_{f-1}),$$ 
where for all {$i \in \mathbb{Z}/f\mathbb{Z}$} we have a log-dR local system $(E_i,\nabla_i,F_i)$ with nilpotent $p$-curvature of level $\leq p-1$, a locally split Griffiths-transverse filtration $F_i$, and an isomorphism $\phi_i\colon C_1^{-1}(\mathrm{gr}_F E_i,\mathrm{KS}_i) \simeq (E_{i+1},\nabla_{i+1})$.
\end{definition}
We denote the \emph{set} of isomorphism classes of stable rank $r$ logarithmic Higgs bundles on $X_{s}$ with Hilbert polynomial $P_0$ by   $M_{Dol}(s)$. Likewise, we write $M_{dR}(s)$   for the set of isomorphism classes of stable rank $r$ log-dR local systems on $X_{s}$ with Hilbert polynomial $P_0$. For the purpose of this subsection, it will not matter that those sets are $\bar{k}$-rational points of moduli spaces, which could be constructed with Langer's methods (see Theorem \ref{thm:langer}).

We informally refer to the following diagram as the \emph{Higgs-de Rham flow}:
\[
\xymatrix{
M_{Dol}(s) \ar@/^/@{-->}[rr]^{C^{-1}} & & M_{dR}(s). \ar@/^/@{-->}[ll]^{\mathrm{gr}}
}
\]
The dashed arrows represent merely correspondences, rather than actual maps. The reason is that $\mathrm{gr}(E,\nabla)$ could be not stable, and $C^{-1}$ can only be defined if the $p$-curvature is nilpotent of level $\leq p-1$ and the residues at infinity are nilpotent. 

Using this viewpoint, one calls an element $[(E,\nabla)]$ of $M_{dR}(s)$
 periodic, if there exists $f \in \mathbb{N}$ with
$$[(E,\nabla)]=(C^{-1}\circ\mathrm{gr})^f([(E,\nabla)]).$$

We let $R_{Dol}(s) \subset M_{Dol}(s)$ denote the subset of stable rank $r$ log Higgs bundles with nilpotent Higgs field $\theta$ and nilpotent $\res_\mu \theta$ for all $\mu=1,\dots,c$ of level $\le p-1$. We denote by $R_{dR}(s) \subset M_{dR}(s)$ the subset of stable log-dR local systems with nilpotent residues or level $\le p-1$. Restricting the Higgs-de Rham flow to these subsets has the added advantage of turning the correspondences above into maps of sets:

\begin{equation}\label{eqn:flow}
\xymatrix{
R_{Dol}(s) \ar@/^/[rr]^{C^{-1}} & & R_{dR}(s). \ar@/^/[ll]^{\mathrm{gr}}
}
\end{equation}
It is not immediately obvious that the above maps are well-defined, since one has to justify that strong cohomological rigidity and stability is preserved by $\mathrm{gr}$ and $C^{-1}$.
\begin{lemma}\label{lemma:well-defined}
The maps in \eqref{eqn:flow} are well-defined.
\end{lemma}
\begin{proof}
Proposition \ref{prop:model}(c) allows us to fix for every $(E_{s},\nabla_s) \in R_{dR}(s)$ an $F$-filtration. It follows from Proposition \ref{prop:model}(d) that $\mathrm{gr}(E_s,\nabla_s)=(\mathrm{gr}_FE_s,\mathrm{KS})$ is stable. This shows that $\mathrm{gr}\colon R_{dR}(s) \to R_{Dol}(s)$ is a well-defined map, which a priori depends on the chosen filtration (but see the end of the proof of Lemma \ref{lemma:bijective}).
Arguing as in \cite[Corollary 5.10]{Lan14} one shows that $C^{-1}$ preserves stability.
\end{proof}
\begin{lemma}
Every element of $R_{Dol}(s)$ is strongly cohomologically rigid.
\end{lemma}
\begin{proof}
There is an equivalence of categories (see \cite[Theorem 6.1]{LSYZ19})
$$C^{-1}\colon \mathsf{Higgs}_{p-1}(\bar{X}_{s},D_{s}) \cong \mathsf{MIC}_{p-1}(\bar{X}_{s},D_{s}),$$
where the left-hand side denotes a subcategory of logarithmic Higgs bundles $(V,\theta)$ satisfying several technical assumptions, and similarly, the right-hand side denotes a subcategory of log-dR local systems with nilpotent $p$-Higgs bundles which are required to satisfy various assumptions. We refer the reader to \cite[Section 6]{LSYZ19} for more details. This is an equivalence of categories, and therefore
$${\rm Ext}(C^{-1}(V_s,\theta_s),C^{-1}(V_s,\theta_s))={\rm Ext}((V_s,\theta_s),(V_s,\theta_s))=0.$$
Here, we implicitly use Proposition \ref{prop:model}(e) to guarantee that all self-extensions of $(V_s,\theta_s)$ (respectively $C^{-1}(V_s,\theta_s)$) belong to $\mathsf{Higgs}_{p-1}(\bar{X}_{s},D_{s})$ (respectively $\mathsf{MIC}_{p-1}(\bar{X}_{s},D_{s})$). Indeed, since $p > 2r+2$ by Proposition \ref{prop:model}(f), the Higgs field of such a self-extension is automatically nilpotent of level $\le p-1$.
Thus, $C^{-1}$ preserves strong cohomological rigidity. The same assertion holds for its inverse functor $C$.

We conclude the proof of the lemma by applying assertion (e) of Proposition \ref{prop:model}, according to which every log-dR local system with Hilbert polynomial $P_0$ on $\bar{X}_s$ is strongly cohomologically rigid. Therefore, for every $[(V,\theta)] \in R_{Dol}(s)$ we have that $C^{-1}(V,\theta)$ is strongly cohomologically rigid. This implies that $(V,\theta)={\rm gr} \circ C^{-1}(V,\theta)$ is strongly cohomologically rigid.
\end{proof}

\begin{lemma}\label{lemma:finite}
The set $R_{\dR}({s})$ is finite.
\end{lemma}
\begin{proof}
Let $(E,\nabla)$ be a strongly cohomologically rigid log-dR local system which is stable and has Hilbert polynomial $P_0$. The hypercohomology group \eqref{eqn:logdRcomplex} computes the tangent space of $\sM_{dR}(\bar{X},D)$ in $[(E,\nabla)]$. By the vanishing assumption, the point $[(E,\nabla)]$ is isolated. We conclude the proof by recalling that the number of isolated points of a Noetherian scheme is finite.
\end{proof}
\begin{lemma}\label{lemma:bijective}
The maps $\mathrm{gr}$ and $C^{-1}$ are bijections.
\end{lemma}
\begin{proof}
It suffices to prove that $\mathrm{gr}$ and $C^{-1}$ are injective. Indeed, it then follows from Lemma \ref{lemma:finite}, they must be of equal cardinality if both maps are injective. The pigeonhole principle is used to conclude that $\mathrm{gr}$ and $C^{-1}$ are bijections.

Since $C^{-1}$ is defined using an equivalence of categories (see \cite[Theorem 6.1]{LSYZ19}), it is clear that $C^{-1}\colon R_{Dol}(s) \to R_{dR}(s)$ is injective.

The associated graded $\mathrm{gr}$ is injective for different reasons. In particular, we will use strong cohomological rigidity to prove this. The Artin-Rees construction applied to the $F$-filtration on $(E,\nabla)$ yields a $\G_m$-equivariant $\mathbb{A}^1_{s}$-family of vector bundles $(\mathcal{V},\nabla_t)$, endowed with a log-$t$-connection $\nabla_t$, where $t\colon \mathbb{A}^1_{s} \to \mathbb{A}^1_{s}$ denotes the identity map. Furthermore, we have
$$(\mathcal{V},\nabla_t)|_{t=0} \simeq (\mathrm{gr}_F E,\mathrm{KS}).$$
Recall from Proposition \ref{prop:model}(d) that the right-hand side is a strongly cohomologically rigid log-Higgs bundle. There is therefore a unique way to lift it to a $t$-connection over $\Spec \bar{k}[t]/(t^2)$, and likewise for $\Spec \bar{k}[t]/(t^n)$. We infer from the Grothendieck existence theorem that there is a unique way to lift it to a $t$-connection on $\Spec \bar{k}[[t]]$. This implies that there cannot be a pair of distinct elements 
$$(E^1_{s},\nabla^1_{s}),(E^2_{s},\nabla^2_{s}) \in R_{dR}(s) \text{ such that }(\mathrm{gr}_FE^1_{s},\mathrm{KS})\simeq(\mathrm{gr}_FE^2_{s},\mathrm{KS}).$$
Otherwise, we would have
$$(E^1_{s},\nabla^1_{s}) \otimes \bar{k}((t)) \simeq (E^2_{s},\nabla^2_{s}) \otimes \bar{k}((t)),$$
which implies the existence of an isomorphism over $\bar{k}$ (by stability). This concludes the proof of injectivity,  and furthermore proves that the map $\mathrm{gr}$ doesn't depend on the chosen $F$-filtration. 
\end{proof}

\begin{proposition}
The $p$-curvature of $[(E,\nabla)] \in R_{dR}(s)$ is nilpotent.
\end{proposition}
\begin{proof}
By virtue of definition of $C^{-1}$, every log-dR local system in the image of $C^{-1}$ has nilpotent $p$-curvature. According to Lemma \ref{lemma:bijective} the map $C^{-1}$ is bijective. This concludes the proof.
\end{proof}

\begin{proposition}\label{prop:periodic}
Every $[(E,\nabla)] \in R_{dR}(s)$ is \emph{periodic}.
\end{proposition}
\begin{proof}
Let $\sigma = C^{-1}\circ \mathrm{gr}$. By definition, it is a permutation of the finite set $R_{dR}(s)$. Let $f'$ be the order of $\sigma$. We then have that $\sigma^{f'}([(E,\nabla)])=[(E,\nabla)]$, and thus $[(E,\nabla)]$ is $f$-periodic for  some $f|f'$.
\end{proof}

\subsection{Higgs-de Rham flow over truncated Witt rings}
As before, we denote by $\bar{k}$ the algebraic closure of a finite field of characteristic $p$, and let $s\colon \Spec \bar{k} \to S$ be a $\bar{k}$-point of $S$. Furthermore, we write $W=W(\bar{k})$ for the associated Witt ring, and $K$ for its fraction field. Hensel's lemma and Proposition \ref{prop:model}(g) implies that $s$ can be extended to a morphism 
$$s_W\colon\Spec W \to S.$$
For $n \in \mathbb{N}$ we denote by $W_n$ the ring of $n$-th Witt vectors and by $\bar{X}_n$ the base change $\bar{X}_S \times_S W_n.$

We define $\mathcal{H}(\bar{X}_n/W_n)$ to be the category of tuples $(V,\theta, \bar{E},\bar{\nabla},\bar{F},\phi)$, where $(V,\theta)$ is a graded log-Higgs bundle on $\bar{X}_n$ of level $\leq p-1$, $(\bar{E},\bar{\nabla},\bar{F})$ is a log-dR local system on $\bar{X}_{n-1}$ with a locally split Griffiths-transverse filtration $\bar{F}$ of level $\leq p-2$, and
 $\phi\colon gr_{\bar F}(\bar{E}, \bar \nabla) \simeq (V,\theta) \times_{W_n} W_{n-1}$ is an isomorphism of graded log-Higgs bundles.

Similarly, we denote by $\mathsf{MIC}(\bar{X}_n/W_n)$ the category of quasi-coherent sheaves with $W_n$-linear flat connections on $\bar{X}_n$. There is a functor 
$$C_n^{-1}\colon {\mathcal{H}}(\bar{X}_n/W_n)\to \mathsf{MIC}(\bar{X}_n/W_n)$$ 
which extends the logarithmic inverse Cartier transform. In the  proper non-logarithmic case this is due to \cite[Theorem 4.1]{LSZ13}. {Closely related results were obtained by Xu in \cite{Xu19}.} The logarithmic version is covered in \cite[Section 5]{LSYZ19} immediately before the proof of Proposition 5.2.

Let $(E_{W_n},\nabla_{W_n},F_n)$ be an $W_n$-linear log-dR local system endowed with an $F$-filtration. We denote by $\overline{\mathrm{gr}}(E,\nabla,F)$ the tuple $(\mathrm{gr}_F(E),\mathrm{KS},(E,\nabla,F)_{W_{n-1}},\mathrm{id})$.

\begin{definition}[Lan--Sheng--Zuo \& Lan--Shen--Yang--Zuo]\label{defi:f-periodic}
An $f$-periodic log-dR local system on $\bar{X}_n/W_n$ is a tuple 
$$(E_{W_n}^0,\nabla_{W_n}^0,F_{W_n}^0,\phi_0,E_{W_n}^1,\nabla_{W_n}^1,F_{W_n}^1,\dots,E_{W_n}^{f-1},\nabla_{W_n}^{f-1},F_{W_n}^{f-1},\phi_{f-1}),$$ 
where for all $i$ we have that $(E^i,\nabla^i,F^i)$ is a log-dR local system on $\bar{X}_{W_n}$ ({nilpotent of level $\leq p-2$} on the special fibre) with a locally split Griffiths-transverse filtration $F^i_{W_n}$, {such that for all integers $n$ we have that $\overline{\mathrm{gr}}_F(E^i_{W_n},\nabla^i_{W_n})$ belongs to $\mathcal{H}({\bar{X}_n/W_n})$ and $\phi_i\colon C_1^{-1}(\mathrm{gr}_F E^i_{W_n},\mathrm{KS}_i) \simeq (E^{i+1}_{W_n},\nabla^{i+1}_{W_n})$}.
\end{definition}

By taking the inverse limit with respect to $n$, we obtain a notion of periodicity relative to $W$. Using  
\cite[p.3, Theorem~3.2, Variant~2]{LSZ13}, \cite[Theorem~1.1]{LSYZ19} together with \cite[Theorem 2.6*,p.43 i)]{Fal88} one obtains:

\begin{theorem}[Lan--Sheng--Zuo \& Lan--Sheng--Yang--Zuo]\label{thm:lsz}
A $1$-periodic log-dR local system on $\bar{X}_W/W$ gives rise to a torsion-free Fontaine--Laffaille module  on $X_W=\bar{X}_W \setminus D_W$.
Furthermore, we can associate to an $f$-periodic log-dR local system on $\bar{X}_W/W$ a crystalline \'etale local system of free $W(\mathbb{F}_{p^f})$-modules on $X_K$. This is a fully faithful functor.
\end{theorem}

 We remark that Faltings only treats the case $f=1$, in which he constructs a fully faithful functor from Fontaine-Laffaille modules to \'etale local systems of $\Zb_p$-modules. The general case can be reduced to this one using a categorical construction, as explained in \cite[Variant~2]{LSZ13}. It is clear that this formal procedure preserves fully faithfulness of the functor.
In combination with the above, the following result concludes the proof of Theorem \ref{thm:main}.

\begin{theorem}
Every element $[(E_s,\nabla_s)] \in R_{dR}(s)$ can be lifted to a periodic Higgs-de Rham flow over $W_n$ on $\bar X_n$.
\end{theorem}
\begin{proof}
Recall from Lemma \ref{lemma:finite} that there is a finite number of non-isomorphic log-dR local systems
$(E^i_s,\nabla^i_s)_{i\in I}$ 
 in $R_{dR}(s)$.
For every $i \in I$ there exists an extension to an $S$-family of log-dR local systems $(E^i_S,\nabla^i_S)$ (see Proposition \ref{prop:model}(a,b)). By pulling back along $s_W\colon \Spec W \to S$ we therefore obtain a lift to a $W$-family $(E^i_W,\nabla^i_W)$, and hence also a $W_n$-lift $(E^i_{W_n},\nabla^i_{W_n})$.

Since $(E^i_s,\nabla_s^i)$ is strongly cohomologically rigid, deformation theory implies that such a $W$-lift is unique up to isomorphism.

We have seen in Proposition \ref{prop:periodic} that every $(E^i_s,\nabla_s^i)$ is periodic over $s$ (this corresponds to the case $n=0$) since the map
$$\sigma=C^{-1}\circ \mathrm{gr}\colon R_{dR}(s) \to R_{dR}(s)$$
is a permutation (Lemma \ref{lemma:bijective}). The $W_n$-relative log-dR local system $(E^i_{W_n},\nabla^i_{W_n})$ is endowed with an $F$-filtration by Proposition \ref{prop:model}(c). We can therefore evaluate $\overline{\mathrm{gr}}(E^i_{W_n},\nabla^i_{W_n})$.

Since the functor $C_n^{-1}$ extends $C^{-1}$ on the special fibre, we see that we have
$$(C^{-1}_n\circ\overline{\mathrm{gr}})(E^i_{W_n},\nabla^i_{W_n})\simeq (E^{\sigma(i)}_{W_n},\nabla^i_{W_n}).$$
In particular, for $(E^i_s,\nabla^i_s)$ being $f$-periodic, we have $\sigma^f(i)=i$, and thus 
$$(C_n^{-1}\circ \overline{\mathrm{gr}})(E^i_{W_n},\nabla^i_{W_n})\simeq (E^i_{W_n},\nabla^i_{W_n}).$$
This equation establishes periodicity of $(E^i_{W_n},\nabla^i_{W_n})$ relative to $W_n$.
\end{proof}

We will now state a Betti version of Theorem~\ref{thm:main}. For this purpose, let us recall that the Betti moduli space $\mathcal{M}_B(X)$ of  {irreducible rank $r$ complex local systems} of is zero-dimensional, since every such {local system}  is assumed to be strongly cohomologically rigid. Furthermore, $\mathcal{M}_B(X)$ is defined over $\mathbb{Q}$, and therefore the irreducible rank $r$ {local systems} $\rho_1,\dots,\rho_{n_r}$ are defined over a number field $F$. By finite generation of $\pi_1^{\rm top}(X)$ we have that the representations $\rho_1,\dots,\rho_{n_r}$ can be defined over $\Oo_F[M^{-1}]$, for a sufficiently big positive integer $M$.

\begin{rmk} As by Remark~\ref{rmk:chern} (a), 
$H^1(X({\mathbb C}), {\End}(E_{\mathbb{C}, {\rm an}})^{\End(\nabla_{\mathbb C})})$ computes the left-hand side of 
\eqref{eqn:logdRcomplex}, so is equal to zero, we can apply  \cite[Theorem~1.1]{EG18}  to conclude  that Simpson's integrality conjecture holds. Therefore, $M$ can be chosen to be $1$.  In fact, under the assumption that 
all log-dR bundles in a given rank  are rigid, the proof of {\it loc. cit.} applies without verifying this vanishing assumption. We do not use this remark in the sequel. 
\end{rmk}

For every prime $p \nmid M$, and every choice of an embedding $\Oo_F \to W(\bar{\mathbb F}_p)$ we can therefore consider the induced $W(\bar{\Fb}_p)$-representations
$$\rho^{W(\bar{\Fb}_p)}_1,\dots,\rho^{W(\bar{\Fb}_p)}_{n_r}\colon \pi_1^{\rm top}(X_{\Cb}) \to \GL_r(W(\bar{\Fb}_p)).$$

The \'etale fundamental group $\pi_1(X_{\Cb})$ is the profinite completion of $\pi_1^{\rm top}(X_{\mathbb C})$.
Thus, we obtain continuous representations
$$\rho^{W(\bar{\Fb}_p)}_1,\dots,\rho^{W(\bar{\Fb}_p)}_{n_r}\colon \pi_1(X_{\Cb}) \to \GL_r(W(\bar{\Fb}_p)).$$

 \begin{theorem} \label{thm:betti}
 Let $(\bar{X}_{\Cb},X_{\Cb},D_{\Cb})$ and $(\bar{X}_{S},X_{S},D_{S})$ be as in Theorem \ref{thm:main}. Suppose that $p$ is a prime, which belongs to the image of $S \to \Spec \Zb$. Let $k$ be a finite field of characteristic $p$ and fix a morphism $\Spec W(k) \to S$. Then, the representations $\{\rho^{W(\bar{\Fb}_p)}_i\}_{i=1,\dots,n_r}$ descend to crystalline representations $\{\rho_{i}^{\rm cris}\}_{i=1,\dots,n_r}: \pi_1(X_K) \to GL_r(W(\bar{\Fb}_p))$, where $K={\rm Frac}(W(k))$.
 \end{theorem}

 \begin{proof}
 By combining Theorem~\ref{thm:lsz} and Theorem \ref{thm:main}, we obtain crystalline representations
 $$\{\pi_{i}\}_{i=1,\dots, n_r}: \pi_1(X_{K})\to GL_r(W (\bar{\mathbb F}_p))$$ associated to the corresponding log-dR systems $(E_i,\nabla_i)_{i=1,\dots,r}$ on $\bar X_K$ 
 Restricting these representations further to the geometric fundamental group $\pi_1(X_{\bar{K}})$, we obtain 
 $$(\pi^{\rm geom}_i)_{i=1,\dots, n_r}\colon \pi_1(X_{\bar{K}}) \to GL_r(W(\bar{\mathbb{F}}_p)).$$
 \begin{claim}
 The geometric representations $(\pi_i^{\rm geom})_{i=1,\dots,n_r}$ are irreducible.
 \end{claim}
 \begin{proof}[Proof of the claim]
Assume by contradiction that there exists $\pi_i^{\rm geom}$ which is reducible. Then, the residual representation $ \pi_i^{\rm geom}\otimes \bar{ \mathbb F}_p:  \pi_1(X_{\bar K})\to GL_r(\bar{ \mathbb F}_p)$ is reducible as well. 
The continuous representation  $\pi_i\otimes \bar{ \mathbb F}_p: \pi_1(X_K)\to GL_r(\bar{\mathbb F}_p)$
factors through the finite group $GL_r(\mathbb{F}_q)$ for a $p$-power $q$. By Proposition~\ref{prop:model} (g)
we may assume that $X_{K}$ has a rational point, which yields a section $\Gal(\bar{K}/K) \to \pi_1(X_K).$ The kernel of the restriction $\pi_i |_{{\rm Gal}(\bar{K}/K)}$ yields a finite extension $K'/K$ such that $\pi_i \otimes \bar{\mathbb F}_p|_{\pi_1(X_{K'})}$ is reducible. 

Let $\alpha$ be a subrepresentation of  $\pi_i \otimes \bar{\mathbb F}_p|_{\pi_1(X_{K'})}$. We will now use work by Sun--Yang--Zuo. It develops a version of the Higgs-de Rham flow over ramified extensions of $W$. Theorem 5.15 in \cite{SYZ} implies that the subrepresentation $\alpha$ gives rise to a sub-log-dR local system of $(E_i,\nabla)$, which is furthermore periodic and thus of slope $0$. This contradicts stability. Note that \emph{loc. cit.} deals with the more general setting of twisted Higgs-de Rham flows and projective representations. When applying their result we may therefore assume that the twisting line bundle $\mathcal{L}$ is trivial, since our representations are not projective.
 \end{proof}
  \begin{claim}
 The geometric representations $\pi_i^{\rm geom}$ are pairwise non-isomorphic.
 \end{claim}
 \begin{proof}[Proof of the claim]
As before, it suffices to show that the residual $\bar{\mathbb{F}}_p$-representations are pairwise non-isomorphic. We will use the same strategy as before. An isomorphism between 
 $\pi_i^{\rm geom} \otimes \bar{\mathbb F}_p$  and $\pi_j^{\rm geom} \otimes \bar{\mathbb F}_p$
 therefore implies the existence of a finite extension $K'/K$ such that there exists an isomorphism 
$$\pi_i^{\rm geom} \otimes \bar{\mathbb F}_p \simeq \pi_j^{\rm geom}  \otimes \bar{\mathbb F}_p \colon \pi_1(X_{K'}) \to GL_r(\bar{\mathbb{F}}_p).$$
According to \cite[Theorem 5.15]{SYZ} (which relies on \cite[Theorem 5.8(ii)]{SYZ}), the functor from periodic Higgs-de Rham flows to $\bar{\mathbb{F}}_p$-linear representations of $\pi_1(X_{K'})$ is fully faithful. Thus, we obtain an isomorphism of the associated Higgs-de Rham flows, and in particular we have that $(E_i,\nabla_i)_{K'} \simeq (E_j,\nabla_j)_{K'}$. This implies $i=j$ since the associated complex log-dR local systems $(E_i,\nabla_i)_{\Cb}$ and $(E_j,\nabla_j)_{\Cb}$ are non-isomorphic, and hence concludes the proof.
 \end{proof}
Applying these two claims we see that 
 the geometric representations
 $$ \pi_i|_{\pi_1(X_{\bar K})} : \pi_1(X_{\bar K})\to GL_r(W (\bar{\mathbb F}_p))$$
 for $i=1,\ldots, n_r$ 
  remain irreducible since the set
  $\{ \pi_i|_{\pi_1(X_{\bar K})}\}_{i=1,\ldots, n_r}$ defines $n_r$ pairwise non-isomorphic $W(\bar{\mathbb F})_p$-local systems 
on $X_{\bar K}$, which by the pigeonhole principle has to be the set of  $W (\bar{\mathbb F}_p)$-local systems defined by $\{\rho^{W(\bar{\Fb}_p)}_i\}_{i=1,\dots,n_r}$, and thus each single one of them descends to a crystalline representation. 
 \end{proof}

\end{document}